\documentclass[11pt]{amsart}

\usepackage{
    amsmath,
    amsfonts,
    amssymb,
    amsthm,
    amscd,
    comment,
    enumitem,
    etoolbox,
    gensymb,    
    mathtools,
    mathdots
}
\usepackage[usenames,dvipsnames]{xcolor}
\usepackage[all]{xy}

\makeatletter
\@namedef{subjclassname@2020}{%
  \textup{2020} Mathematics Subject Classification}
\makeatother

\usepackage[T1]{fontenc}
\usepackage{bbm}                     
\usepackage[colorlinks=true, linkcolor=blue, citecolor=blue, urlcolor=blue, breaklinks=true]{hyperref}

\DeclareFontFamily{OT1}{pzc}{}
\DeclareFontShape{OT1}{pzc}{m}{it}{<-> s * [1.10] pzcmi7t}{}
\DeclareMathAlphabet{\mathpzc}{OT1}{pzc}{m}{it}

\usepackage[capitalize]{cleveref}   

\crefname{defin}{Definition}{Definitions}
\crefname{eg}{Example}{Examples}
\crefname{lem}{Lemma}{Lemmas}
\crefname{theo}{Theorem}{Theorems}
\crefname{equation}{}{}
\crefname{enumi}{}{}

\newcommand\N{\mathbb{N}}

\newcommand\Z{\mathbb{Z}}
\newcommand\kk{\Bbbk}
\newcommand\one{\mathbbm{1}}

\newcommand\ba{\mathbf{a}}
\newcommand\bb{\mathbf{b}}
\newcommand\bt{\mathbf{t}}

\newcommand\bg{\mathbf{g}}
\newcommand\bh{\mathbf{h}}

\newcommand\cA{\mathcal{A}}

\newcommand\cC{\mathcal{C}}

\newcommand\cN{\mathcal{N}}

\newcommand\cS{\mathcal{S}}
\newcommand\cT{\mathcal{T}}

\newcommand\fS{\mathfrak{S}}            

\newcommand\op{\mathrm{op}}
\newcommand\rev{\mathrm{rev}}

\newcommand{\md}{\textup{-mod}}
\newcommand{\bmd}{\textup{-bimod}}

\newcommand\FBA{\mathpzc{FinBoolAlg}}    
\newcommand\FBAlf{\mathpzc{FinBoolAlg}(G)_\textup{lf}}  
\newcommand\FinSet{\mathpzc{FinSet}}    
\newcommand\FinSetf{\mathpzc{FinSet}(G)_\textup{free}} 
\newcommand\Heis{\mathpzc{Heis}}        
\newcommand\Par{\mathpzc{Par}}          
\newcommand\Rep{{\mathrm{\underline{Re}p}}}
\newcommand\go{\mathsf{V}}                       
\newcommand\PS{X}                       
\newcommand\BV[1][k]{\mathbf{B}_{#1}}            

\DeclareMathOperator{\End}{End}

\DeclareMathOperator{\Hom}{Hom}

\DeclareMathOperator{\Mon}{Mon}

\DeclareMathOperator{\Partition}{Par}
\DeclareMathOperator{\Power}{Pow}
\DeclareMathOperator{\Kar}{Kar}

\DeclareMathOperator{\stack}{stack}

\usepackage{tikz}
\usetikzlibrary{arrows,patterns}
\usepackage{tikz-cd}

\newcommand{\dotlabel}[1]{$\scriptstyle{#1}$}
\newcommand{\pd}[2][black]{\filldraw[#1] (#2) circle (1.5pt)} 
\newcommand{\pdg}[3]{
    \filldraw[black] (#1) circle (1.5pt) node[anchor=#2] {\dotlabel{#3}}
}
\newcommand{\braidto}{to[out=up,in=down]}

\newcommand\opendot[1]{\filldraw[fill=white,draw=black] (#1) circle (2pt)}
\newcommand\token[3]{
  \filldraw[blue] (#2) circle (1.5pt) node[anchor=#1] {\dotlabel{#3}}
}
\newcommand\teleport[2]{
  \draw[blue] (#1) to (#2);
  \filldraw[blue] (#1) circle (1.5pt);
  \filldraw[blue] (#2) circle (1.5pt)
}

\newcommand\Peis[1][]{{\Heis_{\uparrow \downarrow}^{#1}}}

\newcommand\merge{
    \begin{tikzpicture}[anchorbase]
      \draw (-0.25,-0.25) to (0,0);
      \draw (0.25,-0.25) to (0,0);
      \draw (0,0) to (0,0.25);
    \end{tikzpicture}
}

\newcommand\spliter{
    \begin{tikzpicture}[anchorbase]
      \draw (-0.25,0.25) to (0,0);
      \draw (0.25,0.25) to (0,0);
      \draw (0,0) to (0,-0.25);
    \end{tikzpicture}
}

\newcommand\crossing{
    \begin{tikzpicture}[anchorbase]
      \draw (-0.25,-0.25) to (0.25,0.25);
      \draw (0.25,-0.25) to (-0.25,0.25);
    \end{tikzpicture}
}

\newcommand\bottompin{
    \begin{tikzpicture}[anchorbase]
      \draw (0,0) to (0,0.25);
      \opendot{0,0};
    \end{tikzpicture}
}

\newcommand\toppin{
    \begin{tikzpicture}[anchorbase]
      \draw (0,0) to (0,-0.25);
      \opendot{0,0};
    \end{tikzpicture}
}

\newcommand\idstrand{
    \begin{tikzpicture}[anchorbase]
      \draw (0,-0.25) to (0,0.25);
    \end{tikzpicture}
}

\newcommand\tokstrand[1][g]{
    \begin{tikzpicture}[anchorbase]
      \draw (0,-0.25) to (0,0.25);
      \token{east}{0,0}{#1};
    \end{tikzpicture}
}

\newcommand\lolly{
    \begin{tikzpicture}[anchorbase]
        \draw (0,-0.2) -- (0,0.2);
        \opendot{0,-0.2};
        \opendot{0,0.2};
    \end{tikzpicture}
}

\tikzset{anchorbase/.style={>=To,baseline={([yshift=-0.5ex]current bounding box.center)}}}
\tikzset{
    centerzero/.style={>=To,baseline={([yshift=-0.5ex](#1))}},
    centerzero/.default={0,0}
}

\newtheorem{theo}{Theorem}[section]

\newtheorem{prop}[theo]{Proposition}
\newtheorem{lem}[theo]{Lemma}
\newtheorem{cor}[theo]{Corollary}

\theoremstyle{definition}
\newtheorem{conv}[theo]{Convention}
\newtheorem{defin}[theo]{Definition}
\newtheorem{rem}[theo]{Remark}
\newtheorem{eg}[theo]{Example}

\numberwithin{equation}{section}
\allowdisplaybreaks

\setenumerate[1]{label=(\alph*)}          

\newtoggle{comments}
\newtoggle{details}
\newtoggle{detailsnote}

\iftoggle{comments}{%
  \newcommand{\acomments}[1]{
    \ \\
    {\color{red}
      \textbf{AS:} #1
    }
    \ \\
    }
  \newcommand{\scomments}[1]{
    \ \\
    {\color{red}
      \textbf{SNL:} #1
    }
    \ \\
    }
}{%
  \newcommand{\acomments}[1]{}
  \newcommand{\scomments}[1]{}
}

\iftoggle{details}{%
  \newcommand{\details}[1]{
      \ \\
      {\color{OliveGreen}
        \textbf{Details:} #1
      }
      \\
  }
}{%
  \newcommand{\details}[1]{}
}

\begin{document}

\title[Group partition categories]{Group partition categories}

\author{Samuel Nyobe Likeng}
\address[S.N.L]{
  Department of Mathematics and Statistics \\
  University of Ottawa \\
  Ottawa, ON K1N 6N5, Canada
}
\email{snyob030@uottawa.ca}

\author{Alistair Savage}
\address[A.S.]{
  Department of Mathematics and Statistics \\
  University of Ottawa \\
  Ottawa, ON K1N 6N5, Canada
}
\urladdr{\href{https://alistairsavage.ca}{alistairsavage.ca}, \textrm{\textit{ORCiD}:} \href{https://orcid.org/0000-0002-2859-0239}{orcid.org/0000-0002-2859-0239}}
\email{alistair.savage@uottawa.ca}

\begin{abstract}
    To every group $G$ we associate a linear monoidal category $\Par(G)$ that we call a \emph{group partition category}.  We give explicit bases for the morphism spaces and also an efficient presentation of the category in terms of generators and relations.  We then define an embedding of $\Par(G)$ into the group Heisenberg category associated to $G$.  This embedding intertwines the natural actions of both categories on modules for wreath products of $G$.  Finally, we prove that the additive Karoubi envelope of $\Par(G)$ is equivalent to a wreath product interpolating category introduced by Knop, thereby giving a simple concrete description of that category.
\end{abstract}



\ifboolexpr{togl{comments} or togl{details}}{%
  {\color{magenta}DETAILS OR COMMENTS ON}
}{%
}

\maketitle
\thispagestyle{empty}

\tableofcontents

\section{Introduction}

The \emph{partition category} is a $\kk$-linear monoidal category, depending on a parameter $d$ in the commutative ground ring $\kk$, that encodes the homomorphism spaces between tensor powers of the permutation representation of all the finite symmetric groups in a uniform way.  Its additive Karoubi envelope is the category $\Rep(S_d)$, introduced by Deligne in \cite{Del07}.  Deligne's category $\Rep(S_d)$ interpolates between categories of representations of symmetric groups in the sense that the category of representations of $S_n$ is equivalent to the quotient of $\Rep(S_n)$ by the tensor ideal of negligible morphisms.

When working with linear monoidal categories in practice, it is useful to have two descriptions.  First, one would like to have an explicit basis for each morphism space, together with an explicit rule for the tensor product and composition of elements of these bases.  Second, one wants an efficient presentation of the category in terms of generators and relations.  Such a presentation is particularly useful when working with categorical actions, since one can define the action of generators and check the relations.  Both descriptions exist for the partition category.  Bases for morphisms spaces are given in terms of \emph{partition diagrams} with simple rules for composition and tensor product.  Additionally, there is an efficient presentation, which can be summarized as the statement that the partition category is the free $\kk$-linear symmetric monoidal category generated by a $d$-dimensional special commutative Frobenius object.

Deligne's original paper \cite{Del07} has inspired a great deal of further research.  Of particular importance for the current paper are the generalizations of Knop and Mori.  In \cite{Kno07}, Knop generalized Deligne's construction by embedding a regular category $\cA$ into a family of pseudo-abelian tensor categories $\cT(\cA,\delta)$, which are the additive Karoubi envelope of categories $\cT^0(\cA,\delta)$ depending on a degree function $\delta$.   Deligne's original construction corresponds to the case where $\cA$ is the category of finite boolean algebras (equivalently, the opposite of the category of finite sets).  Knop's construction, which is inspired by the calculus of relations on $\cA$, has the advantage of being very general, but the disadvantage of being rather abstract.  In particular, Knop does not give a presentation of his categories in terms of generators and relations.

In \cite{Mor12}, Mori generalized Deligne's construction in a somewhat different direction.  For each $d \in \kk$, Mori defines a 2-functor $\cS_d$ sending a tensor category $\cC$ to another tensor category $\cS_d(\cC)$, which should be thought of as a sort of interpolating wreath product functor.  Morphisms are described in terms of recollements and one has a presentation using the string diagram calculus for braided monoidal categories.

In the current paper, we are interested in a setting where the constructions of Knop and Mori are closely related.  This occurs when $\cA$ is the category of finite boolean algebras with a locally free action of a finite group  $G$, and when $\cC$ is the category of representations of $G$.  With these choices, the categories defined by Knop and Mori can both be viewed as interpolating categories for the categories of representations of the wreath products $G^n \rtimes \fS_n$.  In fact, Mori's interpolating category contains Knop's as a full subcategory; see \cite[Rem.~4.14]{Mor12} for a precise statement.  These ``wreath Deligne categories'' and other variations have been further studied in \cite{Eti14,FS18,Har16,Ryb18,Ryb19}.

Wreath products of groups and algebras appear in a surprising number of areas of mathematics, including vertex operators, the geometry of the Hilbert scheme, and categorification.  In particular, to every Frobenius algebra (or, more generally, graded Frobenius superalgebra) $A$ and choice of central charge $k \in \Z$, there is a \emph{Frobenius Heisenberg category} $\Heis_k(A)$, introduced in \cite{RS17,Sav19} and further studied in \cite{BSW20}.  When $k=\pm 1$, this category encodes the representation theory of \emph{all} the wreath product algebras $A^{\otimes n} \rtimes \fS_n$, $n \in \N$, \emph{simultaneously}.  (For other choices of $k$ it encodes the representation theory of more general cyclotomic quotients of the affine wreath product algebras introduced in \cite{Sav20}.)  Generating objects of $\Heis_{\pm 1}(A)$ correspond to induction and restriction functors with respect to the natural embedding $A^{\otimes n} \rtimes \fS_n \hookrightarrow A^{\otimes (n+1)} \rtimes \fS_{n+1}$.

In the current paper, we are interested in the case where the Frobenius algebra $A$ is the group algebra of a finite group $G$.  In this case, we call $\Heis(G) := \Heis_{-1}(\kk G)$ the \emph{group Heisenberg category}.  When $G$ is trivial, this category was first introduced by Khovanov in \cite{Kho14}.  In \cite{NS19}, the authors described a natural embedding of the partition category into Khovanov's Heisenberg category.  This, in turn, induces an embedding of Deligne's interpolating category $\Rep(S_d)$ into the additive Karoubi envelope of the Heisenberg category.

The goal of the current paper is to give simple, explicit descriptions of wreath product analogues of partition categories and to relate these to group Heisenberg categories.  First, to any group $G$ we associate a \emph{$G$-partition category} $\Par(G)$ (\cref{GPCbasis}).  The definition is given in terms of explicit \emph{$G$-partition diagrams}, which form bases for the morphism spaces of the category.  We then give an efficient presentation of $\Par(G)$ in terms of generators and relations (\cref{twocats}).  There is a natural categorical action of the $G$-partition category on tensor products of permutation representations of wreath products of $G$.  This action can be described in terms of the generators (\cref{hide}) or the bases of $G$-partition diagrams (\cref{forest}).  The action functor is full, and we give an explicit description of its kernel (\cref{kangaroo}).  This gives a categorical analogue of a double centralizer property akin to Schur--Weyl duality, generalizing work of Bloss \cite{Blo03} who defined $G$-colored partition algebras which are isomorphic to the endomorphism algebras in $\Par(G)$.

Next, we give an explicit embedding of $\Par(G)$ into the group Heisenberg category $\Heis(G)$ (\cref{hitchcock}), generalizing the main result of \cite{NS19}.  This embedding intertwines the natural categorical actions of $\Par(G)$ and $\Heis(G)$ on modules for wreath products (\cref{wactcom}).

Finally, we prove (\cref{Knop}) that the group partition category $\Par(G)$ is equivalent to Knop's category $\cT^0(\cA,\delta)$ when $\cA$ is the category of finite boolean algebras with a locally free $G$-action (equivalently, the opposite of the category of finite sets with a free $G$-action).  Thus, one can view $\Par(G)$ as a concrete, and very explicit, realization of the wreath product interpolating categories of Knop and Mori.  In particular, the calculus of $G$-partition diagrams is significantly simpler than the previous constructions.  (Although the latter are, of course, more general.)  In addition, the presentation of $\Par(G)$, in terms of generators and relations, given in the current paper is considerably more efficient that the presentation given by Mori in \cite[Prop.~4.26]{Mor12}; see \cref{Mori} for further details.

The structure of the current paper is as follows.  In \cref{sec:wreath} we recall some basic facts about wreath products.  We define our main object of interest, the group partition category, in \cref{sec:ParG}.  We then give a presentation of $\Par(G)$ in terms of generators and relations in \cref{sec:present}.  In \cref{sec:action} we define the natural categorical action of $\Par(G)$.  We recall the definition of the group Heisenberg category $\Heis(G)$ in \cref{sec:HeisG} and then, in \cref{sec:embed}, we define the embedding of $\Par(G)$ into $\Heis(G)$.  In \cref{sec:compat} we prove that this embedding intertwines the natural categorical actions of these categories on modules for wreath products.  Finally, in \cref{sec:Knop} we relate $\Par(G)$ to the constructions of Knop and Mori.

\subsection*{Acknowledgements}

This research of A.~Savage was supported by Discovery Grant RGPIN-2017-03854 from the Natural Sciences and Engineering Research Council of Canada.  S.~Nyobe Likeng was also supported by this Discovery Grant.  The authors would like to thank S.~Henry and P.~Scott for helpful conversations concerning boolean algebras.

\section{Wreath products\label{sec:wreath}}

Fix a commutative ground ring $\kk$ and a group $G$ with identity element $1_G$.  We use an unadorned tensor product $\otimes$ to denote the tensor product over $\kk$.  We will often define linear maps on tensor products by specifying the images of simple tensors; such maps are always extended by linearity.

For $n \ge 1$, the symmetric group $\fS_n$ acts on $G^n$ by permutation of the factors, where we \emph{number factors from right to left}:
\[
    \pi \cdot (g_n, \dotsc, g_1)
    = (g_{\pi^{-1}(n)}, \dotsc, g_{\pi^{-1}(1)}).
\]
The \emph{wreath product group} $G_n := G^n \rtimes \fS_n$ has underlying set $G^n \times \fS_n$, and multiplication
\[
    (\bg, \pi) (\bh, \sigma) = (\bg (\pi \cdot \bh), \pi \sigma),\quad
    \bg,\bh \in G^n,\ \pi, \sigma \in \fS_n.
\]
We identify $G^n$ and $\fS_n$ with the subgroups $G^n \times \{1_{\fS_n}\}$ and $\{1_G\} \times \fS_n$, respectively, of $G^n \rtimes \fS_n$.  Hence we write $\bg \pi$ for $(\bg, \pi)$.

Let $A := \kk G$ be the group algebra of $G$.  Then the group algebra $A_n := \kk G_n$ is isomorphic to the \emph{wreath product algebra} $A^{\otimes n} \rtimes \fS_n$, which is isomorphic to $A^{\otimes n} \otimes \kk \fS_n$ as a $\kk$-module, with multiplication given by
\[
    (\ba \otimes \pi) (\bb \otimes \sigma)
    = \ba (\pi \cdot \bb) \otimes \pi \sigma,\quad
    \ba,\bb \in A^{\otimes n},\ \pi,\sigma \in \fS_n.
\]
We adopt the convention that $G_0$ is the trivial group, so that $A_0 = \kk$.

For $1 \le i \le n-1$, let $s_i \in \fS_n$ be the simple transposition of $i$ and $i+1$.  The elements
\[
  \pi_i := s_i s_{i+1} \dotsm s_{n-1},\quad i=1,\dotsc,n,
\]
form a complete set of left coset representatives of $\fS_{n-1}$ in $\fS_n$.  Here we adopt the convention that $\pi_n=1_{\fS_n}$.

There is an injective group homomorphism
\[
    G_{n-1} \hookrightarrow G_n,\quad
    (g_{n-1},\dotsc,g_1) \mapsto (1_G,g_{n-1},\dotsc,g_1).
\]
This induces an embedding of the wreath product algebra $A_{n-1}$ into $A_n$.  For $g \in G$ and $1 \le i \le n$, define
\[
    g^{(i)} := (\underbrace{1_G,\dotsc,1_G}_{n-i \text{ entries}}, g, \underbrace{1_G,\dotsc,1_G}_{i-1 \text{ entries}}) \in G^n.
\]
Then the set
\begin{equation} \label{Carslaw}
    \{ g^{(i)} \pi_i = \pi_i g^{(n)} : 1 \le i \le n,\ g \in G \}
\end{equation}
is a complete set of left coset representatives of $G_{n-1}$ in $G_n$.  Hence it is a basis for $A_n$ as a right $A_{n-1}$-module.

Throughout the paper, we adopt the convention that $\bg = (g_n,\dotsc,g_1)$ and $\bh = (h_n,\dotsc,h_1)$.  That is, $g_i$ denotes the $i$-th component of $\bg$ (counting from right to left, as usual), and similarly for $h_i$.  This convention applies to all boldface letters denoting elements of $G^n$ for some $n \in \N$.

\begin{defin} \label{permrep}
    The \emph{permutation representation} of $A_n$ is the $\kk$-module $V = A^n$, with action given by
    \begin{equation} \label{perm1}
        \bg \pi \cdot (a_n,\dotsc,a_1)
        = (g_n a_{\pi^{-1}(n)}, \dotsc, g_1 a_{\pi^{-1}(1)}),
    \end{equation}
    $\bg \in G^n$, $\pi \in \fS_n$, $a_n,\dotsc,a_1 \in A$, extended by linearity.  In other words, $\fS_n$ acts by permuting the entries of elements of $A^n$, while $G^n$ acts by componentwise multiplication.
\end{defin}

For $i=1,\dotsc,n$, define $e_i=(0,\dotsc,0,1,0,\dotsc,0) \in V$, where the $1$ appears in the $i$-th position.  Then \cref{perm1} implies that
\begin{equation} \label{perm2}
    \bg \pi \cdot (ae_{i}) = g_{\pi(i)} a e_{\pi(i)},\quad
    \bg \in G^n,\ \pi \in \fS_n,\ a \in A,\ 1 \le i \le n.
\end{equation}
The set
\begin{equation} \label{basis}
    \BV := \{ g_k e_{i_k} \otimes \dotsb \otimes g_1 e_{i_1} : g_1,\dotsc,g_k \in G,\ 1 \le i_1,\dotsc,i_k \le n\}
\end{equation}
is a basis for $V^{\otimes k}$.

For the remainder of the paper, we write $\otimes_n$ for $\otimes_{A_n}$.

\begin{lem}
    For any $A_n$-module $W$, we have an isomorphism of $A_n$-modules
    \[
        A_n \otimes_{n-1} W \to V \otimes W,\quad
        \gamma \otimes w \mapsto \gamma e_n \otimes \gamma w,\quad \gamma \in G_n,\ w \in W,
    \]
    with inverse
    \[
        V \otimes W \to A_n \otimes_{n-1} W,\quad
        g e_i \otimes w \mapsto g^{(i)} \pi_i \otimes \pi_i^{-1} \left(g^{-1}\right)^{(i)} w,\quad 1 \le i \le n,\ g \in G,\ w \in W.
    \]
\end{lem}

\begin{proof}
    The first map is well-defined since $A_{n-1}$ acts trivially on $e_n$, and it is clearly a homomorphism of $A_n$-modules.  It is straightforward to verify that the second map is the inverse of the first.
\end{proof}

Denote by $\mathbf{1}_n$ the trivial one-dimensional $A_n$-module.  Let $B$ denote $A_n$, considered as an $(A_n,A_{n-1})$-bimodule, and, for $k \ge 1$, define
\[
    B^k := \underbrace{B \otimes_{n-1} B \otimes_{n-1} \dotsb \otimes_{n-1} B}_{k \text{ factors}}.
\]

\begin{cor}\label{definbeta}
    For $k \ge 1$, we have an isomorphism of $A_n$-modules
    \begin{align*}
        \beta_k \colon V^{\otimes k} &\xrightarrow{\cong} B^k \otimes \mathbf{1}_n, \\
        g_k e_{i_k} \otimes \dotsb \otimes g_1 e_{i_1}
        &\mapsto g_k^{(i_k)} \pi_{i_k} \otimes \pi_{i_k}^{-1} \left( g_k^{-1} \right)^{(i_k)} g_{k-1}^{(i_{k-1})} \pi_{i_{k-1}} \otimes \dotsb \otimes \pi_{i_2}^{-1} \left( g_2^{-1} \right)^{(i_2)} g_1^{(i_1)} \pi_{i_1} \otimes 1,
    \end{align*}
    with inverse map
    \begin{align*}
        \beta_k^{-1} \colon B^k \otimes \mathbf{1}_n &\xrightarrow{\cong} V^{\otimes k}, \\
        \gamma_k \otimes \dotsb \otimes \gamma_1 \otimes 1 &\mapsto
        (\gamma_k e_n) \otimes (\gamma_k \gamma_{k-1} e_n) \otimes \dotsb \otimes (\gamma_k \dotsm \gamma_1 e_n), \quad \gamma_k, \dotsc, \gamma_1 \in G_n.
    \end{align*}
\end{cor}

\section{Group partition categories\label{sec:ParG}}

We continue to fix a group $G$ and a commutative ring $\kk$.  For $k,l \in \N = \Z_{\ge 0}$, a \emph{partition of type $\binom{l}{k}$} is a partition of the set $\PS_k^l := \{1,\dotsc,k,1',\dotsc,l'\}$.  A \emph{$G$-partition of type $\binom{l}{k}$} is a pair $(P,\bg)$, where $P$ is a partition of type $\binom{l}{k}$ and $\bg = (g_1,\dotsc,g_k,g_{1'},\dotsc,g_{l'}) \in G^{\PS_k^l}$.  We define a \emph{part} of $(P,\bg)$ to be a part of the partition $P$.

Let $(P,\bg)$ and $(P,\bh)$ be $G$-partitions of type $\binom{l}{k}$, with $P = \{P_1, \dotsc, P_r\}$.  We say these $G$-partitions are \emph{equivalent}, and we write $(P,\bg) \sim (P,\bh)$, if there exist $t_1,\dotsc,t_r \in G$ such that, for each $i=1,\dotsc,r$, we have
\[
    h_a = t_i g_a
    \quad \text{for every } a \in P_i.
\]
This clearly defines an equivalence relation on the set of $G$-partitions of type $\binom{l}{k}$.  We let $[P,\bg]$ denote the equivalence class of $(P,\bg)$.

We depict the $G$-partition $(P,\bg)$ of type $\binom{l}{k}$ as a graph with $l$ vertices in the top row, labelled $g_{1'},\dotsc,g_{l'}$ from \emph{right to left}, and $k$ vertices in the bottom row, labelled $g_1,\dotsc,g_k$ from \emph{right to left}.  (We will always number vertices from right to left.)  We draw edges so that the parts of the partition are the connected components of the graph.

For example, the equivalence class of the $G$-partition $(P,\bg)$ of type $\binom{7}{5}$ with
\[
    P = \big\{ \{1,5\}, \{2\}, \{3,1'\}, \{4,4',7'\}, \{2', 3'\}, \{5'\}, \{6'\} \big\}
\]
can be depicted as follows:
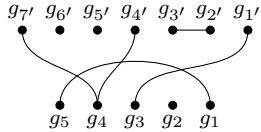
\[
  \begin{tikzpicture}[anchorbase]
    \pdg{0.5,0}{north}{g_5};
    \pdg{1,0}{north}{g_4};
    \pdg{1.5,0}{north}{g_3};
    \pdg{2,0}{north}{g_2};
    \pdg{2.5,0}{north}{g_1};
    \pdg{0,1}{south}{g_{7'}};
    \pdg{0.5,1}{south}{g_{6'}};
    \pdg{1,1}{south}{g_{5'}};
    \pdg{1.5,1}{south}{g_{4'}};
    \pdg{2,1}{south}{g_{3'}};
    \pdg{2.5,1}{south}{g_{2'}};
    \pdg{3,1}{south}{g_{1'}};
    \draw (1,0) \braidto (0,1);
    \draw (1,0) \braidto (1.5,1);
    \draw (0.5,0) to[out=up,in=up] (2.5,0);
    \draw (1.5,0) \braidto (3,1);
    \draw (2,1) to (2.5,1);
  \end{tikzpicture}
\]
We call this a \emph{$G$-partition diagram}.  Forgetting the labels, we obtain a \emph{partition diagram} for $P$.  Note that different $G$-partition diagrams can correspond to the same $G$-partition
since only the connected components of the graph are relevant, and similarly for partition
diagrams.  Two $G$-partition diagrams are equivalent if their graphs have the same connected components and the vertex labels of one are obtained from those of the other by, for each connected component, multiplying the labels in that component on the left by the same element of $G$.

\begin{eg}\label{equiv-G-diag}
    For $g,h,s,t \in G$, the following $G$-partition diagrams of type $\binom{4}{4}$ are equivalent:
    
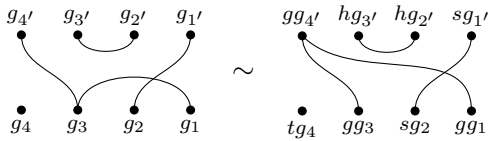
\[
        \begin{tikzpicture}[anchorbase]
            \pdg{0.75,0}{north}{g_4};
            \pdg{1.5,0}{north}{g_3};
            \pdg{2.25,0}{north}{g_2};
            \pdg{3,0}{north}{g_1};
            \pdg{0.75,1}{south}{g_{4'}};
            \pdg{1.5,1}{south}{g_{3'}};
            \pdg{2.25,1}{south}{g_{2'}};
            \pdg{3,1}{south}{g_{1'}};
            \draw (1.5,0) \braidto (0.75,1);
            \draw (2.25,0) \braidto (3,1);
            \draw (1.5,0) to[out=up,in=up] (3,0);
            \draw (1.5,1) to[out=down,in=down] (2.25,1);
        \end{tikzpicture}
        \ \sim \
        \begin{tikzpicture}[anchorbase]
            \pdg{0.75,0}{north}{tg_4};
            \pdg{1.5,0}{north}{gg_3};
            \pdg{2.25,0}{north}{sg_2};
            \pdg{3,0}{north}{gg_1};
            \pdg{0.75,1}{south}{gg_{4'}};
            \pdg{1.5,1}{south}{hg_{3'}};
            \pdg{2.25,1}{south}{hg_{2'}};
            \pdg{3,1}{south}{sg_{1'}};
            \draw (1.5,0) \braidto (0.75,1);
            \draw (2.25,0) \braidto (3,1);
            \draw (3,0) to[out=up,in=-45] (0.75,1);
            \draw (1.5,1) to[out=down,in=down] (2.25,1);
        \end{tikzpicture}
    \]
\end{eg}

Suppose $P$ is a partition of type $\binom{l}{k}$ and $Q$ is a partition of type $\binom{m}{l}$.  We can stack the partition diagram of $Q$ on top of the partition diagram of $P$ and identify the middle row of vertices to obtain a diagram $\stack(Q,P)$ with three rows of vertices.  We define $Q \star P$ to be the partition of type $\binom{k}{m}$  defined as follows: vertices are in the same part of $Q \star P$ if and only if the corresponding vertices in the top and bottom row of $\stack(Q,P)$ are in the same connected component.  We let $\alpha(Q,P)$ denote the number of connected components containing only vertices in the middle row of $\stack(Q,P)$.

\begin{eg} \label{logger}
    If
    \[
        P =
        \begin{tikzpicture}[anchorbase]
            \pd{1.5,0};
            \pd{2,0};
            \pd{2.5,0};
            \pd{3,0};
            \pd{0,1};
            \pd{0.5,1};
            \pd{1,1};
            \pd{1.5,1};
            \pd{2,1};
            \pd{2.5,1};
            \pd{3,1};
            \pd{3.5,1};
            \pd{4,1};
            \pd{4.5,1};
            \pd{5,1};
            \draw (1.5,0) \braidto (0.5,1);
            \draw (1.5,0) \braidto (2.5,1);
            \draw (2.5,0) \braidto (1,1);
            \draw (2.5,0) \braidto (2,1) -- (1.5,1);
            \draw (2,0) \braidto (3,1);
            \draw (3.5,1) to[out=down,in=down] (4.5,1);
            \draw (3,0) \braidto (5,1);
        \end{tikzpicture}
        \qquad \text{and}  \qquad
        Q =
        \begin{tikzpicture}[anchorbase]
            \pd{0,0};
            \pd{0.5,0};
            \pd{1,0};
            \pd{1.5,0};
            \pd{2,0};
            \pd{2.5,0};
            \pd{3,0};
            \pd{3.5,0};
            \pd{4,0};
            \pd{4.5,0};
            \pd{5,0};
            \pd{1,1};
            \pd{1.5,1};
            \pd{2,1};
            \pd{2.5,1};
            \pd{3,1};
            \draw (0,0) -- (0.5,0) \braidto (1,1);
            \draw (1.5,0) \braidto (1,1);
            \draw (1,0) \braidto (1.5,1);
            \draw (2.5,0) \braidto (1.5,1);
            \draw (3,0) -- (3,1);
            \draw (2,1) -- (2.5,1);
            \draw (5,0) \braidto (3,1);
        \end{tikzpicture}
    \]
    then $\alpha(P,Q) = 2$ and
    \[
        \stack(Q,P)
        \  =  \
        \begin{tikzpicture}[anchorbase]
            \pd{1.5,0};
            \pd{2,0};
            \pd{2.5,0};
            \pd{3,0};
            \pd{0,1};
            \pd{0.5,1};
            \pd{1,1};
            \pd{1.5,1};
            \pd{2,1};
            \pd{2.5,1};
            \pd{3,1};
            \pd{3.5,1};
            \pd{4,1};
            \pd{4.5,1};
            \pd{5,1};
            \pd{1,2};
            \pd{1.5,2};
            \pd{2,2};
            \pd{2.5,2};
            \pd{3,2};
            \draw (0,1) -- (0.5,1) \braidto (1,2);
            \draw (1.5,1) \braidto (1,2);
            \draw (1,1) \braidto (1.5,2);
            \draw (2.5,1) \braidto (1.5,2);
            \draw (3,1) -- (3,2);
            \draw (2,2) -- (2.5,2);
            \draw (1.5,0) \braidto (0.5,1);
            \draw (1.5,0) \braidto (2.5,1);
            \draw (2.5,0) \braidto (1,1);
            \draw (2.5,0) \braidto (2,1) -- (1.5,1);
            \draw (2,0) \braidto (3,1);
            \draw (3.5,1) to[out=down,in=down] (4.5,1);
            \draw (3,0) \braidto (5,1) \braidto (3,2);
        \end{tikzpicture}
        \ ,\quad
        Q \star P =
        \begin{tikzpicture}[anchorbase]
            \pd{0.5,0};
            \pd{1,0};
            \pd{1.5,0};
            \pd{2,0};
            \pd{0,1};
            \pd{0.5,1};
            \pd{1,1};
            \pd{1.5,1};
            \pd{2,1};
            \draw (0.5,0) -- (0.5,1) -- (0,1);
            \draw (1.5,0) \braidto (0.5,1);
            \draw (1,0) \braidto (2,1);
            \draw (1,1) -- (1.5,1);
            \draw (2,0) -- (2,1);
        \end{tikzpicture}
        \ .
    \]
\end{eg}

Suppose $(P,\bg)$ and $(Q,\bh)$ are $G$-partitions of types $\binom{l}{k}$ and $\binom{m}{l}$, respectively.  We define $\stack((Q,\bh),(P,\bg))$ to be the graph $\stack(Q,P)$ with vertices labeled by elements of $G$ as follows: vertices in the top and bottom rows are labeled as in the top and bottom rows of $Q$ and $P$, respectively, while the $i$-th vertex in the middle row is labelled by the product $g_{i'} h_i^{-1}$.  (As usual, we label vertices from right to left.)  We say that the pair $((Q,\bh),(P,\bg))$ is \emph{compatible} if any two vertices in the middle row of $\stack((Q,\bh),(P,\bg))$ that are in the same connected component of $Q$ have the same label.  If $((Q,\bh),(P,\bg))$ is compatible, we define $\bh \star_{Q,P} \bg$ to be the element $\bt \in G^{\PS_k^m}$ where
\begin{itemize}
    \item for $1 \le i \le m$, $t_{i'} = g h_{i'}$, where $g$ is the common label of vertices in the middle row that are in the same connected component as the $i$-th vertex of the top row of $Q$ (i.e.\ the vertex labeled by $h_{i'}$), where we adopt the convention $g=1_G$ if there are no such vertices;
    \item for $1 \le i \le k$, $t_i = g_i$.
\end{itemize}

\begin{lem} \label{cat}
    Suppose that $((Q,\bh),(P,\bg))$ and $((Q,\bh'),(P,\bg'))$ are compatible, that $(Q,\bh) \sim (Q,\bh')$, and that $(P,\bg) \sim (P,\bg')$.  Then $(Q \star P, \bh \star_{Q,P} \bg) \sim (Q \star P, \bh' \star_{Q,P} \bg')$.
\end{lem}

\begin{proof}
    By transitivity, it suffices to consider the case where $\bh = \bh'$ and the case where $\bg = \bg'$.  Suppose $\bh = \bh'$ and consider a connected component $Y$ of $\stack(Q,P)$.  Then $Y$ is a union of some connected components of $Q$ and some connected components $P_1,\dotsc,P_r$ of $P$.  (The case where $Y$ is a single connected component of $Q$ or a single connected component of $P$ are straightforward, so we assume $Y$ is a union of a positive number of connected components of $P$ and a positive number of connected components of $Q$.)  Since $(P,\bg) \sim (P,\bg')$, there exist $v_1,\dotsc,v_r \in G$ such that $g'_a = v_i g_a$ for all $a \in P_i$, $i \in \{1,2,\dotsc,r\}$.  Now, the fact that $((Q,\bh),(P,\bg))$ and $((Q,\bh),(P,\bg'))$ are compatible and that $P_1,\dotsc,P_r$ are in the same connected component of $\stack(Q,P)$ implies that $v_1=v_2=\dotsb=v_r$.  Thus, if $\bt = \bh \star_{Q,P} \bg$ and $\bt' = \bh \star_{Q,P} \bg'$, we have $t'_a = v_1 t_a$ for all vertices $a \in Y$.  Since this holds for each connected component $Y$ of $\stack(Q,P)$, we have $(Q \star P, \bh \star_{Q,P} \bg) \sim (Q \star P, \bh \star_{Q,P} \bg')$, as desired.  The case where $\bg = \bg'$ is analogous.
\end{proof}

We say that the pair $([Q,\bh], [P,\bg])$ is compatible if there exist representatives $(Q,\bh')$ and $(P,\bg')$ of the equivalence classes $[Q,\bh]$ and $[P,\bg]$ such that $((Q,\bh'),(P,\bg'))$ is compatible.  Whenever we refer to a compatible pair $([Q,\bh], [P,\bg])$, we assume that $((Q,\bh),(P,\bg))$ is a compatible pair of representatives.  By \cref{cat}, we can define
\begin{equation}
    [Q,\bh] \star [P,\bg] := [Q \star P, \bh \star_{Q,P} \bg]
\end{equation}
for a compatible pair $((Q,\bh), (P,\bg))$, and this definition is independent of our choice of a compatible pair of representatives.

\begin{eg} \label{chipper}
    If $P$ and $Q$ are as in \cref{logger}, then
    \[
        \stack((Q,\bh),(P,\bg))
        \  =  \
        \begin{tikzpicture}[anchorbase]
            \pd{1.5,0} node[anchor=north] {\dotlabel{g_4}};
            \pd{2,0} node[anchor=north] {\dotlabel{g_3}};
            \pd{2.5,0} node[anchor=north] {\dotlabel{g_2}};
            \pd{3,0} node[anchor=north] {\dotlabel{g_1}};
            \pd{0,1};
            \pd{0.5,1};
            \pd{1,1};
            \pd{1.5,1};
            \pd{2,1};
            \pd{2.5,1};
            \pd{3,1};
            \pd{3.5,1};
            \pd{4,1};
            \pd{4.5,1};
            \pd{5,1};
            \pd{1,2} node[anchor=south] {\dotlabel{h_{5'}}};
            \pd{1.5,2} node[anchor=south] {\dotlabel{h_{4'}}};
            \pd{2,2} node[anchor=south] {\dotlabel{h_{3'}}};
            \pd{2.5,2} node[anchor=south] {\dotlabel{h_{2'}}};
            \pd{3,2} node[anchor=south] {\dotlabel{h_{1'}}};
            \draw (0,1) -- (0.5,1) \braidto (1,2);
            \draw (1.5,1) \braidto (1,2);
            \draw (1,1) \braidto (1.5,2);
            \draw (2.5,1) \braidto (1.5,2);
            \draw (3,1) -- (3,2);
            \draw (2,2) -- (2.5,2);
            \draw (1.5,0) \braidto (0.5,1);
            \draw (1.5,0) \braidto (2.5,1);
            \draw (2.5,0) \braidto (1,1);
            \draw (2.5,0) \braidto (2,1) -- (1.5,1);
            \draw (2,0) \braidto (3,1);
            \draw (3.5,1) to[out=down,in=down] (4.5,1);
            \draw (3,0) \braidto (5,1) \braidto (3,2);
        \end{tikzpicture}
    \]
    where the vertices in the middle row are labeled $g_{1'} h_1^{-1}$, $g_{2'} h_2^{-1}$, $g_{3'} h_3^{-1}$, $g_{4'} h_4^{-1}$, $g_{5'} h_5^{-1}$, $g_{6'} h_6^{-1}$, $g_{7'} h_7^{-1}$, $g_{8'} h_8^{-1}$, $g_{9'} h_9^{-1}$, $g_{10'} h_{10}^{-1}$, $g_{11'} h_{11}^{-1}$ from right to left.  The pair $((Q,\bh),(P,\bg))$ is compatible if and only if $g_{1'}h_1^{-1} = g_{5'}h_5^{-1}$, $g_{6'} h_6^{-1} = g_{9'} h_9^{-1}$, and $g_{7'} h_7^{-1} = g_{8'} h_8^{-1} = g_{10'} h_{10}^{-1} = g_{11'} h_{11}^{-1}$.  If the pair is compatible, then
    \[
        (Q \star P, \bh \star_{Q,P} \bg) =
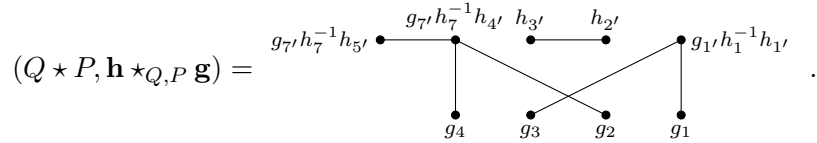
        \begin{tikzpicture}[anchorbase]
            \pd{1,0} node[anchor=north] {\dotlabel{g_4}};
            \pd{2,0} node[anchor=north] {\dotlabel{g_3}};
            \pd{3,0} node[anchor=north] {\dotlabel{g_2}};
            \pd{4,0} node[anchor=north] {\dotlabel{g_1}};
            \pd{0,1} node[anchor=east] {\dotlabel{g_{7'} h_7^{-1} h_{5'}}};
            \pd{1,1} node[anchor=south] {\dotlabel{g_{7'} h_7^{-1} h_{4'}}};
            \pd{2,1} node[anchor=south] {\dotlabel{h_{3'}}};
            \pd{3,1} node[anchor=south] {\dotlabel{h_{2'}}};
            \pd{4,1} node[anchor=west] {\dotlabel{g_{1'} h_1^{-1} h_{1'}}};
            \draw (1,0) -- (1,1) -- (0,1);
            \draw (3,0) -- (1,1);
            \draw (2,0) -- (4,1);
            \draw (2,1) -- (3,1);
            \draw (4,0) -- (4,1);
        \end{tikzpicture}
        \ .
    \]
\end{eg}

\begin{conv}
    From now on we will consider equivalent $G$-partition diagrams to be equal.  In other words, $G$-partition diagrams represent equivalence classes of $G$-partitions.  We will also use the terms \emph{part} (of a partition) and \emph{(connected) component} (of the corresponding partition diagram) interchangeably.
\end{conv}

Recall that $\kk$ is a commutative ring, and fix $d \in \kk$.

\begin{defin} \label{GPCbasis}
    The \emph{$G$-partition category} $\Par(G,d)$ is the strict $\kk$-linear monoidal category whose objects are nonnegative integers and, given two objects $k,l$ in $\Par(G,d)$, the morphisms from $k$ to $l$ are $\kk$-linear combinations of equivalence classes of $G$-partitions of type $\binom{l}{k}$.  The vertical composition is given by defining
    \[
        [Q,\bh] \circ [P,\bg] = d^{\alpha(Q,P)} [Q \star P, \bh \star_{Q,P} \bg]
    \]
    if $([Q,\bh], [P,\bg])$ is compatible, defining $[Q,\bh] \circ [P,\bg] = 0$ otherwise, and then extending by linearity.  The tensor product is given on objects by $k \otimes l := k+l$, and on morphisms by horizontal juxtaposition of $G$-partition diagrams, extended by linearity.  When we do not wish to make the group $G$ explicit, we call $\Par(G,d)$ a \emph{group partition category}.
\end{defin}

For example, if the pair $((Q,\bh),(P,\bg))$ of \cref{chipper} is compatible, then
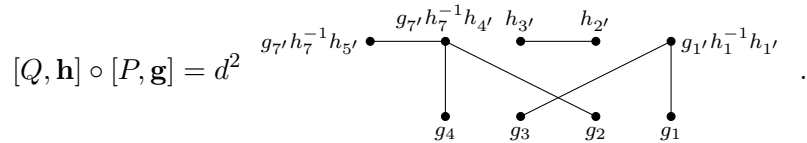
\[
    [Q,\bh] \circ [P,\bg] = d^2\
    \begin{tikzpicture}[anchorbase]
        \pd{1,0} node[anchor=north] {\dotlabel{g_4}};
        \pd{2,0} node[anchor=north] {\dotlabel{g_3}};
        \pd{3,0} node[anchor=north] {\dotlabel{g_2}};
        \pd{4,0} node[anchor=north] {\dotlabel{g_1}};
        \pd{0,1} node[anchor=east] {\dotlabel{g_{7'} h_7^{-1} h_{5'}}};
        \pd{1,1} node[anchor=south] {\dotlabel{g_{7'} h_7^{-1} h_{4'}}};
        \pd{2,1} node[anchor=south] {\dotlabel{h_{3'}}};
        \pd{3,1} node[anchor=south] {\dotlabel{h_{2'}}};
        \pd{4,1} node[anchor=west] {\dotlabel{g_{1'} h_1^{-1} h_{1'}}};
        \draw (1,0) -- (1,1) -- (0,1);
        \draw (3,0) -- (1,1);
        \draw (2,0) -- (4,1);
        \draw (2,1) -- (3,1);
        \draw (4,0) -- (4,1);
    \end{tikzpicture}
    \ .
\]
It is straightforward to verify that $\Par(G,d)$ is, in fact, a category, i.e.\ the composition of morphisms is associative.

The category $\Par(\{1\},d)$ is the \emph{partition category}; see \cite[\S 2]{Com16}.  In fact, we have a faithful functor
\begin{equation} \label{snow}
    \Par(\{1\},d) \to \Par(G,d)
\end{equation}
sending any partition diagram $P$ to the corresponding $G$-partition diagram where all vertices are labelled by $1_G$.  (More generally, we have a faithful functor $\Par(H,d) \to \Par(G,d)$ for any subgroup $H$ of $G$.)  In what follows, we will identify a partition diagram $P$ with its image under this functor.  In other words, we write $P$ and $[P]$ for $(P,\mathbf{1})$ and $[P,\mathbf{1}]$, respectively, where $\mathbf{1}$ is the identity element of $G^{\PS_k^l}$.  An arbitrary equivalence class of $G$-partitions $[P,\bg] \colon k \to l$ can be written in the form
\begin{equation} \label{sea}
    [P,\bg] =
    \begin{tikzpicture}[centerzero]
        \pd{-0.4,-0.2};
        \pd{0.4,-0.2};
        \pd{0.8,-0.2};
        \pd{-0.4,0.2} node[anchor=south] {\dotlabel{g_{l'}}};
        \pd{0.4,0.2} node[anchor=south] {\dotlabel{g_{2'}}};
        \pd{0.8,0.2} node[anchor=south] {\dotlabel{g_{1'}}};
        \draw (-0.4,-0.2) -- (-0.4,0.2);
        \draw (0.4,-0.2) -- (0.4,0.2);
        \draw (0.8,-0.2) -- (0.8,0.2);
        \node at (0.04,0) {$\cdots$};
    \end{tikzpicture}
    \circ
    [P]
    \circ
    \begin{tikzpicture}[centerzero]
        \pd{-0.4,-0.2} node[anchor=north] {\dotlabel{g_k}};
        \pd{0.4,-0.2} node[anchor=north] {\dotlabel{g_2}};
        \pd{0.8,-0.2} node[anchor=north] {\dotlabel{g_1}};
        \pd{-0.4,0.2};
        \pd{0.4,0.2};
        \pd{0.8,0.2};
        \draw (-0.4,-0.2) -- (-0.4,0.2);
        \draw (0.4,-0.2) -- (0.4,0.2);
        \draw (0.8,-0.2) -- (0.8,0.2);
        \node at (0.04,0) {$\cdots$};
    \end{tikzpicture}
    \ ,
\end{equation}
where we adopt the convention that unlabelled vertices implicitly carry the label $1_G$.

\section{Presentation\label{sec:present}}

In this section we give a presentation of group partition categories by generators and relations.  We use the usual string diagram calculus for monoidal categories.  We denote the unit object of a monoidal category by $\one$ and the identity morphism of an object $X$ by $1_X$.

\begin{defin} \label{GPC}
    Let $\Par(G)$ be the strict $\kk$-linear monoidal category with one generating object $\go$, where we denote
    \[
        \begin{tikzpicture}[anchorbase]
          \draw (0,-0.25) -- (0,0.25);
        \end{tikzpicture}
        \ := 1_\go,
    \]
    and generating morphisms
    \begin{gather*}
        \merge \colon \go \otimes \go \to \go, \quad
        \spliter \colon \go \to \go \otimes \go, \quad
        \crossing \colon \go \otimes \go \to \go \otimes \go,
        \\
        \bottompin \colon \one \to \go, \quad
        \toppin \colon \go \to \one, \quad
        \tokstrand \colon \go \to \go,\ g \in G,
    \end{gather*}
    subject to the following relations:
    \begin{gather} \label{GPC1}
        \begin{tikzpicture}[anchorbase]
            \draw (-0.4,-0.4) -- (0,0);
            \draw (0.25,-0.25) -- (0,0);
            \draw (0,0) -- (0,0.5);
            \opendot{0.25,-0.25};
        \end{tikzpicture}
        =
        \begin{tikzpicture}[anchorbase]
            \draw (0,-0.45) -- (0,0.45);
        \end{tikzpicture}
        \ ,\quad
        \begin{tikzpicture}[anchorbase]
            \draw (-0.4,0.4) -- (0,0);
            \draw (0.25,0.25) -- (0,0);
            \draw (0,0) -- (0,-0.5);
            \opendot{0.25,0.25};
        \end{tikzpicture}
        =
        \begin{tikzpicture}[anchorbase]
            \draw (0,0.45) -- (0,-0.45);
        \end{tikzpicture}
        =
        \begin{tikzpicture}[anchorbase]
            \draw (-0.25,0.25) -- (0,0);
            \draw (0.4,0.4) -- (0,0);
            \draw (0,0) -- (0,-0.5);
            \opendot{-0.25,0.25};
        \end{tikzpicture}
        \ ,\quad
        \begin{tikzpicture}[anchorbase]
            \draw (-0.4,-0.4) -- (-0.4,0) -- (-0.2,0.2) -- (0.2,-0.2) -- (0.4,0) -- (0.4,0.4);
            \draw (-0.2,0.2) -- (-0.2,0.4);
            \draw (0.2,-0.2) -- (0.2,-0.4);
        \end{tikzpicture}
        =
        \begin{tikzpicture}[anchorbase]
            \draw (-0.2,-0.4) -- (0,-0.2) -- (0.2,-0.4);
            \draw (-0.2,0.4) -- (0,0.2) -- (0.2,0.4);
            \draw (0,-0.2) -- (0,0.2);
        \end{tikzpicture}
        =
        \begin{tikzpicture}[anchorbase]
            \draw (0.4,-0.4) -- (0.4,0) -- (0.2,0.2) -- (-0.2,-0.2) -- (-0.4,0) -- (-0.4,0.4);
            \draw (0.2,0.2) -- (0.2,0.4);
            \draw (-0.2,-0.2) -- (-0.2,-0.4);
        \end{tikzpicture}
        \ ,
        \\ \label{GPC2}
        \begin{tikzpicture}[anchorbase]
            \draw (-0.2,-0.4) \braidto (0.2,0) \braidto (-0.2,0.4);
            \draw (0.2,-0.4) \braidto (-0.2,0) \braidto (0.2,0.4);
        \end{tikzpicture}
        =
        \begin{tikzpicture}[anchorbase]
            \draw (-0.2,-0.4) -- (-0.2,0.4);
            \draw (0.2,-0.4) -- (0.2,0.4);
        \end{tikzpicture}
        \ ,\quad
        \begin{tikzpicture}[anchorbase]
            \draw (-0.4,-0.4) -- (0.4,0.4);
            \draw (0.4,-0.4) -- (-0.4,0.4);
            \draw (0,-0.4) \braidto (-0.4,0) \braidto (0,0.4);
        \end{tikzpicture}
        =
        \begin{tikzpicture}[anchorbase]
            \draw (-0.4,-0.4) -- (0.4,0.4);
            \draw (0.4,-0.4) -- (-0.4,0.4);
            \draw (0,-0.4) \braidto (0.4,0) \braidto (0,0.4);
        \end{tikzpicture}
        \ ,
        \\ \label{GPC3}
        \begin{tikzpicture}[anchorbase]
            \draw (0.2,-0.2) -- (-0.3,0.3);
            \draw (-0.3,-0.3) -- (0.3,0.3);
            \opendot{0.2,-0.2};
        \end{tikzpicture}
        =
        \begin{tikzpicture}[anchorbase]
            \draw (-0.2,0) to (-0.2,0.3);
            \draw (0.2,-0.3) to (0.2,0.3);
            \opendot{-0.2,0};
        \end{tikzpicture}
        \ ,\quad
        \begin{tikzpicture}[anchorbase]
            \draw (0.2,0.2) -- (-0.3,-0.3);
            \draw (-0.3,0.3) -- (0.3,-0.3);
            \opendot{0.2,0.2};
        \end{tikzpicture}
        =
        \begin{tikzpicture}[anchorbase]
            \draw (-0.2,0) to (-0.2,-0.3);
            \draw (0.2,0.3) to (0.2,-0.3);
            \opendot{-0.2,0};
        \end{tikzpicture}
        \ ,\quad
        \begin{tikzpicture}[anchorbase]
            \draw (-0.4,-0.5) -- (0.2,0.3) -- (0.4,0.1) -- (0,-0.5);
            \draw (0.2,0.3) -- (0.2,0.6);
            \draw (-0.4,0.6) -- (-0.4,0.1) -- (0.4,-0.5);
        \end{tikzpicture}
        =
        \begin{tikzpicture}[anchorbase]
            \draw (-0.4,-0.4) -- (-0.2,-0.2) -- (-0.2,0) -- (0.2,0.4);
            \draw (0,-0.4) -- (-0.2,-0.2);
            \draw (0.2,-0.4) -- (0.2,0) -- (-0.2,0.4);
        \end{tikzpicture}
        \ ,\quad
        \begin{tikzpicture}[anchorbase]
            \draw (-0.4,0.5) -- (0.2,-0.3) -- (0.4,-0.1) -- (0,0.5);
            \draw (0.2,-0.3) -- (0.2,-0.6);
            \draw (-0.4,-0.6) -- (-0.4,-0.1) -- (0.4,0.5);
        \end{tikzpicture}
        =
        \begin{tikzpicture}[anchorbase]
            \draw (-0.4,0.4) -- (-0.2,0.2) -- (-0.2,0) -- (0.2,-0.4);
            \draw (0,0.4) -- (-0.2,0.2);
            \draw (0.2,0.4) -- (0.2,0) -- (-0.2,-0.4);
        \end{tikzpicture}
        \ ,
        \\ \label{GPC4}
        \begin{tikzpicture}[anchorbase]
          \draw (-0.2,-0.4) to[out=45,in=down] (0.2,0) to[out=up,in=-45] (0,0.2) -- (0,0.5);
          \draw (0.2,-0.4) to[out=135,in=down] (-0.2,0) to[out=up,in=225] (0,0.2);
        \end{tikzpicture}
        =
        \merge
        \ ,\quad
        \begin{tikzpicture}[anchorbase]
            \draw (0,-0.5) -- (0,-0.3) to[out=135,in=down] (-0.2,-0.1) -- (-0.2,0.1) to[out=up,in=225] (0,0.3) -- (0,0.5);
            \draw (0,-0.3) to[out=45,in=down] (0.2,-0.1) -- (0.2,0.1) to[out=up,in=-45] (0,0.3);
            \token{east}{-0.2,0}{g};
            \token{west}{0.2,0}{h};
        \end{tikzpicture}
        = \delta_{g,h}\
        \begin{tikzpicture}[anchorbase]
            \draw (0,-0.4) -- (0,0.4);
            \token{west}{0,0}{g};
        \end{tikzpicture}
        \ ,
        \\ \label{GPC5}
        \begin{tikzpicture}[anchorbase]
            \draw (0,-0.3) -- (0,0.3);
            \token{east}{0,-0.15}{h};
            \token{east}{0,0.15}{g};
        \end{tikzpicture}
        =
        \begin{tikzpicture}[anchorbase]
            \draw (0,-0.3) -- (0,0.3);
            \token{east}{0,0}{gh};
        \end{tikzpicture}
        \ ,\quad
        \begin{tikzpicture}[anchorbase]
            \draw (0,-0.3) -- (0,0.3);
            \token{east}{0,0}{1};
        \end{tikzpicture}
        =
        \begin{tikzpicture}[anchorbase]
            \draw (0,-0.3) -- (0,0.3);
        \end{tikzpicture}
        \ ,\quad
        \begin{tikzpicture}[anchorbase]
            \draw (-0.3,-0.3) -- (0.3,0.3);
            \draw (0.3,-0.3) -- (-0.3,0.3);
            \token{east}{-0.15,-0.15}{g};
        \end{tikzpicture}
        =
        \begin{tikzpicture}[anchorbase]
            \draw (-0.3,-0.3) -- (0.3,0.3);
            \draw (0.3,-0.3) -- (-0.3,0.3);
            \token{west}{0.15,0.15}{g};
        \end{tikzpicture}
        \ ,\quad
        \begin{tikzpicture}[anchorbase]
            \draw (0,-0.3) -- (0,0) -- (0.25,0.25);
            \draw (0,0) -- (-0.25,0.25);
            \token{east}{0,-0.15}{g};
        \end{tikzpicture}
        =
        \begin{tikzpicture}[anchorbase]
            \draw (0,-0.3) -- (0,0) -- (0.25,0.25);
            \draw (0,0) -- (-0.25,0.25);
            \token{east}{-0.125,0.125}{g};
            \token{west}{0.125,0.125}{g};
        \end{tikzpicture}
        \ ,\quad
        \begin{tikzpicture}[anchorbase]
            \draw (0,-0.3) -- (0,0.3);
            \opendot{0,-0.3};
            \token{east}{0,0}{g};
        \end{tikzpicture}
        =
        \bottompin
        \ .
    \end{gather}
\end{defin}

\begin{rem} \label{FrobObj}
    The relations \cref{GPC1} are equivalent to the statement that $(\go,\merge,\bottompin,\spliter,\toppin)$ is a Frobenius object (see, for example, \cite[Prop.~2.3.24]{Koc04}).  Relations \cref{GPC2,GPC3} and the third relation in \cref{GPC5} are precisely the statement that $\crossing$ equips $\Par(G)$ with the structure of a symmetric monoidal category (see, for example, \cite[\S1.3.27, \S1.4.35]{Koc04}).  Then the first relation in \cref{GPC4} is the statement that the Frobenius object $\go$ is commutative.  When $g=h=1_G$, the second relation in \cref{GPC4} is the statement that the Frobenius object $\go$ is special.
\end{rem}

We refer to the morphisms $\tokstrand$ as \emph{tokens}, and the open dots in $\toppin$ and $\bottompin$ as \emph{pins}.  We call $\merge$ a \emph{merge}, $\spliter$ a \emph{split}, and $\crossing$ a \emph{crossing}.  Define \emph{cups} and \emph{caps} by
\[
    \begin{tikzpicture}[anchorbase]
        \draw (-0.2,0.1) -- (-0.2,0) arc(180:360:0.2) -- (0.2,0.1);
    \end{tikzpicture}
    :=
    \begin{tikzpicture}[anchorbase]
        \draw (-0.2,0.2) -- (0,0) -- (0.2,0.2);
        \draw (0,0) -- (0,-0.2);
        \opendot{0,-0.2};
    \end{tikzpicture}
    \quad \text{and} \quad
    \begin{tikzpicture}[anchorbase]
        \draw (-0.2,-0.1) -- (-0.2,0) arc(180:0:0.2) -- (0.2,-0.1);
    \end{tikzpicture}
    :=
    \begin{tikzpicture}[anchorbase]
        \draw (-0.2,-0.2) -- (0,0) -- (0.2,-0.2);
        \draw (0,0) -- (0,0.2);
        \opendot{0,0.2};
    \end{tikzpicture}
    \ .
\]

\begin{prop} \label{Sydney}
    The following relations hold in $\Par(G)$:
    \begin{gather} \label{mirror}
        \begin{tikzpicture}[anchorbase]
            \draw (-0.2,-0.2) -- (0.3,0.3);
            \draw (0.3,-0.3) -- (-0.3,0.3);
            \opendot{-0.2,-0.2};
        \end{tikzpicture}
        =
        \begin{tikzpicture}[anchorbase]
            \draw (0.2,0) to (0.2,0.3);
            \draw (-0.2,-0.3) to (-0.2,0.3);
            \opendot{0.2,0};
        \end{tikzpicture}
        \ ,\quad
        \begin{tikzpicture}[anchorbase]
            \draw (-0.2,0.2) -- (0.3,-0.3);
            \draw (0.3,0.3) -- (-0.3,-0.3);
            \opendot{-0.2,0.2};
        \end{tikzpicture}
        =
        \begin{tikzpicture}[anchorbase]
            \draw (0.2,0) to (0.2,-0.3);
            \draw (-0.2,0.3) to (-0.2,-0.3);
            \opendot{0.2,0};
        \end{tikzpicture}
        \ ,\quad
        \begin{tikzpicture}[anchorbase]
            \draw (-0.25,-0.25) -- (0,0);
            \draw (0.4,-0.4) -- (0,0);
            \draw (0,0) -- (0,0.5);
            \opendot{-0.25,-0.25};
        \end{tikzpicture}
        =
        \begin{tikzpicture}[anchorbase]
            \draw (0,-0.45) -- (0,0.45);
        \end{tikzpicture}
        \ ,\quad
        \begin{tikzpicture}[anchorbase]
            \draw (-0.2,0.4) to[out=-45,in=up] (0.2,0) to[out=down,in=45] (0,-0.2) -- (0,-0.5);
            \draw (0.2,0.4) to[out=-135,in=up] (-0.2,0) to[out=down,in=135] (0,-0.2);
        \end{tikzpicture}
        =
        \spliter
        \ ,
        \\ \label{adjunction}
        \begin{tikzpicture}[anchorbase]
            \draw (-0.3,-0.5) to (-0.3,0) to[out=up,in=up,looseness=2] (0,0) to[out=down,in=down,looseness=2] (0.3,0) to (0.3,0.5);
        \end{tikzpicture}
        \ =\
        \begin{tikzpicture}[anchorbase]
            \draw (0,-0.5) to (0,0.5);
        \end{tikzpicture}
        \ =\
        \begin{tikzpicture}[anchorbase]
            \draw (-0.3,0.5) to (-0.3,0) to[out=down,in=down,looseness=2] (0,0) to[out=up,in=up,looseness=2] (0.3,0) to (0.3,-0.5);
        \end{tikzpicture}
        \ ,
        \\ \label{vortex1}
        \begin{tikzpicture}[anchorbase]
            \draw (-0.4,0.2) to[out=down,in=180] (-0.2,-0.2) to[out=0,in=225] (0,0);
            \draw (0,0) -- (0,0.2);
            \draw (0.3,-0.3) -- (0,0);
        \end{tikzpicture}
        =
        \spliter
        =
        \begin{tikzpicture}[anchorbase]
            \draw (0.4,0.2) to[out=down,in=0] (0.2,-0.2) to[out=180,in=-45] (0,0);
            \draw (0,0) -- (0,0.2);
            \draw (-0.3,-0.3) -- (0,0);
        \end{tikzpicture}
        \ ,\quad
        \begin{tikzpicture}[anchorbase]
            \draw (-0.4,-0.2) to[out=up,in=180] (-0.2,0.2) to[out=0,in=135] (0,0);
            \draw (0,0) -- (0,-0.2);
            \draw (0.3,0.3) -- (0,0);
        \end{tikzpicture}
        =
        \merge
        =
        \begin{tikzpicture}[anchorbase]
            \draw (0.4,-0.2) to[out=up,in=0] (0.2,0.2) to[out=180,in=45] (0,0);
            \draw (0,0) -- (0,-0.2);
            \draw (-0.3,0.3) -- (0,0);
        \end{tikzpicture}
        \ ,
        \\ \label{vortex2}
        \begin{tikzpicture}[anchorbase]
            \draw (-0.2,0.2) -- (0.2,-0.2);
            \draw (-0.4,0.2) to[out=down,in=225,looseness=2] (0,0) to[out=45,in=up,looseness=2] (0.4,-0.2);
        \end{tikzpicture}
        =
        \crossing
        =
        \begin{tikzpicture}[anchorbase]
            \draw (0.2,0.2) -- (-0.2,-0.2);
            \draw (0.4,0.2) to[out=down,in=-45,looseness=2] (0,0) to[out=135,in=up,looseness=2] (-0.4,-0.2);
        \end{tikzpicture}
        \ ,\quad
        \begin{tikzpicture}[anchorbase]
          \draw (0,0) arc(180:0:0.2) -- (0.4,-0.2);
          \opendot{0,0};
        \end{tikzpicture}
        =
        \toppin
        =
        \begin{tikzpicture}[anchorbase]
            \draw (0,0) arc(0:180:0.2) -- (-0.4,-0.2);
            \opendot{0,0};
        \end{tikzpicture}
        \ ,\quad
        \begin{tikzpicture}[anchorbase]
            \draw (0,0) arc(180:360:0.2) -- (0.4,0.2);
            \opendot{0,0};
        \end{tikzpicture}
        =
        \bottompin
        =
        \begin{tikzpicture}[anchorbase]
            \draw (0,0) arc(360:180:0.2) -- (-0.4,0.2);
            \opendot{0,0};
        \end{tikzpicture}
        \ ,
        \\ \label{ramp}
        \begin{tikzpicture}[anchorbase]
            \draw (-0.2,0.2) -- (-0.2,0) arc(180:360:0.2) -- (0.2,0.2);
            \token{east}{-0.2,0}{g};
        \end{tikzpicture}
        =
        \begin{tikzpicture}[anchorbase]
            \draw (-0.2,0.2) -- (-0.2,0) arc(180:360:0.2) -- (0.2,0.2);
            \token{west}{0.2,0}{g^{-1}};
        \end{tikzpicture}
        ,\quad
        \begin{tikzpicture}[anchorbase]
            \draw (-0.2,-0.2) -- (-0.2,0) arc(180:0:0.2) -- (0.2,-0.2);
            \token{east}{-0.2,0}{g};
        \end{tikzpicture}
        =
        \begin{tikzpicture}[anchorbase]
            \draw (-0.2,-0.2) -- (-0.2,0) arc(180:0:0.2) -- (0.2,-0.2);
            \token{west}{0.2,0}{g^{-1}};
        \end{tikzpicture}
        \ ,\quad
        \begin{tikzpicture}[anchorbase]
            \draw (-0.3,-0.3) -- (0.3,0.3);
            \draw (0.3,-0.3) -- (-0.3,0.3);
            \token{west}{0.15,-0.15}{g};
        \end{tikzpicture}
        =
        \begin{tikzpicture}[anchorbase]
            \draw (-0.3,-0.3) -- (0.3,0.3);
            \draw (0.3,-0.3) -- (-0.3,0.3);
            \token{east}{-0.15,0.15}{g};
        \end{tikzpicture}
        \ ,\quad
        \begin{tikzpicture}[anchorbase]
            \draw (0,-0.3) -- (0,0.3);
            \opendot{0,0.3};
            \token{east}{0,0}{g};
        \end{tikzpicture}
        =
        \toppin\ ,
        \\ \label{pool}
        \begin{tikzpicture}[anchorbase]
            \draw (0,-0.3) -- (0,0) -- (0.25,0.25);
            \draw (0,0) -- (-0.25,0.25);
            \token{east}{-0.125,0.125}{g};
        \end{tikzpicture}
        =
        \begin{tikzpicture}[anchorbase]
            \draw (0,-0.3) -- (0,0) -- (0.25,0.25);
            \draw (0,0) -- (-0.25,0.25);
            \token{east}{0,-0.15}{g};
            \token{west}{0.125,0.125}{g^{-1}};
        \end{tikzpicture}
        ,\
        \begin{tikzpicture}[anchorbase]
            \draw (0,-0.3) -- (0,0) -- (0.25,0.25);
            \draw (0,0) -- (-0.25,0.25);
            \token{west}{0.125,0.125}{g};
        \end{tikzpicture}
        =
        \begin{tikzpicture}[anchorbase]
            \draw (0,-0.3) -- (0,0) -- (0.25,0.25);
            \draw (0,0) -- (-0.25,0.25);
            \token{east}{0,-0.15}{g};
            \token{east}{-0.125,0.125}{g^{-1}};
        \end{tikzpicture}
        ,\
        \begin{tikzpicture}[anchorbase]
            \draw (0,0.3) -- (0,0) -- (0.25,-0.25);
            \draw (0,0) -- (-0.25,-0.25);
            \token{east}{0,0.15}{g};
        \end{tikzpicture}
        =
        \begin{tikzpicture}[anchorbase]
            \draw (0,0.3) -- (0,0) -- (0.25,-0.25);
            \draw (0,0) -- (-0.25,-0.25);
            \token{east}{-0.125,-0.125}{g};
            \token{west}{0.125,-0.125}{g};
        \end{tikzpicture}
        ,\
        \begin{tikzpicture}[anchorbase]
            \draw (0,0.3) -- (0,0) -- (0.25,-0.25);
            \draw (0,0) -- (-0.25,-0.25);
            \token{east}{-0.125,-0.125}{g};
        \end{tikzpicture}
        =
        \begin{tikzpicture}[anchorbase]
            \draw (0,0.3) -- (0,0) -- (0.25,-0.25);
            \draw (0,0) -- (-0.25,-0.25);
            \token{east}{0,0.15}{g};
            \token{west}{0.125,-0.125}{g^{-1}};
        \end{tikzpicture}
        ,\
        \begin{tikzpicture}[anchorbase]
            \draw (0,0.3) -- (0,0) -- (0.25,-0.25);
            \draw (0,0) -- (-0.25,-0.25);
            \token{west}{0.125,-0.125}{g};
        \end{tikzpicture}
        =
        \begin{tikzpicture}[anchorbase]
            \draw (0,0.3) -- (0,0) -- (0.25,-0.25);
            \draw (0,0) -- (-0.25,-0.25);
            \token{east}{0,0.15}{g};
            \token{east}{-0.125,-0.125}{g^{-1}};
        \end{tikzpicture}
        \ .
    \end{gather}
\end{prop}

\begin{proof}
    The first two relations in \cref{mirror} follow from the first two relations in \cref{GPC3} by composing on the top and bottom, respectively, with the crossing and using \cref{GPC2}.  Then the third relation in \cref{mirror} follows from the first relation in \cref{GPC1} using the first relation in \cref{GPC4} and the first relation in \cref{mirror}.  (The fourth relation in \cref{mirror} will be proved below.)

    The relations \cref{adjunction} follow from the fourth and fifth equalities in \cref{GPC1} after placing pins on the merges and splits.  The first equality in \cref{vortex1} follows from the fifth equality in \cref{GPC1} after placing a pin on the bottom-left of both diagrams involved in the equality.  The remaining equalities in \cref{vortex1} are proved similarly.

    Starting with the third relation in \cref{GPC3}, adding a pin to the top-right strand, a crossing to the two rightmost strands at the bottom, and using the second relation in \cref{GPC3} and the first relation in \cref{GPC2}, we obtain the relation
    \[
        \begin{tikzpicture}[anchorbase]
          \draw (-0.2,-0.1) -- (-0.2,0) arc(180:0:0.2) -- (0.2,-0.1);
          \draw (0,-0.1) -- (-0.3,0.3);
        \end{tikzpicture}
        =
        \begin{tikzpicture}[anchorbase]
          \draw (-0.2,-0.1) -- (-0.2,0) arc(180:0:0.2) -- (0.2,-0.1);
          \draw (0,-0.1) -- (0.3,0.3);
        \end{tikzpicture}
        \ .
    \]
    Adding a strand on the left (i.e.\ tensoring on the left with $1_\go$), then adding a cup to the two leftmost bottom strands and using \cref{adjunction}, yields the first equality in \cref{vortex2}.  The second equality in \cref{vortex2} is proved similarly.  The remaining relations in \cref{vortex2} follow from placing pins at the top of the morphisms in first relation in \cref{GPC1} and the third relation in \cref{mirror} and from placing pins at the bottom of the morphisms in the second two equalities in \cref{GPC1}.

    Now the fourth relation in \cref{mirror} follows from rotating the first relation in \cref{GPC4} using the cups and caps, together with \cref{adjunction,vortex1,vortex2}.

    To prove the first relation in \cref{ramp}, we compute
    \[
        \begin{tikzpicture}[anchorbase]
            \draw (-0.2,0.2) -- (-0.2,0) arc(180:360:0.2) -- (0.2,0.2);
            \token{east}{-0.2,0}{g};
        \end{tikzpicture}
        =
        \begin{tikzpicture}[anchorbase]
            \draw (-0.4,0.4) -- (0,0) -- (0.4,0.4);
            \draw (0,0) -- (0,-0.3);
            \opendot{0,-0.3};
            \token{east}{-0.2,0.2}{g};
        \end{tikzpicture}
        =
        \begin{tikzpicture}[anchorbase]
            \draw (-0.4,0.4) -- (0,0) -- (0.4,0.4);
            \draw (0,0) -- (0,-0.3);
            \opendot{0,-0.3};
            \token{east}{-0.2,0.2}{g};
            \token{north}{0.14,0.14}{g};
            \token{west}{0.27,0.27}{g^{-1}};
        \end{tikzpicture}
        =
        \begin{tikzpicture}[anchorbase]
            \draw (-0.4,0.4) -- (0,0) -- (0.4,0.4);
            \draw (0,0) -- (0,-0.3);
            \opendot{0,-0.3};
            \token{west}{0.2,0.2}{g^{-1}};
            \token{east}{0,-0.1}{g};
        \end{tikzpicture}
        =
        \begin{tikzpicture}[anchorbase]
            \draw (-0.4,0.4) -- (0,0) -- (0.4,0.4);
            \draw (0,0) -- (0,-0.3);
            \opendot{0,-0.3};
            \token{west}{0.2,0.2}{g^{-1}};
        \end{tikzpicture}
        =
        \begin{tikzpicture}[anchorbase]
            \draw (-0.2,0.2) -- (-0.2,0) arc(180:360:0.2) -- (0.2,0.2);
            \token{west}{0.2,0}{g^{-1}};
        \end{tikzpicture}
        \ .
    \]
    Then we prove the second relation in \cref{ramp} as follows:
    \[
        \begin{tikzpicture}[anchorbase]
            \draw (-0.2,-0.2) -- (-0.2,0) arc(180:0:0.2) -- (0.2,-0.2);
            \token{east}{-0.2,0}{g};
        \end{tikzpicture}
        \overset{\cref{adjunction}}{=}
        \begin{tikzpicture}[centerzero={(0,-0.1)}]
            \draw (-0.6,-0.4) -- (-0.6,0) arc(180:0:0.2) arc(180:360:0.2) arc(180:0:0.2) -- (0.6,-0.4);
            \token{west}{0.2,0}{g};
        \end{tikzpicture}
        =
        \begin{tikzpicture}[centerzero={(0,-0.1)}]
            \draw (-0.6,-0.4) -- (-0.6,0) arc(180:0:0.2) arc(180:360:0.2) arc(180:0:0.2) -- (0.6,-0.4);
            \token{south west}{-0.2,0}{g^{-1}};
        \end{tikzpicture}
        \overset{\cref{adjunction}}{=}
        \begin{tikzpicture}[anchorbase]
            \draw (-0.2,-0.2) -- (-0.2,0) arc(180:0:0.2) -- (0.2,-0.2);
            \token{west}{0.2,0}{g^{-1}};
        \end{tikzpicture}
        \ .
    \]
    The third relation in \cref{ramp} is obtained from the third relation in \cref{GPC5} by composing on the top and bottom with a crossing and then using \cref{GPC2}.  The fourth relation in \cref{ramp} is obtained by starting with the third equality in \cref{vortex2}, attaching a token labelled $g$ to the bottom of the strands, then using the second relation in \cref{ramp} and the fifth relation in \cref{GPC5}.

    Finally, the relations in \cref{pool} are obtained from the fourth relation in \cref{GPC5} by attaching the appropriate cups and caps, then using \cref{vortex1,ramp}.
\end{proof}

The relation \cref{adjunction} implies that the object $\go$ is self-dual.   It follows from \cref{Sydney} that the cups and caps endow $\Par(G)$ with the structure of a strict pivotal category: we have an isomorphism of strict monoidal categories
\[
    \ast \colon \Par(G) \to \left( \Par(G)^\op \right)^\rev,
\]
where $\op$ denotes the opposite category and $\rev$ denotes the reversed category (switching the order of the tensor product).  This isomorphism is the identity on objects and, for a general morphism $f$ represented by a single string diagram, the morphism $f^*$ is given by rotating the diagram through $180\degree$.

Moreover, we have that morphisms are invariant under isotopy, except that we must use \cref{ramp} when we slide tokens over cups and caps.  In addition, it follows from the first two relations in \cref{GPC3,mirror} that
\begin{equation}
    \lolly\
    \begin{tikzpicture}[anchorbase]
        \draw (0,-0.4) -- (0,0.4);
    \end{tikzpicture}
    =
    \begin{tikzpicture}[anchorbase]
        \draw (0,-0.4) -- (0,0.4);
    \end{tikzpicture}
    \ \lolly.
\end{equation}
In other words, the morphism $\lolly$ is strictly central.

\begin{theo} \label{twocats}
    Let $d \in \kk$.  As a $\kk$-linear monoidal category, $\Par(G,d)$ is isomorphic to the quotient of $\Par(G)$ by the relation
    \begin{equation} \label{GPC6}
        \lolly\ = d.
    \end{equation}
\end{theo}

\begin{proof}
    Let $\Par'(G,d)$ denote the quotient of $\Par(G)$ by the additional relation \cref{GPC6}.  We define a functor $F \colon \Par'(G,d) \to \Par(G,d)$ as follows:  On objects, define $F(\go^{\otimes k}) = k$, $k \in \N$.  We define $F$ on the generating morphisms by
    \begin{gather*}
        \merge \mapsto
        \begin{tikzpicture}[anchorbase]
            \pd{-0.2,0};
            \pd{0.2,0};
            \pd{0,0.4};
            \draw (-0.2,0) \braidto (0,0.4);
            \draw (0.2,0) \braidto (0,0.4);
        \end{tikzpicture}
        \ ,\quad
        \spliter \mapsto
        \begin{tikzpicture}[anchorbase]
            \pd{-0.2,0.4};
            \pd{0.2,0.4};
            \pd{0,0};
            \draw (0,0) \braidto (-0.2,0.4);
            \draw (0,0) \braidto (0.2,0.4);
        \end{tikzpicture}
        \ ,\quad
        \crossing \mapsto
        \begin{tikzpicture}[anchorbase]
            \pd{-0.2,0.4};
            \pd{0.2,0.4};
            \pd{-0.2,0};
            \pd{0.2,0};
            \draw (0.2,0) -- (-0.2,0.4);
            \draw (-0.2,0) -- (0.2,0.4);
        \end{tikzpicture}
        \ ,\quad
        \bottompin \mapsto
        \begin{tikzpicture}[centerzero]
            \pd{0,0.2};
        \end{tikzpicture}
        \ ,\quad
        \toppin \mapsto
        \begin{tikzpicture}[centerzero]
            \pd{0,-0.2};
        \end{tikzpicture}
        \ ,\quad
        \tokstrand \mapsto
        \begin{tikzpicture}[centerzero]
            \pd{0,-0.2} node[anchor=north] {\dotlabel{g}};
            \pd{0,0.2};
            \draw (0,-0.2) -- (0,0.2);
        \end{tikzpicture}
        =
        \begin{tikzpicture}[centerzero]
            \pd{0,-0.2};
            \pd{0,0.2} node[anchor=south] {\dotlabel{g^{-1}}};
            \draw (0,-0.2) -- (0,0.2);
        \end{tikzpicture}
        \ ,
    \end{gather*}
    where, by convention, unlabelled vertices in $G$-partition diagrams carry the label $1_G$.  Note that the image of $\bottompin$ is the unique $G$-partition diagram of type $\binom{1}{0}$ (i.e.\ the vertex there is in the top row), while the image of $\toppin$ is the unique $G$-partition diagram of type $\binom{0}{1}$.  It is straightforward to verify that the relations \cref{GPC1,GPC2,GPC3,GPC4,GPC5,GPC6} are preserved by $F$, so that $F$ is well defined.

    Since $F$ is clearly bijective on objects, it remains to show that it is full and faithful.  Since \cref{twocats} is known to hold in the case where $G$ is the trivial group (see \cite[Th.~2.1]{Com16}, or \cite[Prop.~2.1]{NS19} for a diagrammatic treatment), it follows from the existence of the functor \cref{snow} that any partition diagram $[P] = [P,\mathbf{1}]$ is in the image of $F$.  Thus, by \cref{sea} so is an arbitrary equivalence class $[P,\bg]$ of $G$-partitions.  Hence $F$ is full.

    It remains to prove that $F$ is faithful.  To do this, it suffices to show that
    \[
        \dim \Hom_{\Par'(G,d)}(\go^{\otimes k}, \go^{\otimes l}) \le
        \dim \Hom_{\Par(G,d)}(k,l)
        \quad \text{for all } k,l \in \N.
    \]
    We do this by showing that every morphism of $\Par'(G,d)$ obtained from the generators by composition and tensor product can be reduced to a scalar multiple of a standard form, with the standard forms being in natural bijection with the number of $G$-partition diagrams.

    We first introduce \emph{star diagrams} $S_a^b \in \Hom_{\Par'(G,d)}(\go^{\otimes a},\go^{\otimes b})$ for $(a,b) \in \N^2 \setminus \{(0,0)\}$ as follows.  Define
    \[
        S_1^0 := \toppin,\quad
        S_0^1 := \bottompin,\quad
        S_1^1 := \idstrand\ .
    \]
    Then define general star diagrams recursively by
    \begin{align*}
        S_a^{b+1} &:= \left( 1_{\go^{\otimes (b-1)}} \otimes \spliter \right) \circ S_a^b \quad \text{for } b \ge 1, \\
        S_{a+1}^b &:= S_a^b \circ \left( 1_{\go^{\otimes (a-1)}} \otimes \merge \right) \quad \text{for } a \ge 1.
    \end{align*}
    For example, we have
    \[
        S_2^0 =
        \begin{tikzpicture}[anchorbase]
            \draw (-0.2,-0.2) -- (0,0) -- (0.2,-0.2);
            \draw (0,0) -- (0,0.2);
            \opendot{0,0.2};
        \end{tikzpicture}
        =
        \begin{tikzpicture}[anchorbase]
            \draw (-0.2,-0.1) -- (-0.2,0) arc(180:0:0.2) -- (0.2,-0.1);
        \end{tikzpicture}
        \ ,\quad
        S_0^3 =
        \begin{tikzpicture}[anchorbase]
            \draw (0,-0.2) -- (0,0.6);
            \draw (0,0) -- (0.3,0.3) -- (0.3,0.6);
            \draw (0.3,0.3) -- (0.6,0.6);
            \opendot{0,-0.2};
        \end{tikzpicture}
        \ ,\quad
        S_3^4 =
        \begin{tikzpicture}[anchorbase]
            \draw (0,0) -- (0,0.9);
            \draw (0,0) -- (0.3,0.3) -- (0.3,0.9);
            \draw (0.3,0.3) -- (0.6,0.6) -- (0.6,0.9);
            \draw (0.6,0.6) -- (0.9,0.9);
            \draw (0,0) -- (0,-0.9);
            \draw (0,-0.3) -- (0.3,-0.6) -- (0.3,-0.9);
            \draw (0.3,-0.6) -- (0.6,-0.9);
        \end{tikzpicture}
        \ .
    \]
    Every permutation $\pi \in \fS_k$ gives rise to a partition of type $\binom{k}{k}$ with parts $\{i,\pi(i)'\}$, $1 \le i \le k$.  Fixing a reduced decomposition for $\pi$ induces a decomposition of the corresponding partition diagram as a composition of tensor products of the generator $\crossing$ and identity morphisms.  We fix such a decomposition for each permutation, writing $D_\pi$ for the corresponding element of $\Par'(G,d)$.  For example, if we choose the reduced decomposition $s_1 s_2 s_1$ for the permutation $(1\, 3) \in \fS_3$, we have
    \[
        D_{(1\, 3)} =
        \begin{tikzpicture}[anchorbase]
            \draw (-0.4,-0.4) -- (0.4,0.4);
            \draw (0.4,-0.4) -- (-0.4,0.4);
            \draw (0,-0.4) \braidto (0.4,0) \braidto (0,0.4);
        \end{tikzpicture}
        \ .
    \]

    Now, fix a representative $(P,\bg)$ of each equivalence class $[P,\bg]$ of $G$-partitions.  Then, for each such representative, fix a \emph{standard decomposition}
    \begin{equation} \label{supernova}
        [P,\bg] =
        \begin{tikzpicture}[centerzero]
            \pd{-0.4,-0.2};
            \pd{0.4,-0.2};
            \pd{0.8,-0.2};
            \pd{-0.4,0.2} node[anchor=south] {\dotlabel{g_{l'}}};
            \pd{0.4,0.2} node[anchor=south] {\dotlabel{g_{2'}}};
            \pd{0.8,0.2} node[anchor=south] {\dotlabel{g_{1'}}};
            \draw (-0.4,-0.2) -- (-0.4,0.2);
            \draw (0.4,-0.2) -- (0.4,0.2);
            \draw (0.8,-0.2) -- (0.8,0.2);
            \node at (0.04,0) {$\cdots$};
        \end{tikzpicture}
        \circ
        F(D_\pi) \circ F(S) \circ F(D_\sigma)
        \circ
        \begin{tikzpicture}[centerzero]
            \pd{-0.4,-0.2} node[anchor=north] {\dotlabel{g_k}};
            \pd{0.4,-0.2} node[anchor=north] {\dotlabel{g_2}};
            \pd{0.8,-0.2} node[anchor=north] {\dotlabel{g_1}};
            \pd{-0.4,0.2};
            \pd{0.4,0.2};
            \pd{0.8,0.2};
            \draw (-0.4,-0.2) -- (-0.4,0.2);
            \draw (0.4,-0.2) -- (0.4,0.2);
            \draw (0.8,-0.2) -- (0.8,0.2);
            \node at (0.04,0) {$\cdots$};
        \end{tikzpicture}
        \ ,
    \end{equation}
    where $\pi \in \fS_l$, $\sigma \in \fS_k$, and $S$ is a tensor product of star diagrams.  The existence of such a standard decomposition follows from \cref{sea}, together with the existence of a decomposition of the form $[P] = F(D_\pi) \circ F(S) \circ F(D_\sigma)$ for any partition diagram $P$. (Precisely, $S$ is a tensor product of star diagrams, one for each connected component in $P$, and then $\pi$ and $\sigma$ are permutations such that $F(D_\pi)$ and $F(D_\sigma)$ connect the top and bottom vertices of $P$ to the corresponding legs of $S$.)  Then define
    \begin{equation} \label{formica}
        y_{P,\bg} :=
        \left(
            \tokstrand[g_{l'}^{-1}]
            \cdots
            \tokstrand[g_{2'}^{-1}]
            \tokstrand[g_{1'}^{-1}]
            \
        \right)
        \circ
        D_\pi \circ S \circ D_\sigma
        \circ
        \left(
            \tokstrand[g_k]
            \cdots
            \tokstrand[g_{2'}]
            \tokstrand[g_{1'}]
            \
        \right).
    \end{equation}
    Hence $F(y_{P,\bg}) = [P,\bg]$

    To complete the proof that $F$ is faithful, it remains to show that any morphism in $\Par'(G,d)$ that is obtained from the generators by tensor product and composition is equal to a scalar multiple of $y_{P,\bg}$ for some chosen representative $(P,\bg)$.  As noted above, \cref{twocats} holds for the partition category, which is the case where $G = \{1\}$ is the trivial group.  Now, if we ignore tokens, the relations \cref{GPC1,GPC2,GPC3,GPC4,GPC6} correspond to the relations in the $G = \{1\}$ case, except for the fact that the second relation in \cref{GPC4} gives zero when $g \ne h$.  It follows that every morphism in $\Par'(G,d)$ obtained from the generators by tensor product and composition is equal to a (potentially zero) scalar multiple of a morphism obtained from some $D_\pi \circ S \circ D_\sigma$ by adding tokens (since, ignoring tokens, this can be done in the partition category).  Then, since the string diagram for $D_\pi \circ S \circ D_\sigma$ is a tree (i.e.\ contains no cycles), one can use relations \cref{GPC5,ramp,pool} to move all tokens to the ends of strings and combine them into a single token at each endpoint.  This yields a diagram of the form \cref{formica}, except that the tokens may not correspond to our chosen representative of the equivalence class of $G$-partitions.  However, we then use the relations \cref{GPC5,ramp,pool} to adjust the tokens at the endpoints so that we obtain the chosen representative.
\end{proof}

In the context of \cref{FrobObj}, \cref{GPC6} is the statement that the Frobenius object $\go$ has dimension $d$.  From now on, we will identify $\Par(G,d)$ with the quotient of $\Par(G)$ by the relation \cref{GPC6} via the isomorphism of \cref{twocats}.  In particular, we will identify the object $k$ of $\Par(G,d)$ with the object $\go^{\otimes k}$ of $\Par(G)$, for $k \in \N$.

\begin{rem} \label{gummies}
    Suppose we define $\Par(G,d)$ as in \cref{GPCbasis}, but over the ring $\kk[d]$, so that $d$ is an indeterminate.  It then follows from \cref{twocats} that $\Par(G,d)$ is isomorphic to $\Par(G)$ as a $\kk$-linear monoidal category.  Under this isomorphism, $d 1_\one$ corresponds to $\lolly$.
\end{rem}

For $k \in \N$, we define the \emph{$G$-partition algebra}
\begin{equation} \label{fort}
    P_k(G,d) := \End_{\Par(G,d)}(\go^{\otimes k}).
\end{equation}
When we do not wish to make $G$ explicit, we call these \emph{group partition algebras}.  These algebras appeared in \cite{Blo03}, where they are called \emph{$G$-colored partition algebras}.  A diagrammatic description of these algebras, different from that of the current paper, is given \cite[\S 6.2]{Blo03}.

\section{Categorical action\label{sec:action}}

In this section we assume that the group $G$ is finite of order $|G|$.  We define a categorical action of the $G$-partition category on the category of modules for the wreath product groups $G_n = G^n \rtimes \fS_n$.  We first describe this action by giving the action of the generators, and then describe the action of an arbitrary $G$-partition diagram.  Recall that $A_n = \kk G_n$ is the group algebra of $G_n$, and so we can naturally identify $A_n$-modules and representations of $G_n$.

\begin{theo} \label{hide}
    For $n \in \N$, we have a strong monoidal $\kk$-linear functor $\Phi_n \colon \Par(G,n|G|) \to A_n\md$ given as follows.  On objects, $\Phi_n$ is determined by $\Phi_n(\go) = V$.  On generating morphisms, $\Phi_n$ is given by
    \begin{align*}
        \Phi_n(\merge) &\colon V \otimes V \to V, &
        g e_i \otimes h e_j &\mapsto \delta_{g,h} \delta_{i,j} g e_i,
        \\
        \Phi_n(\spliter) &\colon V \to V \otimes V, &
        g e_i &\mapsto g e_i \otimes g e_i,
        \\
        \Phi_n(\crossing) &\colon V \otimes V \to V, &
        v \otimes w &\mapsto w \otimes v,
        \\
        \Phi_n(\bottompin) &\colon \mathbf{1}_n \to V, &
        1 &\mapsto \textstyle \sum_{g \in G} \sum_{i=1}^n g e_i,
        \\
        \Phi_n(\toppin) &\colon V \to \mathbf{1}_n, &
        g e_i &\mapsto 1,
        \\
        \Phi_n(\tokstrand) & \colon V \to V, &
        h e_i &\mapsto h g^{-1} e_i,
    \end{align*}
    for $g,h \in G$, $1 \le i,j \le n$, $v,w \in V$.
\end{theo}

\begin{proof}
    We must show that the action preserves the relations \cref{GPC1,GPC2,GPC3,GPC4,GPC5}.

    \medskip

    \noindent \emph{Relations \cref{GPC1}}:  To prove the first three equalities in \cref{GPC1}, we compute
    \begin{gather*}
        \Phi_n \left( \merge \right) \circ \Phi_n \left( \, \idstrand \otimes \bottompin \right) (g e_i)
        = \sum_{h\in G} \sum_{j=1}^{n} \Phi_n \left( \merge \right) (g e_i\otimes h e_j)
        = \sum_{h\in G} \sum_{j=1}^{n} \delta_{g,h}\delta_{i,j} g e_i
        = g e_i,
        \\
        \Phi_n \left(\, \idstrand \otimes \toppin \right) \circ \Phi_n \left( \spliter \right) (g e_i)
        = \Phi_n \left(\, \idstrand \otimes \toppin \right) (g e_i \otimes g e_i)
        = g e_i
        = \Phi_n \left(\toppin \otimes \idstrand \, \right) \circ \Phi_n \left( \spliter \right) (g e_i).
    \end{gather*}
    To prove the third relation, we compute
    \begin{gather*}
        \Phi_n (\spliter) \circ \Phi_n (\merge) (g e_i \otimes h e_j)
        = \delta_{g,h} \delta_{i,j} \Phi_n (\spliter) (g e_i)
        = \delta_{g,h}\delta_{i,j} (g e_i\otimes g e_i),
        \\
        \Phi_n \left( \merge \otimes \idstrand\, \right) \circ \Phi_n \left(\, \idstrand \otimes \spliter \right) (g e_i \otimes h e_j)
        = \Phi_n \left( \merge\otimes \idstrand\, \right) (g e_i\otimes h e_j\otimes h e_j)
        = \delta_{g,h}\delta_{i,j} ( g e_i \otimes g e_i),
        \\
        \Phi_n \left(\, \idstrand \otimes \merge \right) \circ \Phi_n \left( \spliter \otimes \idstrand\, \right) (g e_i \otimes h e_j)
        = \Phi_n \left(\, \idstrand \otimes \merge \right) (g e_i \otimes g e_i \otimes h e_j)
        = \delta_{g,h} \delta_{i,j} ( g e_i\otimes g e_i ),
    \end{gather*}
    concluding that the three maps are identical.

    \medskip

    \noindent \emph{Relations \cref{GPC2}}:  The relations \cref{GPC2} are straightforward.
    \details{
        To prove the second relation, we see that the left-hand side is sent to the map given by
        \begin{multline*}
            \Phi_n \left( \crossing \otimes \idstrand \, \right) \circ \Phi_n \left(\, \idstrand \otimes \crossing \right) \circ \Phi_n \left( \crossing \otimes \idstrand \, \right) (v \otimes w \otimes z)
            \\
            = \Phi_n \left (\crossing \otimes \idstrand \, \right) \circ \Phi_n \left(\, \idstrand \otimes \crossing \right) (w\otimes v\otimes z)
            = \Phi_n \left( \crossing \otimes \idstrand\, \right) (w \otimes z \otimes v)
            = z \otimes w \otimes v
        \end{multline*}
        while the right-hand side is sent to the map given by
        \begin{multline*}
            \Phi_n \left(\, \idstrand \otimes \crossing \right) \circ \Phi_n \left( \crossing \otimes \idstrand \, \right) \circ \Phi_n \left(\, \idstrand \otimes \crossing \right) (v \otimes w \otimes z)
            \\
            = \Phi_n \left(\, \idstrand\, \otimes \crossing \right) \circ \Phi_n \left( \crossing \otimes \idstrand \, \right) (v \otimes z \otimes w)
            = \Phi_n \left(\, \idstrand \otimes \crossing \right) (z \otimes v \otimes w)
            = z \otimes w \otimes v.
        \end{multline*}
    }

    \medskip

    \noindent \emph{Relations \cref{GPC3}}:  To prove the first relation in \cref{GPC3}, we compute
    \[
        \Phi_n\left( \crossing  \right) \circ \Phi_n \left(\, \idstrand \otimes \bottompin \right) (g e_i)
        = \sum_{h\in G} \sum_{j=1}^n \Phi_n \left( \crossing \right) (g e_i \otimes h e_j)
        = \sum_{h\in G} \sum_{j=1}^n (h e_j\otimes g e_i),
    \]
    and
    \[
        \Phi_n \left( \bottompin \otimes \idstrand \, \right) (g e_i)
        = \sum_{h \in G} \sum_{j=1}^n (h e_j \otimes g e_i).
    \]
    Similarly, to prove the second relation, we calculate
    \[
        \Phi_n \left(\, \idstrand \otimes \toppin \right) \circ \Phi_n \left( \crossing  \right) (g e_i \otimes h e_j)
        = \Phi_n \left(\, \idstrand \otimes \toppin \right) (h e_j \otimes g e_i)
        = h e_j
        = \Phi_n \left( \toppin \otimes \idstrand \, \right) (g e_i \otimes h e_j).
    \]
    The proof of the last two equalities in \cref{GPC3} are similar; we only check the last one.  We have
    \begin{multline*}
        \Phi_n \left( \, \idstrand \otimes \crossing \right) \circ \Phi_n \left( \crossing \otimes \idstrand\, \right) \circ \Phi_n \left(\, \idstrand \otimes \spliter \right) (g e_i \otimes h e_j)
        = \Phi_n \left( \, \idstrand \otimes \crossing \right) \circ \Phi_n \left( \crossing \otimes \idstrand\, \right) (g e_i \otimes h e_j \otimes h e_j)
        \\
        = \Phi_n \left( \, \idstrand \otimes \crossing \right) (h e_j \otimes g e_i \otimes h e_j)
        = h e_j \otimes h e_j \otimes g e_i
    \end{multline*}
    and
    \[
        \Phi_n \left( \spliter \otimes \idstrand \, \right) \circ \Phi_n \left( \crossing \right) (g e_i \otimes h e_j)
        = \Phi_n \left( \spliter \otimes \idstrand \, \right) (h e_j \otimes g e_i)
        = h e_j \otimes h e_j \otimes g e_i.
    \]

    \medskip

    \noindent \emph{Relations \cref{GPC4}}:  To prove the first relation in \cref{GPC4}, we compute
    \[
        \Phi_n \left( \merge \right) \circ \Phi_n \left( \crossing \right) (g e_i \otimes h e_j)
        = \Phi_n \left(\merge\right)(h e_j \otimes g e_i)
        = \delta_{g,h} \delta_{i,j} g e_i
        = \Phi_n \left(\merge\right)(g e_i \otimes h e_j).
    \]
    For the second relation, we have
    \begin{multline*}
        \Phi_n \left(\merge\right) \circ \Phi_n \left( \tokstrand \otimes \tokstrand[h] \right) \circ \Phi_n \left(\spliter\right) (k e_i)
        = \Phi_n \left(\merge\right) \circ \Phi_n \left( \tokstrand \otimes \tokstrand[h] \right) (k e_i \otimes k e_i)
        \\
        = \Phi_n \left(\merge\right) (k g^{-1} e_i \otimes k h^{-1} e_i)
        =\delta_{g,h}(k g^{-1} e_i)
        = \delta_{g,h}\Phi_n(\tokstrand)(k e_i).
    \end{multline*}

    \medskip

    \noindent \emph{Relations \cref{GPC5}}: The relations \cref{GPC5} are straightforward to verify.
    \details{
        For the first relation, we have
        \[
            \Phi_n(\tokstrand) \circ \Phi_n(\tokstrand[h])(k e_i)
            = \Phi_n(\tokstrand)(k h^{-1} e_i) \\
            = k h^{-1} g^{-1} e_i
            = \Phi_n \left( \tokstrand[gh] \right) (k e_i).
        \]
        The second relation is clear.  For the third relation, we have:
        \begin{multline*}
            \Phi_n \left(\crossing\right) \circ \Phi_n \left( \tokstrand \otimes \idstrand \, \right) (h e_i \otimes k e_j)
            = \Phi_n \left(\crossing\right) (h g^{-1} e_i \otimes k e_j)
            = k e_j\otimes h g^{-1} e_i
            \\
            = \Phi_n \left( \, \idstrand \otimes \tokstrand \right) (k e_j \otimes h e_i)
            = \Phi_n \left( \, \idstrand \otimes \tokstrand \right) \circ \Phi_n \left(\crossing\right) (h e_i \otimes k e_j).
        \end{multline*}
        The fourth and fifth relations are clear.
    }
\end{proof}

If $(P,\bg)$ is a $G$-partition of type $\binom{l}{k}$, then $\Phi_n([P,\bg]) \in \Hom_{G_n}(V^k,V^l)$ is uniquely described by its matrix coefficients:
\begin{equation}
    \Phi_n([P,\bg])(h_k e_{i_k} \otimes \dotsb \otimes h_1 e_{i_1})
    = \sum_{\substack{h_{1'},\dotsc,h_{l'} \in G \\ 1 \le i_{1'},\dotsc,i_{l'} \le n}} M(P,\bg)_{h_1,\dotsc,h_k,i_1,\dotsc,i_k}^{h_{1'},\dotsc,h_{l'},i_{1'},\dotsc,i_{l'}} h_{l'} e_{i_{l'}} \otimes \dotsb \otimes h_{1'} e_{i_{1'}}.
\end{equation}
The matrix $M(P,\bg)$ depends only on the equivalence class $[P,\bg]$ of $(P,\bg)$.

\begin{prop} \label{forest}
    Suppose $(P,\bg)$ is a $G$-partition of type $\binom{l}{k}$.  Then $M(P,\bg)_{h_1,\dotsc,h_k,i_1,\dotsc,i_k}^{h_{1'},\dotsc,h_{l'},i_{1'},\dotsc,i_{l'}} = 1$ if $i_a = i_b$ and $h_a g_a^{-1} = h_b g_b^{-1}$ for all $a,b \in \PS_k^l$ in the same part of $P$.  Otherwise, $M(P,\bg)_{h_1,\dotsc,h_k,i_1,\dotsc,i_k}^{h_{1'},\dotsc,h_{l'},i_{1'},\dotsc,i_{l'}} = 0$.
\end{prop}

\begin{proof}
    This follows from a straightforward computation using the definition of $\Phi_n$ in \cref{hide} and the isomorphism described in \cref{twocats}, writing each component of $P$ as a composition of tokens, merges, splits, and crossings as in \cref{supernova}.
\end{proof}

Recall the basis $\BV$ for $V^{\otimes k}$ from \cref{basis}.  Given a $G$-partition $(P,\bg)$ of type $\binom{0}{k}$, let $O_{P,\bg}$ denote the set of all $h_k e_{i_k} \otimes \dotsb \otimes h_1 e_{i_1} \in \BV$ such that
\begin{itemize}
    \item $i_a = i_b$ if and only if $a,b$ are in the same part of the partition $P$, and
    \item $h_a g_a^{-1} = h_b g_b^{-1}$ for all $a,b$ in the same part of the partition $P$.
\end{itemize}
Then we have $O_{P,\bg} = O_{P,\bh}$ if and only if $(P,\bg) \sim (P,\bh)$, and so we can define $O_{[P,\bg]} := O_{P,\bg}$.  The $G_n$-orbits of $\BV$ are the $O_{[P,\bg]}$ with $[P,\bg] \colon \go^{\otimes k} \to \one$ having at most $n$ parts (i.e.\ $P$ is a partition of $\{1,\dotsc,k\}$ having at most $n$ parts).

For $[P,\bg] \colon \go^{\otimes k} \to \one$, let $f_{[P,\bg]} \colon V^{\otimes k} \to \kk$ be the $\kk$-linear map determined on the basis $\BV$ by
\[
    f_{[P,\bg]}(h_k e_{i_k} \otimes \dotsb \otimes h_1 e_{i_1})
    =
    \begin{cases}
        1 & \text{if } h_k e_{i_k} \otimes \dotsb \otimes h_1 e_{i_1} \in O_{[P,\bg]}, \\
        0 & \text{otherwise}.
    \end{cases}
\]
The following lemma is now immediate.

\begin{lem} \label{fan}
    The set of all $f_{[P,\bg]}$ with $[P,\bg] \colon \go^{\otimes k} \to \one$ having at most $n$ parts is a basis for $\Hom_{G_n}(V^{\otimes k}, \mathbf{1}_n)$.
\end{lem}

For partitions $P,Q$ of $\PS_k^l$, we write $Q \ge P$ when $Q$ is coarser than $P$.  Thus, $Q \ge P$ if and only if every part of $P$ is a subset of some part of $Q$.  It follows from \cref{forest} that
\begin{equation}
    \Phi_n([P,\bg]) = \sum_{Q \ge P} f_{[Q,\bg]}
    \quad \text{for all } [P,\bg] \colon \go^{\otimes k} \to \one.
\end{equation}
Thus, if we define a new basis $\{x_{[P,\bg]} : [P,\bg] \colon \go^{\otimes k} \to \go^{\otimes l}\}$ of $\Hom_{\Par(G,n|G|)}(\go^{\otimes k}, \go^{\otimes l})$ recursively by
\begin{equation}
    x_{[P,\bg]} = [P,\bg] - \sum_{Q > P} x_{[Q,\bg]},
\end{equation}
then a straightforward argument by induction shows that
\begin{equation} \label{ice} \textstyle
    \Phi_n(x_{[P,\bg]}) = f_{[P,\bg]}
    \quad \text{for any $G$-partition $(P,\bg)$ of type $\binom{0}{k}$}.
\end{equation}
In particular, $\Phi_n(x_{[P,\bg]}) = 0$ if $[P,\bg] \colon \go^{\otimes k} \to \one$ has more than $n$ parts.

\begin{theo} \label{kangaroo}
    \begin{enumerate}
        \item The functor $\Phi_n$ is full.

        \item The kernel of the induced map
            \[
                \Hom_{\Par(G,n|G|)}(\go^{\otimes k},\go^{\otimes l}) \to \Hom_{G_n}(V^{\otimes k}, V^{\otimes l})
            \]
            is the span of all $x_{[P,\bg]}$ with $[P,\bg] \colon \go^{\otimes k} \to \go^{\otimes l}$ having more than $n$ parts.  In particular, this map is an isomorphism if and only if $k + l \leq n$.
    \end{enumerate}
\end{theo}

\begin{proof}
    We have a commutative diagram
    \[
        \begin{tikzcd}
            \Hom_{\Par(G,n|G|)}(\go^{\otimes k}, \go^{\otimes l}) \arrow[r] \arrow[d, "\Phi_n"'] & \Hom_{\Par(G,n|G|)}(\go^{\otimes (k+l)},\one) \arrow[d, "\Phi_n"]
            \\
            \Hom_{G_n}(V^{\otimes k}, V^{\otimes l}) \arrow[r] &
            \Hom_{G_n}(V^{\otimes (k+l)}, \mathbf{1}_n)
        \end{tikzcd}
    \]
    where the horizontal maps are adjunction isomorphisms arising from the fact that $\go$ is a self-dual object in $\Par(G,n|G|)$ and that $V$ is a self-dual object in the category of $A_n$-modules.  (We refer the reader to the proof of \cite[Th.~2.3]{Com16} for more details of these adjunctions in the special case $G = \{1\}$.  The argument is the same in the case of general $G$.)  It follows from \cref{fan,ice} that the right-hand vertical map is surjective and its kernel is the span of all $x_{[P,\bg]}$ with $[P,\bg] \colon \go^{\otimes (k+l)} \to \one$ having more than $n$ parts.  Since the adjunction isomorphisms preserve the number of parts of $G$-partitions, as well as the partial order on $G$-partitions, the result follows.
\end{proof}

When $G = \{1\}$ is the trivial group, \cref{kangaroo} reduces to \cite[Th.~2.3]{Com16}.  In general, \Cref{kangaroo} is a categorical generalization of the double centralizer property \cite[Th.~6.6]{Blo03}.  More precisely, recall the $G$-partition algebras from \cref{fort}. The functor $\Phi_n$ induces an algebra homomorphism
\[
    P_k(G,n|G|) \to \End_{G_n}(V^{\otimes k}).
\]
\Cref{kangaroo} implies that this homomorphism is surjective, and is an isomorphism when $n \ge 2k$.  When the characteristic of $\kk$ does not divide $n!|G| = |G_n|$, so that $A_n$ is semisimple, the Double Centralizer Theorem implies that $A_n$ generates $\End_{P_k(G,n|G|)}(V^{\otimes k})$.  Hence $G_n$ and $P_k(G,n|G|)$ generate the centralizers of each other in $\End_\kk(V^{\otimes k})$.

\section{The group Heisenberg category\label{sec:HeisG}}

In this section we recall a special case of the Frobenius Heisenberg category.  We are interested in the special case of central charge $-1$, where this category was first defined in \cite{RS17}.  Furthermore, we will specialize to the case where the Frobenius algebra is the group algebra of a finite group $G$.  We follow the presentation in \cite{Sav19}, referring the reader to that paper for proofs of the statements made here.

\begin{defin}
    The \emph{group Heisenberg category} $\Heis(G)$ associated to the finite group $G$ is the strict $\kk$-linear monoidal category generated by two objects $\uparrow$, $\downarrow$, and morphisms
    \begin{gather*}
        \begin{tikzpicture}[anchorbase]
            \draw[->] (0.6,0) -- (0,0.6);
            \draw[->] (0,0) -- (0.6,0.6);
        \end{tikzpicture}
        \colon \uparrow \uparrow\ \to\ \uparrow \uparrow
        , \quad
        \begin{tikzpicture}[anchorbase]
            \draw[->] (0,-0.3) -- (0,0.3);
            \token{east}{0,0}{g};
        \end{tikzpicture}
        \colon \uparrow\ \to\ \uparrow
        \ ,\quad g \in G,
        \\
        \begin{tikzpicture}[anchorbase]
            \draw[->] (0,.2) -- (0,0) arc (180:360:.3) -- (.6,.2);
        \end{tikzpicture}
        \ \colon \one \to\ \downarrow \uparrow
        , \quad
        \begin{tikzpicture}[anchorbase]
            \draw[->] (0,-.2) -- (0,0) arc (180:0:.3) -- (.6,-.2);
        \end{tikzpicture}
        \ \colon \uparrow \downarrow\ \to \one
        , \quad
        \begin{tikzpicture}[anchorbase]
            \draw[<-] (0,.2) -- (0,0) arc (180:360:.3) -- (.6,.2);
        \end{tikzpicture}
        \ \colon \one \to\ \uparrow \downarrow
        , \quad
        \begin{tikzpicture}[anchorbase]
            \draw[<-] (0,0) -- (0,.2) arc (180:0:.3) -- (.6,0);
        \end{tikzpicture}
        \ \colon \downarrow \uparrow\ \to \one,
    \end{gather*}
    subject to the relations
    \begin{gather} \label{H1}
        \begin{tikzpicture}[anchorbase]
            \draw[->] (0.3,0) \braidto (-0.3,0.6) \braidto (0.3,1.2);
            \draw[->] (-0.3,0) to[out=up,in=down] (0.3,0.6) \braidto (-0.3,1.2);
        \end{tikzpicture}
        \ =\
        \begin{tikzpicture}[anchorbase]
            \draw[->] (-0.2,0) -- (-0.2,1.2);
            \draw[->] (0.2,0) -- (0.2,1.2);
        \end{tikzpicture}
        \ ,\quad
        \begin{tikzpicture}[anchorbase]
            \draw[->] (0.4,0) -- (-0.4,1.2);
            \draw[->] (0,0) \braidto (-0.4,0.6) \braidto (0,1.2);
            \draw[->] (-0.4,0) -- (0.4,1.2);
        \end{tikzpicture}
        \ =\
        \begin{tikzpicture}[anchorbase]
            \draw[->] (0.4,0) -- (-0.4,1.2);
            \draw[->] (0,0) \braidto (0.4,0.6) \braidto (0,1.2);
            \draw[->] (-0.4,0) -- (0.4,1.2);
        \end{tikzpicture}
        \ ,\quad
        \begin{tikzpicture}[anchorbase]
            \draw[->] (0,-0.4) -- (0,0.4);
            \token{east}{0,0.15}{g};
            \token{east}{0,-0.15}{h};
        \end{tikzpicture}
        \ =\
        \begin{tikzpicture}[anchorbase]
            \draw[->] (0,-0.4) -- (0,0.4);
            \token{west}{0,0}{gh};
        \end{tikzpicture}
        \ ,\quad
        \begin{tikzpicture}[anchorbase]
            \draw[->] (0,-0.4) -- (0,0.4);
            \token{west}{0,0}{1};
        \end{tikzpicture}
        \ =\
        \begin{tikzpicture}[anchorbase]
            \draw[->] (0,-0.4) -- (0,0.4);
        \end{tikzpicture}
        \ ,\quad
        \begin{tikzpicture}[anchorbase]
            \draw[->] (-0.3,-0.3) -- (0.3,0.3);
            \draw[->] (0.3,-0.3) -- (-0.3,0.3);
            \token{east}{-0.15,-0.15}{g};
        \end{tikzpicture}
        \ =\
        \begin{tikzpicture}[anchorbase]
            \draw[->] (-0.3,-0.3) -- (0.3,0.3);
            \draw[->] (0.3,-0.3) -- (-0.3,0.3);
            \token{north west}{0.125,0.125}{g};
        \end{tikzpicture}
        \ ,
        \\ \label{H2}
        \begin{tikzpicture}[anchorbase]
            \draw[->] (0,0) -- (0,0.6) arc(180:0:0.2) -- (0.4,0.4) arc(180:360:0.2) -- (0.8,1);
        \end{tikzpicture}
        \ =\
        \begin{tikzpicture}[anchorbase]
            \draw[->] (0,0) -- (0,1);
        \end{tikzpicture}
        \ ,\qquad
        \begin{tikzpicture}[anchorbase]
            \draw[->] (0,1) -- (0,0.4) arc(180:360:0.2) -- (0.4,0.6) arc(180:0:0.2) -- (0.8,0);
        \end{tikzpicture}
        \ =\
        \begin{tikzpicture}[anchorbase]
            \draw[<-] (0,0) -- (0,1);
        \end{tikzpicture}
        \ ,
        \\ \label{H3}
        \begin{tikzpicture}[anchorbase]
            \draw[->] (-0.3,-0.6) \braidto (0.3,0) \braidto (-0.3,0.6);
            \draw[<-] (0.3,-0.6) \braidto (-0.3,0) \braidto (0.3,0.6);
        \end{tikzpicture}
        \ =\
        \begin{tikzpicture}[anchorbase]
            \draw[->] (-0.2,-0.6) to (-0.2,0.6);
            \draw[<-] (0.2,-0.6) to (0.2,0.6);
        \end{tikzpicture}
        \ ,\quad
        \begin{tikzpicture}[anchorbase]
            \draw[<-] (-0.3,-0.6) \braidto (0.3,0) \braidto (-0.3,0.6);
            \draw[->] (0.3,-0.6) \braidto (-0.3,0) \braidto (0.3,0.6);
        \end{tikzpicture}
        \ =\
        \begin{tikzpicture}[anchorbase]
            \draw[<-] (-0.2,-0.6) to (-0.2,0.6);
            \draw[->] (0.2,-0.6) to (0.2,0.6);
        \end{tikzpicture}
        - \sum_{g \in G}
        \begin{tikzpicture}[anchorbase]
            \draw[<-] (-0.3,-0.6) to (-0.3,-0.4) arc(180:0:0.3) to (0.3,-0.6);
            \draw[->] (-0.3,0.6) to (-0.3,0.4) arc(-180:0:0.3) to (0.3,0.6);
            \token{west}{0.3,0.4}{g};
            \token{west}{0.3,-0.4}{g^{-1}};
        \end{tikzpicture}
        \ ,\quad
        \begin{tikzpicture}[anchorbase]
            \draw[<-] (0,0.6) to (0,0.3);
            \draw (-0.3,-0.2) to [out=180,in=-90](-.5,0);
            \draw (-0.5,0) to [out=90,in=180](-.3,0.2);
            \draw (-0.3,.2) to [out=0,in=90](0,-0.3);
            \draw (0,-0.3) to (0,-0.6);
            \draw (0,0.3) to [out=-90,in=0] (-.3,-0.2);
        \end{tikzpicture}
        \ = 0,
        \quad
        \begin{tikzpicture}[anchorbase]
            \draw[->] (0,0.3) arc(90:450:0.3);
            \token{west}{0.3,0}{g};
        \end{tikzpicture}
        = \delta_{g,1} 1_\one.
    \end{gather}
    Here the left and right crossings are defined by
    \[
        \begin{tikzpicture}[anchorbase]
            \draw[<-] (0,0) -- (0.6,0.6);
            \draw[->] (0.6,0) -- (0,0.6);
        \end{tikzpicture}
        \ :=\
        \begin{tikzpicture}[anchorbase]
            \draw[->] (-0.2,-0.3) to (0.2,0.3);
            \draw[<-] (-0.6,-0.3) to[out=up,in=135,looseness=2] (0,0) to[out=-45,in=down,looseness=2] (0.6,0.3);
        \end{tikzpicture}
        \ ,\qquad
        \begin{tikzpicture}[anchorbase]
            \draw[->] (0,0) -- (0.6,0.6);
            \draw[<-] (0.6,0) -- (0,0.6);
        \end{tikzpicture}
        \ :=\
        \begin{tikzpicture}[anchorbase]
            \draw[->] (0.2,-0.3) to (-0.2,0.3);
            \draw[<-] (0.6,-0.3) to[out=up,in=45,looseness=2] (0,0) to[out=225,in=down,looseness=2] (-0.6,0.3);
        \end{tikzpicture}
        \ .
    \]
\end{defin}

The objects $\uparrow$ and $\downarrow$ are both left and right dual to each other.  Furthermore, the cups and caps endow $\Heis(G)$ with the structure of a strict pivotal category, meaning that morphisms are invariant under isotopy.  We define downwards crossings and downward tokens by
\begin{equation} \label{yoga}
    \begin{tikzpicture}[anchorbase]
        \draw[<-] (0,0) -- (0.6,0.6);
        \draw[<-] (0.6,0) -- (0,0.6);
    \end{tikzpicture}
    \ :=\
    \begin{tikzpicture}[anchorbase]
        \draw[<-] (-0.2,-0.3) to (0.2,0.3);
        \draw[<-] (-0.6,-0.3) to[out=up,in=135,looseness=2] (0,0) to[out=-45,in=down,looseness=2] (0.6,0.3);
    \end{tikzpicture}
    \ =\
    \begin{tikzpicture}[anchorbase]
        \draw[<-] (0.2,-0.3) to (-0.2,0.3);
        \draw[<-] (0.6,-0.3) to[out=up,in=45,looseness=2] (0,0) to[out=225,in=down,looseness=2] (-0.6,0.3);
    \end{tikzpicture}
    \ ,\quad
    \begin{tikzpicture}[anchorbase]
        \draw[<-] (0,0) -- (0,1);
        \token{east}{0,0.5}{g};
    \end{tikzpicture}
    \ :=\
    \begin{tikzpicture}[anchorbase]
        \draw[<-] (0,0) -- (0,0.6) arc(180:0:0.2) -- (0.4,0.4) arc(180:360:0.2) -- (0.8,1);
        \token{east}{0.4,0.5}{g};
    \end{tikzpicture}
    \ =\
    \begin{tikzpicture}[anchorbase]
        \draw[->] (0,1) -- (0,0.4) arc(180:360:0.2) -- (0.4,0.6) arc(180:0:0.2) -- (0.8,0);
        \token{east}{0.4,0.5}{g};
    \end{tikzpicture}
    \ .
\end{equation}
It follows that,
\begin{equation} \label{genbraid}
  \begin{tikzpicture}[anchorbase]
    \draw (0.4,0) -- (-0.4,1.2);
    \draw (0,0) \braidto (-0.4,0.6) \braidto (0,1.2);
    \draw (-0.4,0) -- (0.4,1.2);
  \end{tikzpicture}
  \ =\
  \begin{tikzpicture}[anchorbase]
    \draw (0.4,0) -- (-0.4,1.2);
    \draw (0,0) \braidto (0.4,0.6) \braidto (0,1.2);
    \draw (-0.4,0) -- (0.4,1.2);
  \end{tikzpicture}
  \quad \text{for all possible orientations of the strands,}
\end{equation}
and we that tokens slide over cups and caps:
\begin{equation} \label{cct}
    \begin{tikzpicture}[anchorbase]
        \draw[->] (0,0) -- (0,-0.3) arc (180:360:.3) -- (0.6,0);
        \token{east}{0,-0.3}{g};
    \end{tikzpicture}
    \ =\
    \begin{tikzpicture}[anchorbase]
        \draw[->] (0,0) -- (0,-0.3) arc (180:360:.3) -- (0.6,0);
        \token{west}{0.6,-0.3}{g};
    \end{tikzpicture}
    \ , \quad
    \begin{tikzpicture}[anchorbase]
        \draw[->] (0,0) -- (0,0.3) arc (180:0:.3) -- (0.6,0);
        \token{east}{0,0.3}{g};
    \end{tikzpicture}
    \ = \
    \begin{tikzpicture}[anchorbase]
        \draw[->] (0,0) -- (0,0.3) arc (180:0:.3) -- (0.6,0);
        \token{west}{0.6,0.3}{g};
    \end{tikzpicture}
    \ ,\quad
    \begin{tikzpicture}[anchorbase]
        \draw[<-] (0,0) -- (0,0.3) arc (180:0:.3) -- (0.6,0);
        \token{east}{0.0,0.3}{g};
    \end{tikzpicture}
    \ = \
    \begin{tikzpicture}[anchorbase]
        \draw[<-] (0,0) -- (0,0.3) arc (180:0:.3) -- (0.6,0);
        \token{west}{0.6,0.3}{g};
    \end{tikzpicture}
    \ ,\quad
    \begin{tikzpicture}[anchorbase]
        \draw[<-] (0,0) -- (0,-0.3) arc (180:360:.3) -- (0.6,0);
        \token{west}{0.6,-0.3}{g};
    \end{tikzpicture}
    \ = \
    \begin{tikzpicture}[anchorbase]
        \draw[<-] (0,0) -- (0,-0.3) arc (180:360:.3) -- (0.6,0);
        \token{east}{0.0,-0.3}{g};
    \end{tikzpicture}
    \ .
\end{equation}
(One should compare this to the first two relations in \cref{ramp}.)  Because of this, we will sometimes place tokens at the critical point of cups and caps, since there is no ambiguity.  In addition it follows from \cref{cct,H1,H2} that
\begin{equation} \label{slush}
    \begin{tikzpicture}[anchorbase]
        \draw (-0.3,-0.3) -- (0.3,0.3);
        \draw (0.3,-0.3) -- (-0.3,0.3);
        \token{east}{-0.15,-0.15}{g};
    \end{tikzpicture}
    =
    \begin{tikzpicture}[anchorbase]
        \draw (-0.3,-0.3) -- (0.3,0.3);
        \draw (0.3,-0.3) -- (-0.3,0.3);
        \token{west}{0.15,0.15}{g};
    \end{tikzpicture}
    \ , \quad
    \begin{tikzpicture}[anchorbase]
        \draw (-0.3,-0.3) -- (0.3,0.3);
        \draw (0.3,-0.3) -- (-0.3,0.3);
        \token{west}{0.15,-0.15}{g};
    \end{tikzpicture}
    =
    \begin{tikzpicture}[anchorbase]
        \draw (-0.3,-0.3) -- (0.3,0.3);
        \draw (0.3,-0.3) -- (-0.3,0.3);
        \token{east}{-0.15,0.15}{g};
    \end{tikzpicture}
    \quad \text{for all possible orientations of the strands.}
\end{equation}
We will use \cref{cct,slush} frequently without mention.  Note how tokens multiply on downward strands (using \cref{yoga,cct,H1}):
\begin{equation}
    \begin{tikzpicture}[anchorbase]
        \draw[<-] (0,-0.4) -- (0,0.4);
        \token{east}{0,0.15}{g};
        \token{east}{0,-0.15}{h};
    \end{tikzpicture}
    \ =\
    \begin{tikzpicture}[anchorbase]
        \draw[<-] (0,-0.4) -- (0,0.4);
        \token{west}{0,0}{hg};
    \end{tikzpicture}
    \ .
\end{equation}

\section{The embedding functor\label{sec:embed}}

In this section we define an explicit embedding of the group partition category into the group Heisenberg category.  We assume throughout this section that $G$ is a finite group.

\begin{theo} \label{hitchcock}
    There is a faithful strict $\kk$-linear monoidal functor $\Psi \colon \Par(G) \to \Heis(G)$ defined on objects by $\go \mapsto\ \uparrow \otimes \downarrow$ and on generating morphisms by
    \begin{gather*}
        \merge
        \mapsto
        \begin{tikzpicture}[anchorbase]
            \draw[->] (0,0) \braidto (0.33,1);
            \draw[<-] (1,0) \braidto (0.67,1);
            \draw[<-] (0.33,0) to (0.33,0.1) to[out=up,in=up,looseness=2] (0.67,0.1) to (0.67,0);
        \end{tikzpicture}
        \ ,\qquad
        \spliter
        \mapsto
        \begin{tikzpicture}[anchorbase]
            \draw[->] (0.33,0) \braidto (0,1);
            \draw[<-] (0.67,0) \braidto (1,1);
            \draw[->] (0.33,1) to (0.33,0.9) to[out=down,in=down,looseness=2] (0.67,0.9) to (0.67,1);
        \end{tikzpicture}
        \ ,\qquad
        \crossing
        \mapsto
        \begin{tikzpicture}[anchorbase]
            \draw[->] (0,0) \braidto (1,1);
            \draw[<-] (0.5,0) \braidto (1.5,1);
            \draw[->] (1,0) \braidto (0,1);
            \draw[<-] (1.5,0) \braidto (0.5,1);
        \end{tikzpicture}
        + \sum_{g \in G}
        \begin{tikzpicture}[anchorbase]
            \draw[->] (-0.5,-0.6) -- (-0.5,0.6);
            \draw[->] (-0.3,0.6) -- (-0.3,0.5) arc(180:360:0.3) -- (0.3,0.6);
            \draw[<-] (-0.3,-0.6) -- (-0.3,-0.5) arc(180:0:0.3) -- (0.3,-0.6);
            \draw[<-] (0.5,-0.6) -- (0.5,0.6);
            \token{south}{0,0.2}{g};
            \token{north}{0,-0.2}{g};
            \token{east}{-0.5,0}{g^{-1}};
            \token{west}{0.5,0}{g^{-1}};
        \end{tikzpicture}
        \ ,
        \\
        \bottompin
        \mapsto
        \begin{tikzpicture}[anchorbase]
            \draw[<-] (0,1) -- (0,0.9) arc (180:360:.25) -- (0.5,1);
        \end{tikzpicture}
        \ ,\qquad
        \toppin
        \mapsto
        \begin{tikzpicture}[anchorbase]
            \draw[->] (0,0) -- (0,0.1) arc (180:0:.25) -- (0.5,0);
        \end{tikzpicture}
        \ ,\qquad
        \tokstrand
        \mapsto
        \begin{tikzpicture}[anchorbase]
            \draw[->] (-0.2,-0.2) -- (-0.2,0.2);
            \draw[<-] (0.2,-0.2) -- (0.2,0.2);
            \token{east}{-0.2,0}{g};
            \token{west}{0.2,0}{g^{-1}};
        \end{tikzpicture}
        \ , \quad g \in G.
  \end{gather*}
\end{theo}

The proof of \cref{hitchcock} will occupy the remainder of this section.  We break the proof into two parts, first showing that $\Psi$ is well defined, and then that it is faithful.

\begin{prop} \label{poker}
    The functor $\Psi$ is well defined.
\end{prop}

\begin{proof}
    It suffices to show that the images of the generating morphisms of the partition category $\Par(G)$ satisfy relations \cref{GPC1,GPC2,GPC3,GPC4,GPC5}.

    \medskip

    \noindent \emph{Relations \cref{GPC1}}: These relations are easy to check, using isotopy invariance in $\Heis(G)$.

    \medskip

    \noindent \emph{Relations \cref{GPC2}}:  Using the third relation in \cref{H3}, which says that left curls are equal to zero, we have
    \begin{multline*}
        \Psi \left( \crossing \right) \circ \Psi \left( \crossing \right)
        =
        \begin{tikzpicture}[anchorbase]
            \draw[->] (0,0) \braidto (1,1) \braidto (0,2);
            \draw[<-] (0.5,0) \braidto (1.5,1) \braidto (0.5,2);
            \draw[->] (1,0) \braidto (0,1) \braidto (1,2);
            \draw[<-] (1.5,0) \braidto (0.5,1) \braidto (1.5,2);
        \end{tikzpicture}
        \ + \sum_{g,h \in G}
        \begin{tikzpicture}[anchorbase]
            \draw[->] (-0.8,-1) to (-0.8,1);
            \draw[<-] (0.8,-1) to (0.8,1);
            \draw[<-] (-0.3,-1) to (-0.3,-0.9) arc(180:0:0.3) to (0.3,-1);
            \draw[->] (-0.3,1) to (-0.3,0.9) arc(180:360:0.3) to (0.3,1);
            \draw[->] (0.3,0) arc(0:360:0.3);
            \token{east}{-0.8,0}{h^{-1}g^{-1}};
            \token{west}{0.8,0}{g^{-1} h^{-1}};
            \token{south}{0,0.6}{h};
            \token{north}{0,-0.6}{g};
            \token{west}{-0.3,0}{hg};
        \end{tikzpicture}
        \\
        \overset{\cref{H3}}{=}
        \begin{tikzpicture}[anchorbase]
            \draw[->] (0,0) to (0,1.5);
            \draw[<-] (0.5,0) \braidto (1,0.75) \braidto (0.5,1.5);
            \draw[->] (1,0) \braidto (0.5,0.75) \braidto (1,1.5);
            \draw[<-] (1.5,0) to (1.5,1.5);
        \end{tikzpicture}
        \ + \
        \sum_{g \in G}\
        \begin{tikzpicture}[anchorbase]
            \draw[->] (-0.8,-0.75) to (-0.8,0.75);
            \draw[<-] (0.8,-0.75) to (0.8,0.75);
            \draw[<-] (-0.3,-0.75) to (-0.3,-0.65) arc(180:0:0.3) to (0.3,-0.75);
            \draw[->] (-0.3,0.75) to (-0.3,0.65) arc(180:360:0.3) to (0.3,0.75);
            \token{south}{0,0.35}{g};
            \token{south}{0,-0.35}{g^{-1}};
        \end{tikzpicture}
        \overset{\cref{H3}}{=}
        \begin{tikzpicture}[anchorbase]
            \draw[->] (0,0) to (0,1.5);
            \draw[<-] (0.5,0) to (0.5,1.5);
            \draw[->] (1,0) to (1,1.5);
            \draw[<-] (1.5,0) to (1.5,1.5);
        \end{tikzpicture}
        = \Psi \left( \, \idstrand \ \idstrand \, \right),
    \end{multline*}
    proving the first relation in \cref{GPC2}.  To prove the second relation in \cref{GPC2}, we compute
    \begin{multline} \label{eggs}
        \Psi \left( \, \idstrand \otimes \crossing \right) \circ \Psi \left( \crossing \otimes \idstrand \, \right)
        =
        \begin{tikzpicture}[anchorbase]
            \draw[->] (0,0) to[out=up,in=240] (2,1);
            \draw[<-] (0.5,0) to[out=60,in=down] (2.5,1);
            \draw[->] (1,0) \braidto (0,1);
            \draw[<-] (1.5,0) \braidto (0.5,1);
            \draw[->] (2,0) \braidto (1,1);
            \draw[<-] (2.5,0) \braidto (1.5,1);
        \end{tikzpicture}
        \ + \sum_{g\in G}
        \begin{tikzpicture}[anchorbase]
            \draw[->] (0,0) to (0,1);
            \draw[<-] (0.5,0) to (0.5,0.1) arc(180:0:0.25) to (1,0);
            \draw[<-] (1.5,0) \braidto (2.5,0.8) to (2.5,1);
            \draw[->] (2,0) \braidto (1,1);
            \draw[<-] (2.5,0) \braidto (1.5,1);
            \draw[->] (0.5,1) to[out=down,in=down] (2,1);
            \token{east}{0,0.5}{g^{-1}};
            \token{north}{0.75,0.35}{g};
            \token{east}{0.57,0.8}{g};
            \token{west}{2.5,0.8}{g^{-1}};
        \end{tikzpicture}
        \\
        + \sum_{g \in G}\
        \begin{tikzpicture}[anchorbase]
            \draw[->] (0,0) -- (0,0.2) \braidto (1,1);
            \draw[<-] (0.5,0) to[out=up,in=up] (2,0);
            \draw[->] (1,0) \braidto (0,1);
            \draw[<-] (1.5,0) \braidto (0.5,1);
            \draw[<-] (2.5,0) -- (2.5,1);
            \draw[->] (1.5,1) -- (1.5,0.9) arc(180:360:0.25) -- (2,1);
            \token{east}{0,0.2}{g^{-1}};
            \token{west}{1.93,0.2}{g};
            \token{south}{1.75,0.65}{g};
            \token{west}{2.5,0.5}{g^{-1}};
        \end{tikzpicture}
        \ + \sum_{g,h \in G}
        \begin{tikzpicture}[anchorbase]
            \draw[->] (0,0) to (0,1);
            \draw[<-] (0.5,0) to (0.5,0.1) arc(180:0:0.25) to (1,0);
            \draw[<-] (1.5,0) to (1.5,0.1) arc(180:0:0.25) to (2,0);
            \draw[<-] (2.5,0) to (2.5,1);
            \draw[->] (0.5,1) to (0.5,0.9) arc(180:360:0.25) to (1,1);
            \draw[->] (1.5,1) to (1.5,0.9) arc(180:360:0.25) to (2,1);
            \token{east}{0,0.5}{g^{-1}};
            \token{south}{0.5,0.9}{h^{-1}g};
            \token{north}{0.75,0.35}{g};
            \token{north}{2,0.1}{g^{-1}h};
            \token{south}{1.75,0.65}{h};
            \token{west}{2.5,0.5}{h^{-1}};
        \end{tikzpicture}
        \ .
    \end{multline}
    Therefore, using the fact that left curls are zero, we compute
    \begin{align*}
        \Psi &\left( \, \idstrand \otimes \crossing \right) \circ \Psi \left( \crossing \otimes \idstrand \, \right) \circ \Psi \left( \, \idstrand \otimes \crossing \right)
        \\
        &\underset{\cref{slush}}{\overset{\cref{H3}}{=}}
        \begin{tikzpicture}[anchorbase]
            \draw[->] (0,-0.75) to[out=up,in=240] (2,0.75);
            \draw[<-] (0.5,-0.75) to[out=60,in=down] (2.5,0.75);
            \draw[->] (1,-0.75) \braidto (2.1,0) \braidto (1,0.75);
            \draw[<-] (1.5,-0.75) to[out=60,in=down] (2.5,0) to[out=up,in=-60] (1.5,0.75);
            \draw[->] (2,-0.75) to[out=120,in=down] (0,0.75);
            \draw[<-] (2.5,-0.75) to[out=up,in=-60] (0.5,0.75);
        \end{tikzpicture}
        + \sum_{g\in G}
        \begin{tikzpicture}[anchorbase]
            \draw[->] (0,-0.75) to (0,0.75);
            \draw[<-] (0.5,-0.75) to[out=up,in=up,looseness=0.8] (2,-0.75);
            \draw[->] (1,-0.75) \braidto (1.75,0) \braidto (1,0.75);
            \draw[<-] (1.5,-0.75) \braidto (2.25,0) \braidto (1.5,0.75);
            \draw[<-] (2.5,-0.75) to[out=up,in=down,looseness=0.8] (1.25,0) to[out=up,in=down,looseness=0.8] (2.5,0.75);
            \draw[->] (0.5,0.75) to[out=down,in=down,looseness=0.8] (2,0.75);
            \token{east}{0,0}{g^{-1}};
            \token{south}{0.8,-0.45}{g};
            \token{north}{0.8,0.45}{g};
            \token{west}{2.4,0.57}{g^{-1}};
        \end{tikzpicture}
        + \sum_{g \in G}
        \begin{tikzpicture}[anchorbase]
            \draw[->] (0,-0.75) to (0,-0.55) \braidto (1,0.75);
            \draw[<-] (1.5,-0.75) \braidto (2.5,0.55) to (2.5,0.75);
            \draw[->] (2,-0.75) \braidto (0,0.75);
            \draw[<-] (2.5,-0.75) \braidto (0.5,0.75);
            \draw[->] (1.5,0.75) to (1.5,0.65) to[out=down,in=down,looseness=1.5] (2,0.65) to (2,0.75);
            \draw[->] (1,-0.75) \braidto (1.7,0.1) to[out=up,in=up,looseness=1.5] (1.2,0.1) to[out=down,in=up] (0.5,-0.75);
            \token{east}{0,-0.55}{g^{-1}};
            \token{south}{0.65,-0.45}{g};
            \token{south}{1.75,0.45}{g};
            \token{west}{2.5,0.55}{g^{-1}};
        \end{tikzpicture}
        \\
        &\qquad\qquad\qquad\qquad + \sum_{g \in G}
        \begin{tikzpicture}[anchorbase]
            \draw[->] (0,-0.75) to[out=up,in=240] (2,0.75);
            \draw[<-] (0.5,-0.75) to[out=60,in=down] (2.5,0.75);
            \draw[->] (1,-0.75) \braidto (0,0.75);
            \draw[<-] (1.5,-0.75) to (1.5,-0.65) arc(180:0:0.25) to (2,-0.75);
            \draw[<-] (2.5,-0.75) \braidto (1.5,0.75);
            \draw[->] (0.5,0.75) to (0.5,0.65) arc(180:360:0.25) to (1,0.75);
            \token{east}{0.25,0.2}{g^{-1}};
            \token{north}{1.75,-0.4}{g};
            \token{south}{0.75,0.4}{g};
            \token{west}{2.25,-0.2}{g^{-1}};
        \end{tikzpicture}
        + \sum_{g,h,t\in G}
        \begin{tikzpicture}[anchorbase]
            \draw[->] (0,-0.75) to (0,0.75);
            \draw[<-] (0.5,-0.75) to (0.5,-0.6) arc(180:0:0.25) to (1,-0.75);
            \draw[<-] (1.5,-0.75) to (1.5,-0.6) arc(180:0:0.25) to (2,-0.75);
            \draw[->] (0.5,0.75) to (0.5,0.6) arc(180:360:0.25) to (1,0.75);
            \draw[->] (1.5,0.75) to (1.5,0.6) arc(180:360:0.25) to (2,0.75);
            \draw[<-] (2.5,-0.75) to (2.5,0.75);
            \draw[->] (1,0.2) arc(90:450:0.2);
            \token{east}{0,0}{g^{-1}};
            \token{south}{0.5,0.65}{h^{-1}g};
            \token{north}{1,-0.6}{gt^{-1}};
            \token{west}{1.2,0}{g^{-1}ht};
            \token{north}{1.75,-0.35}{t};
            \token{south}{1.75,0.35}{h};
            \token{west}{2.5,0}{(ht)^{-1}};
        \end{tikzpicture}
        \\
        &\underset{\cref{H3}}{\overset{\cref{H1}}{=}}
        \begin{tikzpicture}[anchorbase]
            \draw[->] (0,-0.75) to[out=up,in=240] (2,0.75);
            \draw[<-] (0.5,-0.75) to[out=60,in=down] (2.5,0.75);
            \draw[->] (1,-0.75) \braidto (2.1,0) \braidto (1,0.75);
            \draw[<-] (1.5,-0.75) to[out=60,in=down] (2.5,0) to[out=up,in=-60] (1.5,0.75);
            \draw[->] (2,-0.75) to[out=120,in=down] (0,0.75);
            \draw[<-] (2.5,-0.75) to[out=up,in=-60] (0.5,0.75);
        \end{tikzpicture}
        + \sum_{g \in G}
        \begin{tikzpicture}[anchorbase]
            \draw[->] (0,-0.75) to (0,0.75);
            \draw[->] (1,-0.75) to (1,0.75);
            \draw[<-] (1.5,-0.75) to (1.5,0.75);
            \draw[<-] (2.5,-0.75) to (2.5,0.75);
            \draw[<-] (0.5,-0.75) to[out=up,in=up] (2,-0.75);
            \draw[->] (0.5,0.75) to[out=down,in=down] (2,0.75);
            \token{west}{0,0}{g^{-1}};
            \token{north}{0.73,-0.44}{g};
            \token{south}{0.73,0.44}{g};
            \token{east}{2.5,0}{g^{-1}};
        \end{tikzpicture}
        + \sum_{g \in G}
        \begin{tikzpicture}[anchorbase]
            \draw[->] (0,-0.75) to (0,-0.55) \braidto (1,0.75);
            \draw[<-] (1.5,-0.75) \braidto (2.5,0.55) to (2.5,0.75);
            \draw[->] (2,-0.75) \braidto (0,0.75);
            \draw[<-] (2.5,-0.75) \braidto (0.5,0.75);
            \draw[->] (1.5,0.75) to (1.5,0.7) arc(180:360:0.25) to (2,0.75);
            \draw[->] (1,-0.75) to (1,-0.7) arc(0:180:0.25) to (0.5,-0.75);
            \token{east}{0,-0.55}{g^{-1}};
            \token{south}{0.75,-0.45}{g};
            \token{north}{1.75,0.45}{g};
            \token{west}{2.5,0.55}{g^{-1}};
        \end{tikzpicture}
        \\
        & \qquad\qquad\qquad\qquad + \sum_{g \in G}
        \begin{tikzpicture}[anchorbase]
            \draw[->] (0,-0.75) to[out=up,in=240] (2,0.75);
            \draw[<-] (0.5,-0.75) to[out=60,in=down] (2.5,0.75);
            \draw[->] (1,-0.75) \braidto (0,0.75);
            \draw[<-] (1.5,-0.75) to (1.5,-0.65) arc(180:0:0.25) to (2,-0.75);
            \draw[<-] (2.5,-0.75) \braidto (1.5,0.75);
            \draw[->] (0.5,0.75) to (0.5,0.65) arc(180:360:0.25) to (1,0.75);
            \token{east}{0.25,0.2}{g^{-1}};
            \token{north}{1.75,-0.4}{g};
            \token{south}{0.75,0.4}{g};
            \token{west}{2.25,-0.2}{g^{-1}};
        \end{tikzpicture}
        + \sum_{g,h \in G}
        \begin{tikzpicture}[anchorbase]
            \draw[->] (0,-0.75) to (0,0.75);
            \draw[<-] (0.5,-0.75) to (0.5,-0.6) arc(180:0:0.25) to (1,-0.75);
            \draw[<-] (1.5,-0.75) to (1.5,-0.6) arc(180:0:0.25) to (2,-0.75);
            \draw[->] (0.5,0.75) to (0.5,0.6) arc(180:360:0.25) to (1,0.75);
            \draw[->] (1.5,0.75) to (1.5,0.6) arc(180:360:0.25) to (2,0.75);
            \draw[<-] (2.5,-0.75) to (2.5,0.75);
            \token{east}{0,0}{g^{-1}};
            \token{north}{0.75,0.35}{h^{-1}g};
            \token{north}{0.75,-0.35}{h};
            \token{south}{1.75,-0.35}{h^{-1}g};
            \token{south}{1.75,0.35}{h};
            \token{west}{2.5,0}{g^{-1}};
        \end{tikzpicture}
        \ .
    \end{align*}
    Similarly,
    \begin{align*}
        \Psi &\left( \, \crossing \otimes \idstrand \, \right) \circ \Psi_{n} \left( \, \idstrand \otimes \crossing \, \right) \circ \Psi \left( \, \crossing  \otimes \idstrand \, \right)
        \\
        &=
        \begin{tikzpicture}[anchorbase]
            \draw[->] (0,-0.75) to[out=up,in=240] (2,0.75);
            \draw[<-] (0.5,-0.75) to[out=60,in=down] (2.5,0.75);
            \draw[->] (1,-0.75) to[out=120,in=down] (0,0) to[out=up,in=240] (1,0.75);
            \draw[<-] (1.5,-0.75) \braidto (0.4,0) \braidto (1.5,0.75);
            \draw[->] (2,-0.75) to[out=120,in=down] (0,0.75);
            \draw[<-] (2.5,-0.75) to[out=up,in=-60] (0.5,0.75);
        \end{tikzpicture}
        + \sum_{g\in G}
        \begin{tikzpicture}[anchorbase]
            \draw[->] (0,-0.75) to (0,0.75);
            \draw[->] (1,-0.75) to (1,0.75);
            \draw[<-] (1.5,-0.75) to (1.5,0.75);
            \draw[<-] (2.5,-0.75) to (2.5,0.75);
            \draw[<-] (0.5,-0.75) to[out=up,in=up] (2,-0.75);
            \draw[->] (0.5,0.75) to[out=down,in=down] (2,0.75);
            \token{west}{0,0}{g^{-1}};
            \token{north}{0.73,-0.44}{g};
            \token{south}{0.73,0.44}{g};
            \token{east}{2.5,0}{g^{-1}};
        \end{tikzpicture}
        + \sum_{g \in G}
        \begin{tikzpicture}[anchorbase]
            \draw[->] (0,-0.75) to (0,-0.55) \braidto (1,0.75);
            \draw[<-] (1.5,-0.75) \braidto (2.5,0.55) to (2.5,0.75);
            \draw[->] (2,-0.75) \braidto (0,0.75);
            \draw[<-] (2.5,-0.75) \braidto (0.5,0.75);
            \draw[->] (1.5,0.75) to (1.5,0.7) arc(180:360:0.25) to (2,0.75);
            \draw[->] (1,-0.75) to (1,-0.7) arc(0:180:0.25) to (0.5,-0.75);
            \token{east}{0,-0.55}{g^{-1}};
            \token{south}{0.75,-0.45}{g};
            \token{north}{1.75,0.45}{g};
            \token{west}{2.5,0.55}{g^{-1}};
        \end{tikzpicture}
        \\
        & \qquad\qquad\qquad\qquad + \sum_{g \in G}
        \begin{tikzpicture}[anchorbase]
            \draw[->] (0,-0.75) to[out=up,in=240] (2,0.75);
            \draw[<-] (0.5,-0.75) to[out=60,in=down] (2.5,0.75);
            \draw[->] (1,-0.75) \braidto (0,0.75);
            \draw[<-] (1.5,-0.75) to (1.5,-0.65) arc(180:0:0.25) to (2,-0.75);
            \draw[<-] (2.5,-0.75) \braidto (1.5,0.75);
            \draw[->] (0.5,0.75) to (0.5,0.65) arc(180:360:0.25) to (1,0.75);
            \token{east}{0.25,0.2}{g^{-1}};
            \token{north}{1.75,-0.4}{g};
            \token{south}{0.75,0.4}{g};
            \token{west}{2.25,-0.2}{g^{-1}};
        \end{tikzpicture}
        + \sum_{g,h\in G}
        \begin{tikzpicture}[anchorbase]
            \draw[->] (0,-0.75) to (0,0.75);
            \draw[<-] (0.5,-0.75) to (0.5,-0.6) arc(180:0:0.25) to (1,-0.75);
            \draw[<-] (1.5,-0.75) to (1.5,-0.6) arc(180:0:0.25) to (2,-0.75);
            \draw[->] (0.5,0.75) to (0.5,0.6) arc(180:360:0.25) to (1,0.75);
            \draw[->] (1.5,0.75) to (1.5,0.6) arc(180:360:0.25) to (2,0.75);
            \draw[<-] (2.5,-0.75) to (2.5,0.75);
            \token{east}{0,0}{g^{-1}};
            \token{north}{0.75,0.35}{h^{-1}g};
            \token{north}{0.75,-0.35}{h};
            \token{south}{1.75,-0.35}{h^{-1}g};
            \token{south}{1.75,0.35}{h};
            \token{west}{2.5,0}{g^{-1}};
        \end{tikzpicture}
        \ .
    \end{align*}
    We then use \cref{genbraid} to see that the expressions are equal.

    \medskip

    \noindent \emph{Relations \cref{GPC3}}:  We will check the first and the fourth relations, since the proofs of the second and third are analogous.  For the first relation we compute:
    \[
        \Psi \left( \crossing \right) \circ \Psi \left( \, \idstrand \otimes \bottompin \right)
        \overset{\cref{H1}}{=}\
        \begin{tikzpicture}[anchorbase]
            \draw[->] (0,-0.5) \braidto (1,0.5);
            \draw[<-] (0.5,-0.5) \braidto (1.5,0.5);
            \draw[<-] (0,0.5) to[out=down,in=up] (1,-0.2) to[out=down,in=down,looseness=1.5] (1.4,-0.2) to[out=up,in=-60] (0.5,0.5);
        \end{tikzpicture}
        \ + \sum_{g \in G}
        \begin{tikzpicture}[anchorbase]
            \draw[->] (0.5,-0.5) to[out=up,in=-45] (0.25,0) to[out=135,in=down] (0,0.5);
            \draw[->] (0.5,0.5) to (0.5,0.4) arc(180:360:0.25) to (1,0.5);
            \draw[<-] (1,-0.5) to[out=up,in=225] (1.25,0) to[out=45,in=down] (1.5,0.5);
            \token{east}{0.25,0}{g^{-1}};
            \token{south}{0.75,0.15}{g};
        \end{tikzpicture}
        \ \overset{\mathclap{\cref{H1}}}{\underset{\mathclap{\cref{H3}}}{=}}\
        \begin{tikzpicture}[anchorbase]
            \draw[<-] (0,0.5) to (0,0.3) arc(180:360:0.25) to (0.5,0.5);
            \draw[->] (1,-0.5) to (1,0.5);
            \draw[<-] (1.5,-0.5) to (1.5,0.5);
        \end{tikzpicture}
        = \Psi \left( \bottompin \otimes \idstrand \, \right).
    \]
    For the fourth relation, we use \cref{eggs} and the fact that left curls are zero to compute
    \begin{multline*}
        \Psi \left( \, \idstrand \otimes \crossing \right) \circ \Psi \left( \crossing \otimes \idstrand \, \right) \circ \Psi_{n} \left( \, \idstrand \otimes \spliter \right)
        =
        \begin{tikzpicture}[anchorbase]
            \draw[->] (0,0) to[out=up,in=240] (2,1);
            \draw[-] (0.5,0) to[out=60,in=down] (2.5,1);
            \draw[->] (1,0) \braidto (0,1);
            \draw[-] (1.5,0) \braidto (0.5,1);
            \draw[->] (2,0) \braidto (1,1);
            \draw[-] (2.5,0) \braidto (1.5,1);
            \draw[<-] (2,-1) \braidto (2.5,0);
            \draw[-] (1.5,-1) \braidto (1,0);
            \draw[-] (0,-1) \braidto (0,0);
            \draw[<-] (0.5,-1) \braidto (0.5,0);
            \draw[-] (1.5,0) to (1.5,-0.1) to[out=down,in=down,looseness=1.5] (2,-0.1) to (2,0);
        \end{tikzpicture}
        + \sum_{g,h \in G}
        \begin{tikzpicture}[anchorbase]
            \draw[->] (0,-1) -- (0,1);
            \draw[<-] (0.5,-1) -- (0.5,-0.8) arc(180:0:0.25) -- (1,-1);
            \draw[->] (0.5,1) -- (0.5,0.8) arc(180:360:0.25) -- (1,1);
            \draw[->] (1.5,1) -- (1.5,0.8) arc(180:360:0.25) -- (2,1);
            \draw[->] (1.75,0.25) arc(90:450:0.25);
            \draw[<-] (2.5,-1) -- (2.5,1);
            \token{east}{0,0}{g^{-1}};
            \token{west}{2.5,0}{h^{-1}};
            \token{north}{0.75,0.55}{h^{-1}g};
            \token{south}{1.75,0.55}{h};
            \token{north}{0.75,-0.55}{g};
            \token{north}{1.75,-0.25}{g^{-1}h};
        \end{tikzpicture}
        \\
        \overset{\mathclap{\cref{H1}}}{\underset{\mathclap{\cref{H3}}}{=}}\
        \begin{tikzpicture}[anchorbase]
            \draw[->] (0,0) to[out=up,in=240] (2,1);
            \draw[<-] (0.5,0) to[out=60,in=down] (2.5,1);
            \draw[->] (1,0) \braidto (0,1);
            \draw[<-] (2.5,0) \braidto (1.5,1);
            \draw[->] (0.5,1) to (0.5,0.9) to[out=down,in=down,looseness=1.5] (1,0.9) to (1,1);
        \end{tikzpicture}
        \ + \sum_{g\in G}
        \begin{tikzpicture}[anchorbase]
            \draw[->] (0,-0.5) to (0,0.5);
            \draw[<-] (1,-0.5) to (1,-0.4) arc(180:0:0.25) to (1.5,-0.5);
            \draw[->] (0.5,0.5) to (0.5,0.4) arc(180:360:0.25) to (1,0.5);
            \draw[->] (1.5,0.5) to (1.5,0.4) arc(180:360:0.25) to (2,0.5);
            \draw[<-] (2.5,-0.5) to (2.5,0.5);
            \token{east}{0,0}{g^{-1}};
            \token{north}{1.25,0.-0.15}{g};
            \token{south}{1.75,0.15}{g};
            \token{west}{2.5,0}{g^{-1}};
        \end{tikzpicture}
        =
        \Psi \left( \spliter \ \idstrand \, \right) \circ \Psi \left( \crossing \right).
    \end{multline*}

    \medskip

    \noindent \emph{Relations \cref{GPC4}}:  To check the first relation in \cref{GPC4}, we use the fact that left curls are zero to see that
    \[
        \Phi_n \left( \merge \right) \circ \Phi_n \left( \crossing \right)
        = \sum_{g \in G}
        \begin{tikzpicture}[anchorbase]
            \draw[->] (0,-0.5) to (0,0.5);
            \draw[<-] (0.4,-0.5) to (0.4,-0.4) arc(180:0:0.2) to (0.8,-0.5);
            \draw[->] (0.8,0.2) arc(0:360:0.2);
            \draw[<-] (1.2,-0.5) to (1.2,0.5);
            \token{east}{0,0}{g^{-1}};
            \token{west}{1.2,0}{g^{-1}};
            \token{north}{0.6,-0.2}{g};
            \token{east}{0.4,0.2}{g};
        \end{tikzpicture}
        \overset{\cref{H3}}{=}
        \begin{tikzpicture}[anchorbase]
            \draw[->] (0,1) \braidto (0.5,2);
            \draw[<-] (1.5,1) \braidto (1,2);
            \draw[<-] (0.5,1) to (0.5,1.1) to[out=up,in=up,looseness=2] (1,1.1) to (1,1);
        \end{tikzpicture}
        = \Phi_n \left( \merge \right).
    \]
    For the second relation in \cref{GPC4}, we compute
    \[
        \Phi_n \left( \merge \right) \circ \Phi_n
        \left(
            \begin{tikzpicture}[anchorbase]
                \draw (-0.2,-0.25) -- (-0.2,0.25);
                \draw (0.2,-0.25) -- (0.2,0.25);
                \token{east}{-0.2,0}{g};
                \token{west}{0.2,0}{h};
            \end{tikzpicture}
        \right)
        \circ \Phi_n \left( \spliter \right)
        \overset{\cref{H1}}{=}
        \begin{tikzpicture}[anchorbase]
            \draw[->] (0,-0.2) arc(-90:270:0.2);
            \draw[->] (-0.5,-0.7) -- (-0.5,0.7);
            \draw[<-] (0.5,-0.7) -- (0.5,0.7);
            \token{east}{-0.5,0}{g};
            \token{west}{0.5,0}{h^{-1}};
            \token{south}{0,0.2}{g^{-1}h};
        \end{tikzpicture}
        \overset{\cref{H3}}{=} \delta_{g,h}
        \begin{tikzpicture}[anchorbase]
            \draw[->] (-0.2,-0.7) -- (-0.2,0.7);
            \draw[<-] (0.2,-0.7) -- (0.2,0.7);
            \token{east}{-0.2,0}{g};
            \token{west}{0.2,0}{g^{-1}};
        \end{tikzpicture}
        = \delta_{g,h} \Phi_n \left( \tokstrand \right).
    \]

    \medskip

    \noindent \emph{Relations \cref{GPC5}}: For the third relation in \cref{GPC5}, we compute
    \begin{multline*}
        \Psi \left( \crossing \right) \circ \Psi
        \left(
            \begin{tikzpicture}[anchorbase]
                \draw (-0.2,-0.25) -- (-0.2,0.25);
                \draw (0.2,-0.25) -- (0.2,0.25);
                \token{east}{-0.2,0}{g};
            \end{tikzpicture}\,
        \right)
        =
        \begin{tikzpicture}[anchorbase]
            \draw[->] (0,-0.5) -- (0,-0.3)  \braidto (1,0.5);
            \draw[<-] (0.5,-0.5) \braidto (1.5,0.3) -- (1.5,0.5);
            \draw[->] (1,-0.5) \braidto (0,0.5);
            \draw[<-] (1.5,-0.5) \braidto (0.5,0.5);
            \token{east}{0,-0.3}{g};
            \token{west}{1.5,0.3}{g^{-1}};
        \end{tikzpicture}
        + \sum_{h \in G}
        \begin{tikzpicture}[anchorbase]
            \draw[->] (0,-0.7) -- (0,0.7);
            \draw[<-] (0.5,-0.7) -- (0.5,-0.6)  arc(180:0:0.25) -- (1,-0.7);;
            \draw[->] (0.5,0.7) -- (0.5,0.6) arc(180:360:0.25) -- (1,0.7);
            \draw[<-] (1.5,-0.7) -- (1.5,0.7);
            \token{east}{0,0}{h^{-1}g};
            \token{west}{1.5,0}{h^{-1}};
            \token{south}{0.75,-0.35}{g^{-1}h};
            \token{south}{0.75,0.35}{h};
        \end{tikzpicture}
        \\
        =
        \begin{tikzpicture}[anchorbase]
            \draw[->] (0,-0.5) -- (0,-0.3)  \braidto (1,0.5);
            \draw[<-] (0.5,-0.5) \braidto (1.5,0.3) -- (1.5,0.5);
            \draw[->] (1,-0.5) \braidto (0,0.5);
            \draw[<-] (1.5,-0.5) \braidto (0.5,0.5);
            \token{east}{0,-0.3}{g};
            \token{west}{1.5,0.3}{g^{-1}};
        \end{tikzpicture}
        + \sum_{t \in G}
        \begin{tikzpicture}[anchorbase]
            \draw[->] (0,-0.7) -- (0,0.7);
            \draw[<-] (0.5,-0.7) -- (0.5,-0.6)  arc(180:0:0.25) -- (1,-0.7);;
            \draw[->] (0.5,0.7) -- (0.5,0.6) arc(180:360:0.25) -- (1,0.7);
            \draw[<-] (1.5,-0.7) -- (1.5,0.7);
            \token{east}{0,0}{t^{-1}};
            \token{west}{1.5,0}{t^{-1}g^{-1}};
            \token{north}{0.75,-0.35}{t};
            \token{north}{0.75,0.35}{gt};
        \end{tikzpicture}
        =
        \Psi
        \left(\,
            \begin{tikzpicture}[anchorbase]
                \draw (-0.2,-0.25) -- (-0.2,0.25);
                \draw (0.2,-0.25) -- (0.2,0.25);
                \token{west}{0.2,0}{g};
            \end{tikzpicture}
        \right)
        \circ \Psi \left( \crossing \right),
    \end{multline*}
    where we let $t=g^{-1}h$.  The other relations in \cref{GPC5} are straightforward to verify.
    \details{
        For the first relation, we have
        \[
            \Psi \left( \tokstrand \right) \circ \Psi \left( \tokstrand[h] \right)
            \overset{\cref{H1}}{=}
            \begin{tikzpicture}[anchorbase]
                \draw[->] (-0.2,-0.3) -- (-0.2,0.3);
                \draw[<-] (0.2,-0.3) -- (0.2,0.3);
                \token{east}{-0.2,0}{gh};
                \token{west}{0.2,0}{h^{-1}g^{-1}};
            \end{tikzpicture}
            =
            \begin{tikzpicture}[anchorbase]
                \draw[->] (-0.2,-0.3) -- (-0.2,0.3);
                \draw[<-] (0.2,-0.3) -- (0.2,0.3);
                \token{east}{-0.2,0}{gh};
                \token{west}{0.2,0}{(gh)^{-1}};
            \end{tikzpicture}
            = \Psi \left( \tokstrand[gh] \, \right).
        \]
        The second relation is clear.  We check the fourth relation:
        \[
            \Psi
            \left(
                \begin{tikzpicture}[anchorbase]
                    \draw (-0.2,-0.25) -- (-0.2,0.25);
                    \draw (0.2,-0.25) -- (0.2,0.25);
                    \token{east}{-0.2,0}{g};
                    \token{west}{0.2,0}{g};
                \end{tikzpicture}
            \right)
            \circ \Psi \left( \spliter \right)
            \overset{\cref{H1}}{=}
            \begin{tikzpicture}[anchorbase]
                \draw[->] (0.5,-0.5) to[out=up,in=-45] (0.25,0) to[out=135,in=down] (0,0.5);
                \draw[->] (0.5,0.5) to (0.5,0.4) arc(180:360:0.25) to (1,0.5);
                \draw[<-] (1,-0.5) to[out=up,in=225] (1.25,0) to[out=45,in=down] (1.5,0.5);
                \token{east}{0.25,0}{g};
                \token{west}{1.25,0}{g^{-1}};
            \end{tikzpicture}
            =
            \Phi_n \left( \spliter \right) \circ \Phi_n \left( \tokstrand\, \right).
        \]
        The fifth relation follows immediately from \cref{cct}.
    }
\end{proof}

We now wish to show that $\Psi$ is faithful.  Our approach is inspired by that of \cite[App.~A]{NS19}, which deals with the case where $G$ is the trivial group.

In what follows, we will identify a permutation $\pi \in \fS_k$ with the partition of type $\binom{k}{k}$ with parts $\{i,\pi(i)'\}$, $1 \le i \le k$.  Recall that, for $1 \le i < j \le k$, the pair $(i,j)$ is an \emph{inversion} in $\pi \in \fS_k$ if $\pi(i) > \pi(j)$.  Suppose $Q$ is a partition of type $\binom{l}{k}$.  We say that a permutation $P \in \fS_l$ is a \emph{left shuffle} for $Q$ if there is no inversion $(i,j)$ in $P$ such that vertices $i'$ and $j'$ lie in the same connected component of $Q$.  Intuitively, $P$ is a left shuffle for $Q$ if it does not change the relative order of vertices in each component.  Similarly, we say that a permutation $P \in \fS_k$ is a \emph{right shuffle} for $Q$ if there is no inversion $(i,j)$ in $P^{-1}$ such that vertices $i$ and $j$ lie in the same component of $Q$.  For example, if
\[
    P =
    \begin{tikzpicture}[centerzero]
        \pd{-0.5,-0.25};
        \pd{0,-0.25};
        \pd{0.5,-0.25};
        \pd{-0.5,0.25};
        \pd{0,0.25};
        \pd{0.5,0.25};
        \draw (-0.5,-0.25) -- (0,0.25);
        \draw (0,-0.25) -- (0.5,0.25);
        \draw (0.5,-0.25) -- (-0.5,0.25);
    \end{tikzpicture}
    \quad \text{and} \quad
    Q =
    \begin{tikzpicture}[centerzero]
        \pd{-0.5,-0.25};
        \pd{0,-0.25};
        \pd{0.5,-0.25};
        \pd{-0.5,0.25};
        \pd{0,0.25};
        \pd{0.5,0.25};
        \draw (-0.5,0.25) to (0,0.25) to (0,-0.25) to (-0.5,-0.25) to (-0.5,0.25);
    \end{tikzpicture}
    \ ,
\]
then $P$ is a left shuffle for $Q$ but not a right shuffle for $Q$.

We say a partition diagram is \emph{tensor-planar} if it is a tensor product (horizontal juxtaposition) of partition diagrams consisting of a single connected component.  Note that every tensor-planar partition diagram is planar (i.e.\ can be represented as a graph without edge crossings inside of the rectangle formed by its vertices) but the converse is false.

Every equivalence class $[P,\bg]$ of $G$-partitions can be factored as a product
\begin{equation} \label{plum}
    [P,\bg]
    =
    \begin{tikzpicture}[centerzero]
        \pd{-0.4,-0.2};
        \pd{0.4,-0.2};
        \pd{0.8,-0.2};
        \pd{-0.4,0.2} node[anchor=south] {\dotlabel{g_{l'}}};
        \pd{0.4,0.2} node[anchor=south] {\dotlabel{g_{2'}}};
        \pd{0.8,0.2} node[anchor=south] {\dotlabel{g_{1'}}};
        \draw (-0.4,-0.2) -- (-0.4,0.2);
        \draw (0.4,-0.2) -- (0.4,0.2);
        \draw (0.8,-0.2) -- (0.8,0.2);
        \node at (0.04,0) {$\cdots$};
    \end{tikzpicture}
    \circ [P_1] \circ [P_2] \circ [P_3] \circ
    \begin{tikzpicture}[centerzero]
        \pd{-0.4,-0.2} node[anchor=north] {\dotlabel{g_k}};
        \pd{0.4,-0.2} node[anchor=north] {\dotlabel{g_2}};
        \pd{0.8,-0.2} node[anchor=north] {\dotlabel{g_1}};
        \pd{-0.4,0.2};
        \pd{0.4,0.2};
        \pd{0.8,0.2};
        \draw (-0.4,-0.2) -- (-0.4,0.2);
        \draw (0.4,-0.2) -- (0.4,0.2);
        \draw (0.8,-0.2) -- (0.8,0.2);
        \node at (0.04,0) {$\cdots$};
    \end{tikzpicture}
    \ .
\end{equation}
where $P_2$ is tensor-planar, $P_1$ is a left shuffle for $P_2$, and $P_3$ is a right shuffle for $P_2$.  (See \cref{supernova}.)  The number of connected components in $P$ is equal to the number of connected components in $P_2$.  For example, the $G$-partition diagram
\[
    \begin{tikzpicture}[anchorbase]
        \pd{0.5,0} node[anchor=north] {\dotlabel{g_4}};
        \pd{1,0} node[anchor=north] {\dotlabel{g_3}};
        \pd{1.5,0} node[anchor=north] {\dotlabel{g_2}};
        \pd{2,0} node[anchor=north] {\dotlabel{g_1}};
        \pd{0,0.5} node[anchor=south] {\dotlabel{g_{5'}}};
        \pd{0.5,0.5} node[anchor=south] {\dotlabel{g_{4'}}};
        \pd{1,0.5} node[anchor=south] {\dotlabel{g_{3'}}};
        \pd{1.5,0.5} node[anchor=south] {\dotlabel{g_{2'}}};
        \pd{2,0.5} node[anchor=south] {\dotlabel{g_{1'}}};
        \draw (0,0.5) to[out=down,in=down] (1,0.5);
        \draw (0.5,0) to[out=up,in=up,looseness=0.5] (1,0);
        \draw (0.5,0.5) to[out=down,in=down,looseness=0.5] (1.5,0.5);
        \draw (1,0) \braidto (2,0.5);
        \draw (2,0) \braidto (1.5,0.5);
    \end{tikzpicture}
\]
has four connected components and decomposition
\[
    \begin{tikzpicture}[anchorbase]
        \pd{0.5,0} node[anchor=north] {\dotlabel{g_4}};
        \pd{1,0} node[anchor=north] {\dotlabel{g_3}};
        \pd{1.5,0} node[anchor=north] {\dotlabel{g_2}};
        \pd{2,0} node[anchor=north] {\dotlabel{g_1}};
        \pd{0,0.5} node[anchor=south] {\dotlabel{g_{5'}}};
        \pd{0.5,0.5} node[anchor=south] {\dotlabel{g_{4'}}};
        \pd{1,0.5} node[anchor=south] {\dotlabel{g_{3'}}};
        \pd{1.5,0.5} node[anchor=south] {\dotlabel{g_{2'}}};
        \pd{2,0.5} node[anchor=south] {\dotlabel{g_{1'}}};
        \draw (0,0.5) to[out=down,in=down] (1,0.5);
        \draw (0.5,0) to[out=up,in=up,looseness=0.5] (1,0);
        \draw (0.5,0.5) to[out=down,in=down,looseness=0.5] (1.5,0.5);
        \draw (1,0) \braidto (2,0.5);
        \draw (2,0) \braidto (1.5,0.5);
    \end{tikzpicture}
    =
    \begin{tikzpicture}[centerzero]
        \pd{-1,-0.25};
        \pd{-0.5,-0.25};
        \pd{0,-0.25};
        \pd{0.5,-0.25};
        \pd{1,-0.25};
        \pd{-1,0.25} node[anchor=south] {\dotlabel{g_{5'}}};
        \pd{-0.5,0.25} node[anchor=south] {\dotlabel{g_{4'}}};
        \pd{0,0.25} node[anchor=south] {\dotlabel{g_{3'}}};
        \pd{0.5,0.25} node[anchor=south] {\dotlabel{g_{2'}}};
        \pd{1,0.25} node[anchor=south] {\dotlabel{g_{1'}}};
        \draw (-1,-0.25) -- (-1,0.25);
        \draw (-0.5,-0.25) -- (-0.5,0.25);
        \draw (0,-0.25) -- (0,0.25);
        \draw (0.5,-0.25) -- (0.5,0.25);
        \draw (1,-0.25) -- (1,0.25);
    \end{tikzpicture}
    \circ [P_1] \circ [P_2] \circ [P_3] \circ
    \begin{tikzpicture}[centerzero]
        \pd{-0.5,-0.25} node[anchor=north] {\dotlabel{g_4}};
        \pd{0,-0.25} node[anchor=north] {\dotlabel{g_3}};
        \pd{0.5,-0.25} node[anchor=north] {\dotlabel{g_2}};
        \pd{1,-0.25} node[anchor=north] {\dotlabel{g_1}};
        \pd{-0.5,0.25};
        \pd{0,0.25};
        \pd{0.5,0.25};
        \pd{1,0.25};
        \draw (-0.5,-0.25) -- (-0.5,0.25);
        \draw (0,-0.25) -- (0,0.25);
        \draw (0.5,-0.25) -- (0.5,0.25);
        \draw (1,-0.25) -- (1,0.25);
    \end{tikzpicture}
    \ ,
\]
where
\[
    P_1 =
    \begin{tikzpicture}[anchorbase]
        \pd{0,0};
        \pd{0.5,0};
        \pd{1,0};
        \pd{1.5,0};
        \pd{2,0};
        \pd{0,0.5};
        \pd{0.5,0.5};
        \pd{1,0.5};
        \pd{1.5,0.5};
        \pd{2,0.5};
        \draw (0,0) \braidto (0,0.5);
        \draw (1.5,0) \braidto (1.5,0.5);
        \draw (2,0) \braidto (2,0.5);
        \draw (0.5,0) \braidto (1,0.5);
        \draw (1,0) \braidto (0.5,0.5);
    \end{tikzpicture},
    \quad
    P_2 =
    \begin{tikzpicture}[anchorbase]
        \pd{0,0};
        \pd{0.5,0};
        \pd{1,0};
        \pd{1.5,0};
        \pd{0,0.5};
        \pd{0.5,0.5};
        \pd{1,0.5};
        \pd{1.5,0.5};
        \pd{2,0.5};
        \draw (0,0.5) to[out=down,in=down] (0.5,0.5);
        \draw (0.5,0) \braidto (1,0.5) to[out=down,in=down] (1.5,0.5);
        \draw (1,0) to[out=up,in=up] (1.5,0) \braidto (2,0.5);
    \end{tikzpicture}
    =
    \begin{tikzpicture}[{>=To,baseline={(0,0.15)}}]
        \pd{0,0};
        \node at (0,.5) {};
    \end{tikzpicture}
    \otimes
    \begin{tikzpicture}[{>=To,baseline={(0,0.15)}}]
        \pd{0,0.5};
        \pd{0.5,0.5};
        \draw (0,0.5) to[out=down,in=down,looseness=1.5] (0.5,0.5);
        \node at (0,0) {};
    \end{tikzpicture}
    \otimes
    \begin{tikzpicture}[{>=To,baseline={(0,0.15)}}]
        \pd{0,0};
        \pd{0,0.5};
        \pd{0.5,0.5};
        \draw (0,0) to (0,0.5) to[out=down,in=down,looseness=1.5] (0.5,0.5);
    \end{tikzpicture}
    \otimes
    \begin{tikzpicture}[{>=To,baseline={(0,0.15)}}]
        \pd{0,0};
        \pd{0.5,0};
        \pd{0,0.5};
        \draw (0.5,0) to[out=up,in=up,looseness=1.5] (0,0) to (0,0.5);
    \end{tikzpicture}
  ,
  \quad
  P_3 =
  \begin{tikzpicture}[anchorbase]
    \pd{0.5,0};
    \pd{1,0};
    \pd{1.5,0};
    \pd{2,0};
    \pd{0.5,0.5};
    \pd{1,0.5};
    \pd{1.5,0.5};
    \pd{2,0.5};
    \draw (0.5,0) to[out=up,in=down,looseness=0.5] (1.5,0.5);
    \draw (1,0) to[out=up,in=down,looseness=0.5] (2,0.5);
    \draw (1.5,0) to[out=up,in=down,looseness=0.5] (0.5,0.5);
    \draw (2,0) to[out=up,in=down,looseness=0.5] (1,0.5);
  \end{tikzpicture}.
\]

For $n,k,l \in \N$, let $\Hom_{\Par(G)}^{\le n}(\go^{\otimes k},\go^{\otimes l})$ denote the subspace of $\Hom_{\Par(G)}(\go^{\otimes k},\go^{\otimes l})$ spanned by $G$-partition diagrams with at most $n$ connected components.  Composition respects the corresponding filtration on morphism spaces since the composite of a partition with $n$ connected components and a partition with $m$ connected components has at most $n+m$ connected components.

Recall the bases of the morphism spaces of $\Heis(G)$ given in \cite[Thm.~7.2]{BSW20}.  (The category $\Heis(G)$ is $\Heis_{-1}(\kk G)$ in the notation of \cite{BSW20}.)  For any such basis element $f$ in \linebreak $\Hom_{\Heis(G)} \big( (\uparrow \downarrow)^{\otimes k}, (\uparrow \downarrow)^{\otimes l} \big)$, define the \emph{block number} of $f$ to be number of distinct closed (possibly intersecting) loops in the diagram
\[
  \begin{tikzpicture}[anchorbase]
    \draw[->] (0,0) -- (0,0.1) arc (180:0:.25) -- (0.5,0);
  \end{tikzpicture}^{\otimes l}
  \circ
  f
  \circ
  \begin{tikzpicture}[anchorbase]
    \draw[<-] (0,1) -- (0,0.9) arc (180:360:.25) -- (0.5,1);
  \end{tikzpicture}^{\otimes k}.
\]
For $n \in \N$, let $\Hom_{\Heis(G)}^{\le n} \big( (\uparrow \downarrow)^{\otimes k}, (\uparrow \downarrow)^{\otimes l} \big)$ denote the subspace of $\Hom_{\Heis(G)} \big( (\uparrow \downarrow)^{\otimes k}, (\uparrow \downarrow)^{\otimes l} \big)$ spanned by basis elements with block number at most $n$.  Composition respects the resulting filtration on morphism spaces.

The image under $\Psi$ of tensor-planar partition diagrams (writing the image in terms of the aforementioned bases of the morphism spaces of $\Heis(G)$) is particularly simple to describe.  Since each tensor-planar partition diagram is a tensor product of single connected components, consider the case of a single connected component.  Then, for example, we have
\[
    \Psi
    \left(
        \begin{tikzpicture}[anchorbase]
            \pd{0,0.5};
            \pd{0.5,0.5};
            \pd{1,0.5};
            \draw (0,0.5) to[out=up,in=up,looseness=1.5] (0.5,0.5);
            \draw (0.5,0.5) to[out=up,in=up,looseness=1.5] (1,0.5);
        \end{tikzpicture}
    \right)
    =
    \begin{tikzpicture}[>=To,baseline={(0,0.2)}]
        \draw[<-] (1.5,0) to[out=up,in=up,looseness=2] (2,0);
        \draw[<-] (0.5,0) to[out=up,in=up,looseness=2] (1,0);
        \draw[->] (0,0) to[out=up,in=up] (2.5,0);
    \end{tikzpicture}
    \quad \text{and} \quad
    \Psi
    \left(
        \begin{tikzpicture}[anchorbase]
            \pd{0,0};
            \pd{0.5,0};
            \pd{1,0};
            \pd{0,0.5};
            \pd{0.5,0.5};
            \draw (1,0) to (0,0) to (0,0.5) to (0.5,0.5);
        \end{tikzpicture}
    \right)
    =
    \begin{tikzpicture}[anchorbase]
        \draw[->] (0,0) \braidto (0.5,1);
        \draw[<-] (0.5,0) to[out=up,in=up,looseness=2] (1,0);
        \draw[<-] (1.5,0) to[out=up,in=up,looseness=2] (2,0);
        \draw[<-] (2.5,0) \braidto(2,1);
        \draw[->] (1,1) to[out=down,in=down,looseness=2] (1.5,1);
    \end{tikzpicture}
    \ .
\]
The general case is analogous.  (In fact, the images of all planar partition diagrams are similarly easy to describe.)  In particular, if $P$ is a tensor-planar partition diagram with $n$ connected components, then $\Psi(P)$ is a planar diagram with block number $n$.

For $i=1,\dotsc,k-1$, consider the morphism
\begin{multline}
    \chi_i :=
    \Psi
    \left(
        1_\go^{\otimes (k-i-1)} \otimes
        \left(
            \begin{tikzpicture}[centerzero]
                \pd{-0.25,-0.25};
                \pd{-0.25,0.25};
                \pd{0.25,-0.25};
                \pd{0.25,0.25};
                \draw (-0.25,-0.25) \braidto (0.25,0.25);
                \draw (0.25,-0.25) \braidto (-0.25,0.25);
            \end{tikzpicture}
            -
            \begin{tikzpicture}[centerzero]
                \pd{-0.25,-0.25};
                \pd{-0.25,0.25};
                \pd{0.25,-0.25};
                \pd{0.25,0.25};
                \draw (-0.25,-0.25) to (-0.25,0.25) to (0.25,0.25) to (0.25,-0.25) to (-0.25,-0.25);
            \end{tikzpicture}
        \right)
        \otimes 1_\go^{\otimes (i-1)}
    \right)
    \\
    =
    \begin{tikzpicture}[anchorbase]
        \draw[->] (0,0) to (0,1);
        \draw[<-] (0.5,0) to (0.5,1);
        \node at (1.25,0.5) {$\cdots$};
        \draw[->] (2,0) to (2,1);
        \draw[<-] (2.5,0) to (2.5,1);
        \draw[->] (3,0) \braidto (4,1);
        \draw[<-] (3.5,0) \braidto (4.5,1);
        \draw[->] (4,0) \braidto (3,1);
        \draw[<-] (4.5,0) \braidto (3.5,1);
        \draw[->] (5,0) to (5,1);
        \draw[<-] (5.5,0) to (5.5,1);
        \node at (6.25,0.5) {$\cdots$};
        \draw[->] (7,0) to (7,1);
        \draw[<-] (7.5,0) to (7.5,1);
    \end{tikzpicture}
    \ \in \End_{\Heis(G)} \big( (\uparrow \downarrow)^{\otimes k} \big).
\end{multline}
For a permutation partition diagram $P \colon k \to k$, let $T(P)$ be the morphism in $\Heis(G)$ defined as follows:  Write $D = s_{i_1} \circ s_{i_2} \circ \dotsb \circ s_{i_r}$ as a reduced word in simple transpositions and let
\[
    T(P)
    = \chi_{i_1} \circ \chi_{i_2} \circ \dotsb \circ \chi_{i_r}.
\]
It follows from the braid relations \cref{genbraid} that $T(P)$ is independent of the choice of reduced word for $P$.

\begin{lem} \label{towel}
    Suppose $(P,\bg)$ is a $G$-partition diagram with decomposition \cref{plum}.  Then
    \[
        \Psi([P,\bg]) -
        \begin{tikzpicture}[centerzero]
            \pd{-0.4,-0.2};
            \pd{0.4,-0.2};
            \pd{0.8,-0.2};
            \pd{-0.4,0.2} node[anchor=south] {\dotlabel{g_{l'}}};
            \pd{0.4,0.2} node[anchor=south] {\dotlabel{g_{2'}}};
            \pd{0.8,0.2} node[anchor=south] {\dotlabel{g_{1'}}};
            \draw (-0.4,-0.2) -- (-0.4,0.2);
            \draw (0.4,-0.2) -- (0.4,0.2);
            \draw (0.8,-0.2) -- (0.8,0.2);
            \node at (0.04,0) {$\cdots$};
        \end{tikzpicture}
        \circ [T(P_1)] \circ [P_2] \circ [T(P_3)] \circ
        \begin{tikzpicture}[centerzero]
            \pd{-0.4,-0.2} node[anchor=north] {\dotlabel{g_k}};
            \pd{0.4,-0.2} node[anchor=north] {\dotlabel{g_2}};
            \pd{0.8,-0.2} node[anchor=north] {\dotlabel{g_1}};
            \pd{-0.4,0.2};
            \pd{0.4,0.2};
            \pd{0.8,0.2};
            \draw (-0.4,-0.2) -- (-0.4,0.2);
            \draw (0.4,-0.2) -- (0.4,0.2);
            \draw (0.8,-0.2) -- (0.8,0.2);
            \node at (0.04,0) {$\cdots$};
        \end{tikzpicture}
        \in \Hom_{\Heis(G)}^{\le n-1} \big( (\uparrow \downarrow)^{\otimes k}, (\uparrow \downarrow)^{\otimes l} \big).
    \]
\end{lem}

\begin{proof}
    The case where $\bg = \mathbf{1}$ is \cite[Prop.~A.1]{NS19}.  Since composition with $G$-partition diagrams of the form
    \[
        \begin{tikzpicture}[centerzero]
            \pd{-0.4,-0.2};
            \pd{0.4,-0.2};
            \pd{0.8,-0.2};
            \pd{-0.4,0.2} node[anchor=south] {\dotlabel{g_{l'}}};
            \pd{0.4,0.2} node[anchor=south] {\dotlabel{g_{2'}}};
            \pd{0.8,0.2} node[anchor=south] {\dotlabel{g_{1'}}};
            \draw (-0.4,-0.2) -- (-0.4,0.2);
            \draw (0.4,-0.2) -- (0.4,0.2);
            \draw (0.8,-0.2) -- (0.8,0.2);
            \node at (0.04,0) {$\cdots$};
        \end{tikzpicture}
        \quad \text{and} \quad
        \begin{tikzpicture}[centerzero]
            \pd{-0.4,-0.2} node[anchor=north] {\dotlabel{g_k}};
            \pd{0.4,-0.2} node[anchor=north] {\dotlabel{g_2}};
            \pd{0.8,-0.2} node[anchor=north] {\dotlabel{g_1}};
            \pd{-0.4,0.2};
            \pd{0.4,0.2};
            \pd{0.8,0.2};
            \draw (-0.4,-0.2) -- (-0.4,0.2);
            \draw (0.4,-0.2) -- (0.4,0.2);
            \draw (0.8,-0.2) -- (0.8,0.2);
            \node at (0.04,0) {$\cdots$};
        \end{tikzpicture}
    \]
    does not change the block number, the general case follows.
\end{proof}

\begin{proof}[Proof of \cref{hitchcock}]
    Since $\Psi$ is well defined by \cref{poker}, it remains to show it is faithful.  As in \cref{gummies}, we view $\Par(G)$ as $\Par(G,d)$ over the ring $\kk[d]$.  Note that the image of $\Psi$ is contained in the full monoidal subcategory $\Peis(G)$ of $\Heis(G)$ generated by the object $\uparrow \otimes \downarrow$.  It follows from the defining relations of $\Heis(G)$ that
    \[
        \begin{tikzpicture}[anchorbase]
            \draw[->] (0,0.3) arc(450:90:0.3);
            \draw[->] (0.6,-0.5) to (0.6,0.5);
            \draw[<-] (0.9,-0.5) to (0.9,0.5);
        \end{tikzpicture}
        \ =\
        \begin{tikzpicture}[anchorbase]
            \draw[->] (0,0.3) arc(450:90:0.3);
            \draw[<-] (-0.6,-0.5) to (-0.6,0.5);
            \draw[->] (-0.9,-0.5) to (-0.9,0.5);
        \end{tikzpicture}
        \ .
    \]
    In other words, the clockwise bubble is strictly central in $\Peis(G)$.    Let $\Peis(G,d)$ be the quotient of $\Peis(G)$, defined over $\kk[d]$, by the additional relation
    \[
        \begin{tikzpicture}[anchorbase]
            \draw[->] (0,0.3) arc(450:90:0.3);
        \end{tikzpicture}
        \ = d 1_\one.
    \]
    Then, as in \cref{gummies}, $\Peis(G)$, defined over $\kk$, is isomorphic as a $\kk$-linear category to $\Peis(G,d)$.  In other words, we can view the clockwise bubble of $\Peis(G)$ and the morphism $\lolly$ of $\Par(G)$, both of which are strictly central, as elements of the ground ring.

    Now, it is clear that, in the setting of \cref{towel},
    \begin{equation} \label{hula}
        \begin{tikzpicture}[centerzero]
            \pd{-0.4,-0.2};
            \pd{0.4,-0.2};
            \pd{0.8,-0.2};
            \pd{-0.4,0.2} node[anchor=south] {\dotlabel{g_{l'}}};
            \pd{0.4,0.2} node[anchor=south] {\dotlabel{g_{2'}}};
            \pd{0.8,0.2} node[anchor=south] {\dotlabel{g_{1'}}};
            \draw (-0.4,-0.2) -- (-0.4,0.2);
            \draw (0.4,-0.2) -- (0.4,0.2);
            \draw (0.8,-0.2) -- (0.8,0.2);
            \node at (0.04,0) {$\cdots$};
        \end{tikzpicture}
        \circ [T(P_1)] \circ [P_2] \circ [T(P_3)] \circ
        \begin{tikzpicture}[centerzero]
            \pd{-0.4,-0.2} node[anchor=north] {\dotlabel{g_k}};
            \pd{0.4,-0.2} node[anchor=north] {\dotlabel{g_2}};
            \pd{0.8,-0.2} node[anchor=north] {\dotlabel{g_1}};
            \pd{-0.4,0.2};
            \pd{0.4,0.2};
            \pd{0.8,0.2};
            \draw (-0.4,-0.2) -- (-0.4,0.2);
            \draw (0.4,-0.2) -- (0.4,0.2);
            \draw (0.8,-0.2) -- (0.8,0.2);
            \node at (0.04,0) {$\cdots$};
        \end{tikzpicture}
    \end{equation}
    is uniquely determined by $[P,\bg]$.  Indeed, $P$ is the partition diagram obtained from $T(P_1) \circ \Psi(P_2) \circ T(P_3)$ by replacing each pair $\uparrow \downarrow$ by a vertex and each strand by an edge.  Furthermore, the diagrams of the form \cref{hula} are linearly independent by \cite[Thm.~7.2]{BSW20}.  The result then follows by a standard triangularity argument.
\end{proof}

\section{Compatibility of categorical actions\label{sec:compat}}

We continue to assume that $G$ is a finite group.  The group Heisenberg category acts naturally on the direct sum of the categories $A_n\md$, $n \in \N$.  In this section, we recall this action and show that it is compatible with the embedding of the group partition category into the group Heisenberg category and the action of the group partition category described in \cref{hide}.

For $0 \le m,k \le n$, let ${}_k(n)_m$ denote $A_n$, considered as an $(A_k,A_m)$-bimodule.  We will omit the subscript $k$ or $m$ when $k=n$ or $m=n$, respectively.  We denote tensor product of such bimodules by juxtaposition.  For instance $(n)_{n-1}(n)$ denotes $A_n \otimes_{n-1} A_n$, considered as an $(A_n,A_n)$-bimodule, where we write $\otimes_m$ for the tensor product over $A_m$. As explained in \cite[\S7]{RS17}, we have a strong $\kk$-linear monoidal functor
\[
    \Theta \colon \Heis(G) \to \prod_{m \in \N} \left( \bigoplus_{n \in \N} (A_n,A_m)\bmd \right)
\]
given by
\begin{align*}
    \Theta
    \left(
        \begin{tikzpicture}[anchorbase]
            \draw[->] (-0.2,-0.2) -- (0.2,0.2);
            \draw[->] (0.2,-0.2) -- (-0.2,0.2);
        \end{tikzpicture}
    \right)
    &=
    \left(
        (n)_{n-2} \to (n)_{n-2},\ x \mapsto x s_{n-1}
    \right)_{n \ge 2},
    \\
    \Theta
    \left(
        \begin{tikzpicture}[anchorbase]
            \draw[->] (-0.2,0.2) -- (-0.2,0) arc (180:360:0.2) -- (0.2,0.2);
        \end{tikzpicture}
    \right)
    &=
    \left(
        (n-1) \to {}_{n-1}(n)_{n-1},\ x \mapsto x
    \right)_{n \ge 1},
    \\
    \Theta
    \left(
        \begin{tikzpicture}[anchorbase]
            \draw[->] (-0.2,-0.2) -- (-0.2,0) arc (180:0:0.2) -- (0.2,-0.2);
        \end{tikzpicture}
    \right)
    &=
    \left(
        (n)_{n-1}(n) \to (n),\ x \otimes y \mapsto xy
    \right)_{n \ge 1},
    \\
    \Theta
    \left(
        \begin{tikzpicture}[anchorbase]
            \draw[<-] (-0.2,0.2) -- (-0.2,0) arc (180:360:0.2) -- (0.2,0.2);
        \end{tikzpicture}
    \right)
    &= \textstyle
    \left(
        (n) \to (n)_{n-1}(n),\ x \mapsto x \sum_{1 \le i \le n,\ h \in G} h^{(i)} \pi_i \otimes \pi_{i}^{-1} \left(h^{-1}\right)^{(i)}
    \right)_{n \ge 1},
    \\
    \Theta
    \left(
        \begin{tikzpicture}[anchorbase]
            \draw[<-] (-0.2,-0.2) -- (-0.2,0) arc (180:0:0.2) -- (0.2,-0.2);
        \end{tikzpicture}
    \right)
    &=
    \left(
        {}_{n-1}(n)_{n-1} \to (n-1),\ \bg \pi \mapsto
        \begin{cases}
            (g_{n-1},g_{n-2},\dotsc,g_1)\pi & \text{if } \pi \in \fS_{n-1},\, g_n=1_G \\
            0 & \text{otherwise}
        \end{cases}
    \right)_{n \ge 1},
    \\
    \Theta
    \left(
        \begin{tikzpicture}[anchorbase]
            \draw[->] (0,-0.2) -- (0,0.2);
            \token{west}{0,0}{g};
        \end{tikzpicture}
    \right)
    &=
    \left(
        (n)_{n-1} \to (n)_{n-1},\ x \mapsto x \left( g^{-1} \right)^{(n)}
    \right)_{n \ge 1},
    \\
    \Theta
    \left(
        \begin{tikzpicture}[anchorbase]
            \draw[<-] (-0.2,-0.2) -- (0.2,0.2);
            \draw[->] (0.2,-0.2) -- (-0.2,0.2);
        \end{tikzpicture}
    \right)
    &=
    \left(
        {}_{n-1}(n)_{n-1} \to (n-1)_{n-2}(n-1),
    \right.
    \\
    &\qquad
    \left.
        \bg \pi \mapsto
        \begin{cases}
             (g_{n-1},\dotsc,g_1) \sigma_1 \otimes g_n^{(n-1)} \sigma_2 &\text{if } \pi = \sigma_1 s_{n-1} \sigma_2 \text{ for } \sigma_1,\sigma_2 \in \fS_{n-1}, \\
            0 & \text{if } \pi \in \fS_{n-1}
        \end{cases}
    \right)_{n \ge 2},
    \\
    \Theta
    \left(
        \begin{tikzpicture}[anchorbase]
            \draw[->] (-0.2,-0.2) -- (0.2,0.2);
            \draw[<-] (0.2,-0.2) -- (-0.2,0.2);
        \end{tikzpicture}
    \right)
    &=
    \left(
        (n-1)_{n-2}(n-1) \to {}_{n-1}(n)_{n-1},\ x \otimes y \mapsto x s_{n-1} y
    \right)_{n \ge 2},
    \\
    \Theta
    \left(
        \begin{tikzpicture}[anchorbase]
            \draw[<-] (-0.2,-0.2) -- (0.2,0.2);
            \draw[<-] (0.2,-0.2) -- (-0.2,0.2);
        \end{tikzpicture}
    \right)
    &=
    \left(
        {}_{n-2}(n) \to {}_{n-2}(n),\ x \mapsto s_{n-1} x
    \right)_{n \ge 2},\\
    \Theta
    \left(
        \begin{tikzpicture}[anchorbase]
            \draw[<-] (0,-0.2) -- (0,0.2);
            \token{east}{0,0}{g};
        \end{tikzpicture}
    \right)
    &=
    \left(
        {}_{n-1}(n)\to {}_{n-1}(n),\  x \mapsto \left( g^{-1} \right)^{(n)} x
    \right)_{n \ge 1}.
\end{align*}

As noted in the proof of \cref{hitchcock}, the image of $\Psi$ lies in the full monoidal subcategory $\Peis(G)$ of $\Heis(G)$ generated by $\uparrow \otimes \downarrow$.  For $n \in \N$, consider the composition
\[
    \Omega_n \colon \Peis(G)
    \xrightarrow{\Theta} \bigoplus_{m \in \N} (A_m,A_m)\bmd
    \xrightarrow{- \otimes_{A_n} \mathbf{1}_n} \bigoplus_{m \in \N} A_m\md,
\]
where we declare $M \otimes_{A_n} \mathbf{1}_n = 0$ for $M \in (A_m,A_m)\bmd$ with $m \ne n$.  The functor $\Omega_n$ is $\kk$-linear, but no longer monoidal.

\begin{theo} \label{wactcom}
    Consider the functors:
    \begin{equation}
        \begin{tikzcd}[column sep=2cm]
            \Par(G) \arrow[r,"\Psi"]
            \arrow[rd, swap,"\Phi_n"] &
            \Peis(G)
            \arrow[d,"\Omega_n"]
            \\
            & A_n\md
        \end{tikzcd}
      \ .
    \end{equation}
    The isomorphisms $\beta_k$, $k \in \N$, defined in \cref{definbeta}, give a natural isomorphism of functors $\Omega_n \circ \Psi \cong \Phi_n$.
\end{theo}

\begin{proof}
    Since the $\beta_k$ are isomorphisms of $A_n$-modules it suffices to show that they determine a natural transformation between the given functors.  Therefore, following the argument used in the proof of \cite[Theorem~5.1]{NS19} we need to check elements of the form
    \[
        1_{\go^{\otimes k}} \otimes x \otimes 1_{\go^{\otimes j}},\quad
        k,j \in \N,\quad x \in \left\{ \tokstrand[h],\ \crossing,\ \spliter,\ \merge,\ \toppin,\ \bottompin \, : h \in G \right\}.
    \]

    \medskip

    \noindent \emph{Tokens}:  For $h \in G$,
    \[
        \beta_k^{-1} \circ \left( \Omega_n \circ \Psi \left( 1_{\go^{\otimes(k-j-1)}} \otimes \tokstrand[h] \otimes 1_{\go^{\otimes (j-1)}} \right) \right) \circ \beta_k
        \colon V^{\otimes k} \to V^{\otimes k}
    \]
    is the $A_n$-module map given by
    \begin{align*}
        g_k e_{i_k} \otimes \dotsb &\otimes g_1 e_{i_1}
        \mapsto g_k^{(i_k)} \pi_{i_k} \otimes \pi_{i_k}^{-1} \left( g_k^{-1} \right)^{(i_k)} g_{k-1}^{(i_{k-1})} \pi_{i_{k-1}} \otimes \dotsb \otimes \pi_{i_2}^{-1} \left( g_2^{-1} \right)^{(i_2)} g_1^{(i_1)} \pi_{i_1} \otimes 1
        \\
        &\mapsto g_k^{(i_k)} \pi_{i_k} \otimes \pi_{i_k}^{-1} \left( g_k^{-1} \right)^{(i_k)} g_{k-1}^{(i_{k-1})} \pi_{i_{k-1}} \otimes \dotsb \otimes \pi_{i_{j+1}}^{-1} \left( g_{j+1}^{-1} \right)^{(i_{j+1})} g_1^{(i_j)} \pi_{i_{j}} \left( h^{-1} \right)^{(n)}
        \\
        &\qquad \qquad \otimes h^{(n)}\pi_{i_j}^{-1} \left( g_j^{-1} \right)^{(i_j)} g_1^{i_{j-1}} \pi_{i_{j-1}} \otimes\dotsb \otimes \pi_{i_2}^{-1} \left( g_2^{-1} \right)^{(i_2)} g_1^{(i_1)} \pi_{i_1} \otimes 1
        \\
        &\mapsto g_k^{(i_k)} \pi_{i_k} \otimes \pi_{i_k}^{-1} \left( g_k^{-1} \right)^{(i_k)} g_{k-1}^{(i_{k-1})} \pi_{i_{k-1}} \otimes \dotsb \otimes \pi_{i_{j+1}}^{-1} \left( g_{j+1}^{-1} \right)^{(i_{j+1})} \left( g_j h^{-1} \right)^{(i_j)} \pi_{i_{j}}
        \\
        &\qquad \qquad \otimes \pi_{i_j}^{-1} \left( h g_j^{-1} \right)^{(i_j)} g_1^{i_{j-1}} \pi_{i_{j-1}} \otimes\dotsb \otimes \pi_{i_2}^{-1} \left( g_2^{-1} \right)^{(i_2)} g_1^{(i_1)} \pi_{i_1} \otimes 1
        \\
        &\mapsto g_{k}e_{i_k} \otimes \dotsb \otimes g_{j+1}e_{i_{j+1}} \otimes g_j h^{-1} e_{i_j} \otimes g_{j-1} e_{i_{j-1}}\otimes \dotsb \otimes g_1 e_{i_1}.
    \end{align*}
    This is precisely the map $\Phi_n (1_{\go^{\otimes (k-j-1)}} \otimes \tokstrand[h] \otimes 1_{\go^{\otimes (j-1)}})$.

    \medskip

    \noindent \emph{Merge}:  For $g,h \in G$ and $1 \le i,j \le n$,
    \begin{equation} \label{cliff}
        \pi_i^{-1} \left( g^{-1} \right)^{(i)} h^{(j)} \pi_j
        = \left(g^{-1}\right)^{(n)} h^{(\pi_i^{-1}(j))} \pi_i^{-1} \pi_j.
    \end{equation}
    We have $\pi_i^{-1} \pi_j \in \fS_{n-1}$ if and only if $i=j$, in which case \cref{cliff} is equal to $\left( g^{-1} h \right)^{(n)}$.  Thus, the composition
    \[
        \beta_{k-1}^{-1} \circ \left( \Omega_n \circ \Psi \left( 1_{\go^{\otimes (k-j-1)}} \otimes \merge \otimes 1_{\go^{\otimes (j-1)}} \right) \right) \circ \beta_k
        \colon V^{\otimes k} \to V^{\otimes (k-1)}
    \]
    is the $A_n$-module map given by
    \begin{align*}
        g_k e_{i_k} \otimes \dotsb \otimes g_1 e_{i_1}
        &\mapsto g_k^{(i_k)} \pi_{i_k} \otimes \pi_{i_k}^{-1} \left( g_k^{-1} \right)^{(i_k)} g_{k-1}^{(i_{k-1})} \pi_{i_{k-1}}\otimes \dotsb \otimes \pi_{i_{j+1}}^{-1} \left( g_{j+1}^{-1} \right)^{(i_{j+1})} g_{j}^{(i_{j})} \pi_{i_{j}} \otimes \\
        &\qquad \qquad \dotsb \otimes \pi_{i_2}^{-1} \left( g_2^{-1} \right)^{(i_2)} g_1^{(i_1)} \pi_{i_1} \otimes 1
        \\
        &\mapsto \delta_{i_j,i_{j+1}} \delta_{g_j,g_{j+1}} g_ke_{i_k} \otimes \dotsb \otimes g_{j+2}e_{i_{j+2}} \otimes g_{j} e_{i_j} \otimes \dotsb \otimes g_1 e_{i_1}.
    \end{align*}
    This is precisely the map $\Phi_n \left( 1_{\go^{\otimes (k-j-1)}} \otimes \merge \otimes 1_{\go^{\otimes (j-1)}} \right)$.

    \medskip

    \noindent \emph{Split}:  The composition
    \[
        \beta_{k+1}^{-1} \circ \left( \Omega_n \circ \Psi \left( 1_{\go^{\otimes (k-j)}} \otimes \spliter \otimes 1_{\go^{\otimes (j-1)}} \right) \right) \circ \beta_k
        \colon V^{\otimes k} \to V^{\otimes (k+1)}
    \]
    is the $A_n$-module map given by
    \begin{align*}
        g_k e_{i_k} \otimes \dotsb \otimes g_1 e_{i_1}
        &\mapsto g_k^{(i_k)} \pi_{i_k} \otimes \pi_{i_k}^{-1} \left( g_k^{-1} \right)^{(i_k)} g_{k-1}^{(i_{k-1})} \pi_{i_{k-1}}\otimes\dotsb\otimes \pi_{i_{j+1}}^{-1} \left( g_{j+1}^{-1} \right)^{(i_{j+1})} g_{j}^{(i_{j})} \pi_{i_{j}} \otimes 1 \\
        &\qquad \qquad \otimes  \pi_{i_{j}}^{-1} \left( g_{j}^{-1} \right)^{(i_{j})} g_{j-1}^{(i_{j-1})} \otimes \dotsb \otimes \pi_{i_2}^{-1} \left( g_2^{-1} \right)^{(i_2)} g_1^{(i_1)} \pi_{i_1} \otimes 1
        \\
        &\mapsto g_{k}e_{i_k} \otimes \dotsb \otimes g_{j+1} e_{i_{j+1}} \otimes g_{j}e_{i_j} \otimes g_{j}e_{i_j} \otimes g_{j-1}e_{i_{j-1}} \otimes \dotsb \otimes g_1 e_{i_1}.
    \end{align*}
    This is precisely the map $\Phi_n \left( 1_{\go^{\otimes (k-j)}} \otimes \spliter \otimes 1_{\go^{\otimes (j-1)}} \right)$.

    \medskip

    \noindent \emph{Unit pin}:  The composition
    \[
        \beta_{k+1}^{-1} \circ \left( \Omega_n \circ \Psi \left( 1_{\go^{\otimes (k-j)}} \otimes \bottompin \otimes 1_{\go^{\otimes j}} \right) \right) \circ \beta_k
        \colon V^{\otimes k} \to V^{\otimes (k+1)}
    \]
    is the map
    \[
        g_k e_{i_k} \otimes \dotsb \otimes g_1 e_{i_1}
        \mapsto \sum_{h \in G} \sum_{i=1}^n g_k e_{i_k} \otimes \dotsb \otimes g_{j+1} e_{i_{j+1}} \otimes h e_i \otimes g_j e_{i_j} \otimes \dotsb \otimes g_1 e_{i_1},
    \]
    which is equal to the map $\Phi_n (1_{\go^{\otimes (k-j)}} \otimes \bottompin \otimes 1_{\go^{\otimes j}})$.

    \medskip

    \noindent \emph{Counit pin}: The composition
    \[
        \beta_{k-1}^{-1} \circ \left( \Omega_n \circ \Psi \left( 1_{\go^{\otimes (k-j)}} \otimes \toppin \otimes 1_{\go^{\otimes (j-1)}} \right) \right) \circ \beta_k
        \colon V^{\otimes k} \to V^{\otimes (k-1)}
    \]
    is the map
    \[
        g_k e_{i_k} \otimes \dotsb \otimes g_1 e_{i_1}
        \mapsto g_k e_{i_k} \otimes \dotsb \otimes g_{j+1} e_{i_{j+1}} \otimes g_{j-1} e_{i_{j-1}} \otimes \dotsb \otimes g_1 e_{i_1},
    \]
    which is equal to the map $\Phi_n(1_{\go^{\otimes (k-j)}} \otimes \toppin \otimes 1_{\go^{\otimes (j-1)}})$.

    \medskip

    \noindent \emph{Crossing}:  Define the elements $f,f' \in \End_{\Heis}(\uparrow \downarrow \uparrow \downarrow)$ by
    \begin{equation} \label{breakdown}
        f =
        \begin{tikzpicture}[anchorbase]
            \draw[->] (0,0) \braidto (1,1);
            \draw[<-] (0.5,0) \braidto (1.5,1);
            \draw[->] (1,0) \braidto (0,1);
            \draw[<-] (1.5,0) \braidto (0.5,1);
        \end{tikzpicture}
        \ ,\qquad
        f' =
        \sum_{g \in G}
        \begin{tikzpicture}[anchorbase]
            \draw[->] (-0.5,-0.6) -- (-0.5,0.6);
            \draw[->] (-0.3,0.6) -- (-0.3,0.5) arc(180:360:0.3) -- (0.3,0.6);
            \draw[<-] (-0.3,-0.6) -- (-0.3,-0.5) arc(180:0:0.3) -- (0.3,-0.6);
            \draw[<-] (0.5,-0.6) -- (0.5,0.6);
            \token{south}{0,0.2}{g};
            \token{north}{0,-0.2}{g};
            \token{east}{-0.5,0}{g^{-1}};
            \token{west}{0.5,0}{g^{-1}};
        \end{tikzpicture}
        \ .
    \end{equation}
    Note that
    \[
        f = f_3 \circ f_2 \circ f_1,
    \]
    where
    \[
        f_1 =
        \begin{tikzpicture}[anchorbase]
            \draw[->] (0,0) to (0,0.6);
            \draw[<-] (0.5,0) \braidto (1,0.6);
            \draw[->] (1,0) \braidto (0.5,0.6);
            \draw[<-] (1.5,0) to (1.5,0.6);
        \end{tikzpicture}
        \ ,\quad
        f_2 =
        \begin{tikzpicture}[anchorbase]
            \draw[->] (0,0) \braidto (0.5,0.6);
            \draw[->] (0.5,0) \braidto (0,0.6);
            \draw[<-] (1,0) \braidto (1.5,0.6);
            \draw[<-] (1.5,0) \braidto (1,0.6);
        \end{tikzpicture}
        \ ,\quad
        f_3 =
        \begin{tikzpicture}[anchorbase]
            \draw[->] (0,0) to (0,0.6);
            \draw[->] (0.5,0) \braidto (1,0.6);
            \draw[<-] (1,0) \braidto (0.5,0.6);
            \draw[<-] (1.5,0) to (1.5,0.6);
        \end{tikzpicture}
        \ .
    \]
    Suppose $i,j \in \{1,\dotsc,n\}$ and $x,y \in A_n$.  We first compute the action of $\Theta(f)$ and $\Theta(f')$ on
    \begin{align*}
        \alpha &= x {g_1}^{(i)}\pi_i \otimes \pi_i^{-1}\left(g_1^{-1}\right)^{(i)} {g_2}^{(j)}\pi_j \otimes \pi_j^{-1}\left(g_{2}^{-1}\right)^{(j)} y
        \\
        &= x {g_1}^{(i)}\pi_i \otimes \left(g_1^{-1}\right)^{(n)} {g_2}^{(\pi_i^{-1}(j))} \pi_i^{-1}\pi_j \otimes \pi_j^{-1}\left(g_{2}^{-1}\right)^{(j)} y
        \in (n)_{n-1}(n)_{n-1}(n),
    \end{align*}
    where $x,y \in A_n$.  If $i = j$, then $\pi_i^{-1} \pi_j = 1_{\fS_n}$, and so $\Omega_n(f_1)(\alpha) = 0$.  Now suppose $i < j$ so that
    \[
        \pi_i^{-1} \pi_j
        = s_{n-1} \dotsm s_i s_j \dotsm s_{n-1}
        = s_{j-1} \dotsm s_{n-2} s_{n-1} s_{n-2} \dotsm s_i.
    \]
    Thus
    \[
        \Theta(f_1)(\alpha)
        = x g_1^{(i)} \pi_i g_2^{(\pi_i^{-1}(j))} s_{j-1} \dotsm s_{n-2} \otimes \left( g_1^{-1} \right)^{(n-1)} s_{n-2} \dotsb s_i \pi_j^{-1}\left(g_{2}^{-1}\right)^{(j)} y
        \in (n)_{n-2}(n).
    \]
    Hence
    \[
        \Theta(f_2 \circ f_1)(\alpha)
        = x g_1^{(i)} \pi_i g_2^{(\pi_i^{-1}(j))} \pi_{j-1} \otimes \left( g_1^{-1} \right)^{(n)} \pi_i^{-1} \pi_j^{-1}\left(g_{2}^{-1}\right)^{(j)} y
        \in (n)_{n-2}(n),
    \]
    and so
    \begin{align*}
        \Theta(f)(\alpha)
        &= x g_1^{(i)} \pi_i g_2^{(\pi_i^{-1}(j))} \pi_{j-1} \otimes s_{n-1} \otimes \left( g_1^{-1} \right)^{(n)} \pi_i^{-1} \pi_j^{-1}\left(g_{2}^{-1}\right)^{(j)} y
        \\
        &= x g_2^{(j)} \pi_j g_1^{(i)} s_i \dotsm s_{n-2} \otimes s_{n-1} \otimes s_{n-2} \dotsm s_{j-1} \pi_i^{-1} \left( g_2^{-1} \right)^{(j)} \left( g_1^{-1} \right)^{(i)} y
        \\
        &= x g_2^{(j)} \pi_j \otimes g_1^{(i)} \pi_i s_{n-2} \dotsm s_{j-1} \left( g_2^{-1} \right)^{(j-1)} \otimes \pi_i^{-1}\left( g_1^{-1} \right)^{(i)} y
        \\
        &= x g_2^{(j)} \pi_j \otimes \pi_j^{-1} g_1^{(i)} \pi_i \left( g_2^{-1} \right)^{(j-1)} \otimes \pi_i^{-1} \left( g_1^{-1} \right)^{(i)} y
        \\
        &= x g_2^{(j)} \pi_j \otimes \pi_j^{-1} \left( g_2^{-1} \right)^{(j)} g_1^{(i)} \pi_i \otimes \pi_i^{-1} \left( g_1^{-1} \right)^{(i)} y.
    \end{align*}
    The case $i > j$ is similar,
    \details{
        Suppose $i > j$.  Then we have
        \[
            \pi_i^{-1} \pi_j
            = s_{n-1} \dotsm s_i s_j \dotsm s_{n-1}
            = s_j \dotsm s_{n-2} s_{n-1} s_{n-2} \dotsm s_{i-1}.
        \]
        So,
        \[
            \Theta(f_1)(\alpha)
            = x g_1^{(i)} \pi_i g_2^{(j)} s_{j} \dotsm s_{n-2} \otimes \left( g_1^{-1} \right)^{(n-1)} s_{n-2} \dotsb s_{i-1} \pi_j^{-1}\left(g_{2}^{-1}\right)^{(j)} y
            \in (n)_{n-2}(n).
        \]
        Hence
        \[
            \Theta(f_2 \circ f_1)(\alpha)
            = x g_1^{(i)} \pi_i g_2^{(j)} \pi_{j} \otimes \left( g_1^{-1} \right)^{(n)} \pi_{i-1}^{-1} \pi_j^{-1}\left(g_{2}^{-1}\right)^{(j)} y
            \in (n)_{n-2}(n),
        \]
        and thus
        \begin{align*}
            \Theta(f)(\alpha)
            &= x g_1^{(i)} \pi_i g_2^{(j)} \pi_{j} \otimes s_{n-1} \otimes \left( g_1^{-1} \right)^{(n)} \pi_{i-1}^{-1} \pi_j^{-1}\left(g_{2}^{-1}\right)^{(j)} y
            \\
            &= x g_1^{(i)} g_2^{(j)} \pi_i \pi_{j} \otimes s_{n-1} \otimes \pi_{i-1}^{-1} \pi_j^{-1}\left( g_1^{-1} \right)^{(i)} \left(g_{2}^{-1}\right)^{(j)} y
            \\
            &= x g_2^{(j)} g_1^{(i)} \pi_j s_{i-1} \dotsm s_{n-2}\otimes s_{n-1} \otimes s_{n-2}\dotsm s_{j} \pi_i^{-1} \left(g_{2}^{-1}\right)^{(j)} \left( g_1^{-1} \right)^{(i)} y
            \\
            &= x g_2^{(j)} \pi_j \otimes g_1^{(i-1)} s_{i-1}\dotsm s_{n-2} \pi_j^{-1}  \left(g_{2}^{-1}\right)^{(j)} \otimes \pi_i^{-1}\left( g_1^{-1} \right)^{(i)} y
            \\
            &= x g_2^{(j)} \pi_j \otimes g_1^{(i-1)} \pi_{j}^{-1} \pi_i \left(g_{2}^{-1}\right)^{(j)} \otimes \pi_i^{-1}\left( g_1^{-1} \right)^{(i)} y
            \\
            &= x g_2^{(j)} \pi_j \otimes \pi_{j}^{-1} \left(g_{2}^{-1}\right)^{(j)} g_1^{(i)} \pi_i \otimes \pi_i^{-1}\left( g_1^{-1} \right)^{(i)} y.
        \end{align*}
    }
    giving
    \[
        \Theta(f) (\alpha)
        =
        \begin{cases}
            0 & \text{if } i=j, \\
            x g_2^{(j)} \pi_j \otimes \pi_j^{-1} \left( g_2^{-1} \right)^{(j)} g_1^{(i)} \pi_i \otimes \pi_i^{-1} \left( g_1^{-1} \right)^{(i)} y & \text{if } i \ne j.
        \end{cases}
    \]
    We also compute that
    \[
        \Theta(f')(\alpha)
        =
        \begin{cases}
            x g_2^{(j)} \pi_j \otimes \pi_j^{-1} \left( g_2^{-1} \right)^{(j)} g_1^{(i)} \pi_i \otimes \pi_i^{-1} \left( g_1^{-1} \right)^{(i)} y & \text{if } i = j, \\
            0 & \text{if } i \ne j.
        \end{cases}
    \]
    Thus, for all $i,j \in \{1,\dotsc,n\}$, we have
    \[
        \Theta(f+f')(\alpha)
        = x g_2^{(j)} \pi_j \otimes \pi_j^{-1} \left( g_2^{-1} \right)^{(j)} g_1^{(i)} \pi_i \otimes \pi_i^{-1} \left( g_1^{-1} \right)^{(i)} y.
    \]
    Therefore we have that
    \begin{multline*}
        \beta_k^{-1} \circ \left( \Omega_n \circ \Psi \left( 1_{\go^{\otimes (k-j-1)}} \otimes \crossing \otimes 1_{\go^{\otimes (j-1)}} \right) \right) \circ \beta_k
        \\
        = \beta_k^{-1} \circ \left( \Omega_n \left( 1_{\uparrow \downarrow}^{\otimes (k-j-1)} \otimes (f+f') \otimes 1_{\uparrow \downarrow}^{\otimes (j-1)} \right) \right) \circ \beta_k
    \end{multline*}
    is the map
    \[
        g_k e_{i_k} \otimes \dotsb \otimes g_1 e_{i_1}
        \mapsto g_k e_{i_k} \otimes \dotsb \otimes g_{j+2} e_{i_{j+2}} \otimes g_j e_{i_j} \otimes g_{j+1} e_{i_{j+1}} \otimes g_{j-1} e_{i_{j-1}} \otimes \dotsb g_1 e_{i_1},
    \]
    which is precisely the map $\Phi_n \left( 1_{\go^{\otimes (k-j-1)}} \otimes \crossing \otimes 1_{\go^{\otimes (j-1)}} \right)$.
\end{proof}

\section{Interpolating categories\label{sec:Knop}}

We assume throughout this section that $G$ is a finite group.  In \cite{Kno07}, Knop generalized the work \cite{Del07} of Deligne by embedding a regular category $\cA$ into a family of pseudo-abelian tensor categories $\cT(\cA,\delta)$, which are the additive Karoubi envelope of categories $\cT^0(\cA,\delta)$ depending on a degree function $\delta$.   Deligne's original construction corresponds to the case where $\cA$ is the category of finite boolean algebras.

As we now explain, the group partition category $\Par(G,d)$ is equivalent to $\cT^0(\cA,\delta)$, where $\cA$ is the category of finite boolean algebras with a locally free $G$-action and $\delta$ is a degree function depending on $d$.  In this way, $\Par(G,d)$ can be viewed as a concrete realization (including explicit bases of morphisms spaces) of the category $\cT^0(\cA,\delta)$, whose definition is rather abstract.  Moreover, \cref{twocats} can be viewed as giving an efficient presentation of Knop's category.  On the other hand, the equivalence of $\Par(G,d)$ and $\cT^0(\cA,\delta)$ allows us to deduce from Knop's work several important properties of $\Par(G,d)$.

For an arbitrary finite set $X$, let $\Power(X)$ denote the power set of $X$.  For $Y \subseteq X$, let $\neg Y = X - Y$ denote its complement.  The $5$-tuple $(\Power(X), \cap, \cup, \neg, \varnothing, X)$ is an example of a \emph{finite boolean algebra}.  In what follows, we simply denote this boolean algebra by $\Power(X)$.  In turns out that \emph{every} finite boolean algebra is isomorphic to one of this form.  In fact, the category $\FBA$ of finite boolean algebras is equivalent to the opposite of the category $\FinSet$ of finite sets.  To a map $f \colon X \to Y$ of finite sets, the corresponding homomorphism of boolean algebras is the map $\Power(Y) \to \Power(X)$, $Z \mapsto f^{-1}(Z)$.  We refer the reader to \cite[Ch.~15]{GH09} for details.

By definition, an action of a group $G$ on the boolean algebra $\Power(X)$ is a group homomorphism from $G$ to the automorphism group of $\Power(X)$ in $\FBA$.  It follows from the axioms of a boolean algebra that this action is uniquely determined by the action of $G$ on singletons or, equivalently, by a $G$-action on the set $X$.  In this way, the category $\FBA(G)$ of finite boolean algebras with $G$-actions (with morphisms being homomorphisms of boolean algebras that intertwine the $G$-actions) is equivalent to the opposite of the category of finite $G$-sets.

We say that a $G$-action on a boolean algebra is \emph{locally free} if every element of the boolean algebra is a union of elements on which $G$ acts freely.  In the case of the finite boolean algebra $\Power(X)$, this is equivalent to the condition that $G$ acts freely on the singletons.  (Note that, since $X$ is finite, this forces the group $G$ to be finite.)  Hence the category $\FBAlf$ of finite boolean algebras with locally free $G$-actions is equivalent to the opposite of the category of finite sets with free $G$-action:
\begin{equation} \label{bowl}
    \FBAlf \simeq \FinSetf^\op.
\end{equation}
The category $\FBAlf$ is regular, exact, and Malcev, using the definitions of these concepts given in \cite{Kno07}.

Knop's definition of the category $\cT^0(\cA,\delta)$ involves the diagram \cite[(3.2)]{Kno07}:
\begin{equation} \label{corn}
    \begin{tikzcd}
        & & r \times_y s \arrow[dl] \arrow[d,dashed] \arrow[dr] & & \\
        & r \arrow[dl] \arrow[dr] & s \circ r \arrow[dll,dashed] \arrow[drr,dashed] & s \arrow[dl] \arrow[dr] & \\
        x & & y & & z,
    \end{tikzcd}
\end{equation}
where $x,y,z \in \cA$, $r$ is a subobject of $x \times y$, $s$ is a subobject of $y \times z$, and $s \circ r$ is the image of the natural surjective map $r \times_y s \to x \times z$.  To relate Knop's construction to the $G$-partition category, we consider the diagram \cref{corn} in the case where $\cA = \FBAlf$.

First note that the product $\Power(X) \times \Power(Y)$ is isomorphic to $\Power(X \sqcup Y)$.  Our next goal is to interpret the subobjects $r,s$ in \cref{corn} as $G$-partition diagrams.

Every finite free $G$-set is isomorphic to one of the form $X \times G$, where $X$ is a finite set (indexing the $G$-orbits), with $G$ action given by
\[
    g \cdot (x,h) = (x,hg^{-1}),\quad x \in X,\ g,h \in G.
\]
Thus, by \cref{bowl}, every element of $\FBAlf$ is isomorphic to one of the form $\Power(X \times G)$.  Moreover, every element is, in fact, isomorphic to $\Power(\{1,2,\dotsc,r\} \times G)$ for some $r \in \N$.  (We adopt the convention that $\{1,2,\dotsc,r\} = \varnothing$ when $r=0$.)

Define the natural projection map $p_X \colon X \times G \to X$.  For a morphism $\varphi \colon \Power(\{1,2,\dotsc,r\} \times G) \to \Power(X \times G)$, define
\begin{align*}
    P^\varphi_i &:= p_X \circ \varphi(\{(i,1_G)\}) \subseteq X,\quad 1 \le i \le r, \\
    \lambda^\varphi &:= \bigcup_{i=1}^r \varphi(\{(i,1_G)\}) \in G^X,
\end{align*}
and set $\vec{P}^\varphi = (P^\varphi_1,\dotsc,P^\varphi_r)$.  Here we use the formal definition of an element of $G^X$, the set of functions $X \to G$, as a subset of $X \times G$.  Let
\[ \textstyle
    \Partition_r(X) := \{(P_1,\dotsc,P_r) \in \Power(X)^r : \bigcup_{i=1}^r P_i = X,\ P_i \ne \varnothing,\ P_i \cap P_j = \varnothing \text{ for all } 1 \le i,j \le r\}.
\]
In other words $\Partition_r(X)$ is the set of all $r$-tuples of nonempty disjoint sets whose union is $X$.  For $x,y \in \FBAlf$, let $\Mon(x,y)$ denote the set of monomorphisms $x \to y$ in $\FBAlf$.

\begin{lem}
    The map
    \begin{equation} \label{mango}
        \Mon ( \Power(\{1,2,\dotsc,r\} \times G), \Power(X \times G) )
        \to \Partition_r(X) \times G^X,\quad
        \varphi \mapsto (\vec{P}^\varphi,\lambda^\varphi),
    \end{equation}
    is a bijection.
\end{lem}

\begin{proof}
    For $(\vec{P},\lambda) \in \Partition_r(X) \times G^X$, define $\varphi'_{\vec{P},\lambda} \colon \{1,2,\dotsc,r\} \times G \to \Power(X \times G)$ by
    \[
        \varphi'_{\vec{P},\lambda} (i,g) = \{(x,\lambda(x)g) : x \in P_i\} \subseteq X \times G.
    \]
    This induces a map $\varphi_{\vec{P},\lambda} \colon \Power(\{1,2,\dotsc,r\} \times G) \to \Power(X \times G)$.  It is straightforward to verify that the map $(\vec{P},\lambda) \mapsto \varphi_{\vec{P},\lambda}$ is inverse to \cref{mango}.
\end{proof}

The proof of the following lemma is straightforward.

\begin{lem} \label{vinyl}
    The automorphism group of $\Power(\{1,2,\dotsc,r\} \times G)$ in $\FBAlf$ is the wreath product $G_r = G^r \rtimes \fS_r$, where the action is determined by its action on elements of $\{1,2,\dotsc,r\} \times G$ as follows:
    \begin{equation} \label{bump}
        (\bg, \pi) \cdot (i,h) := ( \pi(i), g_{\pi(i)} h ),\quad
        \pi \in \fS_r,\ i = \{1,2,\dotsc,r\},\ h \in G,\ \bg \in G^r.
    \end{equation}
\end{lem}

We also have an action of $G_r$ on $\Power_r(X) \times G^X$ given by
\begin{equation} \label{thump}
    (\bg, \pi) \cdot \big( (P_1,\dotsc,P_r), (h_x)_{x \in X} \big)
    := \big( (P_{\pi^{-1}(1)}, \dotsc, P_{\pi^{-1}(r)}), (g_{\pi(i_x)} h_x)_{x \in X} \big),
\end{equation}
where $i_x \in \{1,2,\dotsc,r\}$ is determined by $x \in P_{i_x}$.  The following lemma is also a straightforward verification.

\begin{lem}
    The bijection \cref{mango} intertwines the actions \cref{bump,thump}.
\end{lem}

Let $\Partition(X)$ denote the set of partitions of $X$.  For $\vec{P} = (P_1,\dotsc,P_r) \in \Partition_r(X)$, let $P = \{P_1,\dotsc,P_r\} \in \Partition(X)$ denote the corresponding partition of $X$.  For $P = \{P_1,\dotsc,P_r\} \in \Partition(X)$, define an equivalence relation $\sim_P$ on $G^X$ as follows: $(g_x)_{x \in X} \sim_P (h_x)_{x \in X}$ if and only if there exist $t_1,\dotsc,t_r \in G$ such that $g_x = t_i h_x$ for all $x \in P_i$.

\begin{cor}
    The subobjects of $\Power(X \times G)$ in the category $\FBAlf$ are naturally enumerated by the set
    \begin{equation} \label{croissant}
        \bigsqcup_{P \in \Partition(X)} G^X/\sim_P.
    \end{equation}
\end{cor}

Now consider the diagram \cref{corn} with $x = \Power(\{1,2,\dotsc,k\} \times G)$, $y = \Power(\{1,2,\dotsc,l\} \times G)$, and $z = \Power(\{1,2,\dotsc,m\} \times G)$.  When $X = \PS_k^l \cong \{1,2,\dotsc,k\} \sqcup \{1,2,\dotsc,l\}$, the set \cref{croissant} can be naturally identified with the equivalence classes of $G$-partitions of type $\binom{l}{k}$.  Thus we can view $r$ and $s$ as equivalence classes of $G$-partitions $[P,\bg]$ and $[Q,\bh]$, respectively.  We then leave it to the reader to verify that $r \times_y s$ exists if and only if the pair $([Q,\bg], [P,\bh])$ is compatible.  If this pair is compatible, then $r \times_y s$ is the equivalence class of $\stack((Q,\bh),(P,\bg))$ and $s \circ r = [Q \star P, \bh \star_{Q,P} \bg]$.  Thus we have the following result.

\begin{theo} \label{Knop}
    The $G$-partition category $\Par(G,d)$ is equivalent to the category $\cT^0(\cA,\delta)$ defined in \cite[Def.~3.2]{Kno07}, where $\cA = \FBAlf$ and $\delta$ is the degree function of \cite[(8.15)]{Kno07} with the $t$ there equal to $d$.
\end{theo}

\begin{proof}
    This follows from the above discussion and \cite[Example~2, p.~596]{Kno07}.
\end{proof}

Let $\Kar(\Par(G,d))$ be the additive Karoubi envelope (also known as the pseudo-abelian completion) of $\Par(G,d)$.  Let $\cN(G,d)$ be the tensor radical (also known as the tensor ideal of negligible morphisms) of $\Kar(\Par(G,d))$.

\begin{cor} \label{fire}
    Suppose $\kk$ is a field of characteristic zero.
    \begin{enumerate}
        \item The category $\Kar(\Par(G,d))/\cN(G,d)$ is a semisimple (hence abelian) category.

        \item We have $\cN(G,d) = 0$ if and only if $d \notin \N |G|$.

        \item If $\cN(G,d) = 0$, then the simple objects of $\Kar(\Par(G,d))$ are naturally parameterized by the set of $N$-tuples of Young diagrams, where $N$ is the number of isomorphism classes of simple $G$-modules.

        \item If $d = n |G|$, then $\Kar(\Par(G,d))/\cN(G,d)$ is equivalent to the category of $\kk G_n$-modules.
    \end{enumerate}
\end{cor}

\begin{proof}
    \begin{enumerate}[wide]
        \item This follows from \cite[Th.~6.1(i)]{Kno07}.

        \item This follows from \cite[Example~2, p.~596]{Kno07}.

        \item By parts (iii) and (iv) of \cite[Th.~6.1]{Kno07}, the simple objects of $\Kar(\Par(G,d))$ are in bijection with the simple modules of the automorphism groups of objects of $\FBAlf$ which, by \cref{vinyl}, are precisely the wreath products $G_n$, $n \in \N$.  The statement then follows from the classification of irreducible modules of wreath product groups.  (See, for example, \cite[Prop.~4.3]{RS17}.)

        \item This is explained in \cite[Example~2, p.~606]{Kno07}. \qedhere
    \end{enumerate}
\end{proof}

\begin{rem} \label{Mori}
    Deligne's construction has also been generalized by Mori \cite{Mor12}, who defined, for each $d \in \kk$, a 2-functor $\cS_d$ sending a tensor category $\cC$ to another tensor category $\cS_d(\cC)$, which should be thought of as a sort of interpolating wreath product functor.  When $\cC$ is the category $G\md$ of $G$-modules, $\cS_t(G\md)$ can also be thought of as a family of interpolating categories for modules of the wreath products $G_n$, $n \in \N$.  Mori's interpolating category contains Knop's as a full subcategory; see \cite[Rem.~4.14]{Mor12}.  Mori gives a presentation of his categories, the relations of which can be found in \cite[Prop.~4.26]{Mor12}.  The presentation of \cref{GPC} is considerably more efficient.  For example, $\Par(G)$ has just \emph{one} generating object, whereas Mori's category (before taking the additive Karoubi envelope) has a generator for each representation of $G$.  In addition, the presentation of \cite{Mor12} includes as generating morphisms \emph{all} morphisms in the category $G\md$, whereas the presentation of \cref{GPC} only includes a morphism for each element of the group (the tokens).
\end{rem}


\bibliographystyle{alphaurl}
\bibliography{GroupPartition}

\end{document}